\documentclass[preprint,12pt,3p]{elsarticle}
\usepackage{natbib}
\usepackage{times}
\usepackage{amsthm}
\usepackage{amsmath}
\usepackage{amssymb,bm}
\usepackage{amsfonts}
\usepackage{bbold}
\usepackage{latexsym}
\usepackage{graphics}
\usepackage{epsfig}
\usepackage{epstopdf}

\usepackage{enumitem}
\usepackage{diagbox} 
\usepackage{bbm}
\usepackage{bm} 

\usepackage{graphicx}
\usepackage{float} 
\usepackage{subfigure}
\theoremstyle{definition}
\newtheorem{theorem}{Theorem}[section]

\newtheorem{remark}{Remark}[section]

\newtheorem{corollary}{Corollary}[section]

\newtheorem{lemma}{Lemma}[section]

\journal{Journal Name}

\makeatletter
\def\ps@pprintTitle{%
	\let\@oddhead\@empty
	\let\@evenhead\@empty
	\def\@oddfoot{\centerline{\thepage}}%
	\let\@evenfoot\@oddfoot}
\makeatother

\begin{document}
\begin{frontmatter}

\title{Nonlocal Fully Nonlinear Parabolic Differential Equations Arising in Time-Inconsistent Problems\footnote{Chi Seng Pun gratefully acknowledges Ministry of Education (MOE), AcRF Tier 2 grant (Reference No: MOE2017-T2-1-044) for the funding of this research.}}

\author[label1]{Qian Lei}
\ead{leiq0002@e.ntu.edu.sg}

\author[label1]{Chi Seng Pun\corref{cor1}}
\address[label1]{School of Physical and Mathematical Sciences, Nanyang Technological University, Singapore}
\ead{cspun@ntu.edu.sg}
\cortext[cor1]{Corresponding author}

\begin{abstract}
We prove the well-posedness results, i.e. existence, uniqueness, and stability, of the solutions to a class of nonlocal fully nonlinear parabolic partial differential equations (PDEs), where there is an external time parameter $t$ on top of the temporal and spatial variables $(s,y)$ and thus the problem could be considered as a flow of equations. The nonlocality comes from the dependence on the unknown function and its first- and second-order derivatives evaluated at not only the local point $(t,s,y)$ but also at the diagonal line of the time domain $(s,s,y)$. Such equations arise from time-inconsistent problems in game theory or behavioural economics, where the observations and preferences are (reference-)time-dependent. To address the open problem of the well-posedness of the corresponding nonlocal PDEs (or the time-inconsistent problems), we first study the linearized version of the nonlocal PDEs with an innovative construction of appropriate norms and Banach spaces and contraction mappings over which. With fixed-point arguments, we obtain the well-posedness of nonlocal linear PDEs and establish a Schauder-type prior estimate for the solutions. Then, by the linearization method, we analogously establish the well-posedness under the fully nonlinear case. Moreover, we reveal that the solution of a nonlocal fully nonlinear parabolic PDE is an adapted solution to a flow of second-order forward-backward stochastic differential equations.
\end{abstract}

\end{frontmatter}


\section{Introduction} \label{sec:intro}
We study the existence and uniqueness problems for a class of nonlocal fully nonlinear parabolic partial differential equations (PDEs) of the form 
\begin{equation} \label{Nonlocal fully nonlinear equation}
	\left\{
	\begin{array}{rcl}
		u_s(t,s,y) & = & F\big(t,s,y,u(t,s,y),u_y(t,s,y),u_{yy}(t,s,y), \\
		& & \qquad \qquad u(s,s,y),u_y(s,s,y),u_{yy}(s,s,y)\big), \\
		u(t,0,y) & = & g(t,y),\quad 0\leq s\leq t\leq T,\quad y\in\mathbb{R}^d.
	\end{array}
	\right.
\end{equation} 
where the mapping $F$ could be nonlinear with respect to all its arguments, $s$ and $y$ are temporal and spatial variables, respectively, while $t$ could be considered as an external temporal parameter, and the temporal variables $(s,t)$ are defined in a triangular region $\Delta[0,T]:=\left\{(t,s)\in[0,T]^2:~0\leq s\leq t\leq T\right\}$. The nonlocality of \eqref{Nonlocal fully nonlinear equation} stems from the ``diagonal dependence," where $u$, $u_y$, and $u_{yy}$ in $F$ are evaluated not only at $(t,s,y)$ but also at $(s,s,y)$. The PDEs \eqref{Nonlocal fully nonlinear equation} originate from various problems in theories of stochastic control and stochastic differential equation with their applications in financial economics when behavioural factors are considered. \\

\paragraph{\textbf{Organization of this paper}} Subsection \ref{sec:introTIC} briefly introduces the nonlocal PDEs as the driving force of a Nobel-prize winning theory, namely prospect theory, while Subsection \ref{sec:limitation} reviews the literature on this research direction, followed by a succinct account of our insights and main contributions in Subsections \ref{sec:insights} and \ref{sec:contributions}. Our study begins with a linearized case (linear $F$) in Section \ref{sec:linear}, where we first introduce appropriate norms and Banach spaces for our study in Subsection \ref{sec:normspace} and then the well-posedness results in Subsection \ref{sec:linearresult}, which acquires a Schauder-type prior estimates revealing the relationship among solutions of the nonlocal equations, the nonhomogeneous terms, and the initial conditions. Under mild conditions, Section \ref{sec:nonlinear} is devoted to prove the existence and uniqueness of nonlocal fully nonlinear parabolic PDEs by fixed-point arguments. Section \ref{sec:stochrep} presents the close connection between nonlocal parabolic PDEs and forward-backward stochastic differential equations (FBSDEs). Moreover Subsection \ref{sec:2FBVIEs} shows that the solution to a nonlocal fully nolinear parabolic PDE solves a corresponding flow of second-order FBSDEs, which implies a new Feynman--Kac formula. Finally, Section \ref{sec:conclusion} concludes. 


\subsection{Nonlocality, Time-Inconsistency, and the Lack of Flow Property} \label{sec:introTIC}
Pontryagin's maximum principle and Bellman's principle of optimality (or dynamic programming, DP) are two main pillars of solving a stochastic control problem; see \cite{Yong1999}. However, these two principles could be violated in a common situation in economics when a decision-maker's preferences change over time, i.e. dynamic inconsistency (in the context of game theory) or time inconsistency (in the context of behavioral economics). Some well-known examples are the financial mean-variance portfolio selection \cite{Markowitz1952}, hyperbolic discounting in behavioural economics \cite{Frederick2002}, and endogenous habit formulation. As evidenced in the prospect theory (\cite{Kahneman1979}) and \cite{Thaler1981}, the psychological principles of reference dependence and state-dependent objectives should be taken into account during the decision making, but they cause time-inconsistent (TIC) issues. In what follows, we briefly introduce the relevance of solving \eqref{Nonlocal fully nonlinear equation} and its connection with stochastic differential equation (SDE) theory.

Let $\left(\Omega,\mathcal{F},\mathbb{F},\mathbb{P}\right)$ be a completed filtered probability space on which a $k$-dimensional Brownian motion $W(\cdot)$ with the natural filtration $\mathbb{F}=\left\{\mathcal{F}_s\right\}_{s\geq 0}$ augmented by all the $\mathbb{P}$-null sets in $\mathcal{F}$ is well-defined. By choosing a suitable control process $\alpha(\cdot):[s,T]\times\Omega\to U$ with $U\subseteq\mathbb{R}^m$ being a non-empty set, we aim to optimize the following cost functional:
\begin{eqnarray}
	& & \inf_{\alpha}J\left(s,y;\alpha(\cdot)\right) \label{eq:value} \\
	\hbox{with } J\left(s,y;\alpha(\cdot)\right) & := & \mathbb{E}\left[\left.\int^T_s h\left(s,\tau,X(\tau),\alpha(\tau)\right)d\tau
	+g\left(s,X(T)\right)\right|\mathcal{F}_s\right] \label{Cost functional}
\end{eqnarray} 
and the controlled state process $\{X(\tau)\}_{\tau\in[s,T]}$ driven by the following SDE: 
\begin{equation} \label{State equation}
	\left\{
	\begin{array}{lr}
		dX(\tau)=b(\tau,X(\tau),\alpha(\tau))d\tau+\sigma(\tau,X(\tau),\alpha(\tau))dW(\tau), \quad \tau\in[s,T], \\
		X(s)=y, \quad y\in\mathbb{R}^d.  
	\end{array}
	\right.
\end{equation} 
For illustration, we assume a Markovian framework and that all the coefficient and objective functions, $b:[0,T]\times\mathbb{R}^d\times U\mapsto\mathbb{R}^d$, $\sigma:[0,T]\times\mathbb{R}^d\times U\mapsto\mathbb{R}^{d\times k}$, $h:\nabla[0,T]\times\mathbb{R}^d\times U\mapsto\mathbb{R}$ and $g:[0,T]\times\mathbb{R}^d\mapsto\mathbb{R}$ are deterministic, where $\nabla[0,T]:=\{(s,\tau)\in[0,T]^2:~0\le s\le \tau \le T\}$. A distinct feature of problem \eqref{eq:value} is its running and terminal costs, $h$ and $g$ in \eqref{Cost functional}, varying in the time variable $s$. Hyperbolic or delay discounting (see \cite{Thaler1981,Laibson1997}) is a typical example of \eqref{Cost functional}, which induces time (dynamic) inconsistency and incents the decision-makers deviating their pre-committed plan for some future period when it arrives. This is also a common feature of decision-making under the prospect theory \cite{Kahneman1979}. The time inconsistency (TIC) of problem \eqref{eq:value} also manifests itself through a violation of Bellman's principle of optimality and thus standard DP arguments are not directly applicable in this case. In the context of game theory, a pre-commitment plan to problem \eqref{eq:value} is subgame imperfect, where players are the incarnations of the agent parametrized by $s\in[0,T]$. In this context, finding a Nash subgame perfect equilibrium is essentially a consistent planning that is economically meaningful and appealing; see \cite{Strotz1955,Pollak1968}.

Recent studies have contributed to convert the problem \eqref{eq:value} to solving Hamilton--Jacobi--Bellman (HJB) equation(s), which is essentially a (nonlocal) nonlinear PDE. There are two derivations of the PDE based on a modified recursive equation of the value function $V$ in \eqref{eq:value} and a discretization method, respectively, with game-theoretic and DP arguments. \cite{Bjoerk2017} characterize the subgame perfect equilibrium via the solution $(v(s,y),u(t,s,y))$ of \textit{a system of HJB equations}: for any $(t,s)\in\nabla[0,T]$ and $y\in\mathbb{R}^d$, $v(s,y)=u(s,s,y)$ and
\begin{equation} \label{HJB system}  
	\left\{
	\begin{array}{rcl}
		v_s(s,y)+\inf\limits_{a\in U}\left\{\mathcal{H}\big(s,s,y,a,v_y(s,y),v_{yy}(s,y)\big)\right\} & = & 0, \\
		u_s(t,s,y)+\mathcal{H}\big(t,s,y,e(s,y),u_y(t,s,y),u_{yy}(t,s,y)\big) & = & 0,
	\end{array}
	\right. 
\end{equation}
with boundary conditions $v(T,y)=g(T,y)$ and $u(t,T,y)=g(t,y)$ for $t\in[0,T]$ and $y\in\mathbb{R}^d$, where the Hamiltonian $\mathcal{H}:\nabla[0,T]\times\mathbb{R}^d\times U\times\mathbb{R}^d\times\mathbb{S}^{d}\to\mathbb{R}$ with $\mathbb{S}^d\subseteq\mathbb{R}^{d\times d}$ being the set of all $d\times d$ symmetric matrices is defined as: 
\begin{equation} \label{Hamiltonian} 
	\mathcal{H}(t,s,y,a,p,q)=\frac{1}{2}\mathrm{tr}\left[q\cdot\left(\sigma\sigma^\top\right)(s,y,a)\right]+p^\top b(s,y,a)+h\left(t,s,y,a\right)
\end{equation} 
and the $e(s,y)$ in the second equation of \eqref{HJB system} realizes the supremum in the first equation (if it always exists) and constitutes an equilibrium control policy for $s\in[0,T]$. We first provide two important observations about the system \eqref{HJB system}: first, we need to determine simultaneously $u$ and $v$ (or equivalently $e$ that depends on $v$); second, from the first equation and the relation between $v$ and $u$, we know that the $e(s,y)$ depends on $s,y$, $v_y(s,y)=u_y(s,s,y)$, and $v_{yy}(s,y)=u_{yy}(s,s,y)$, i.e.
\begin{equation} \label{Closed form of Equilibrium controls} 
	e(s,y)=\phi\big(s,s,y,u_y(s,s,y),u_{yy}(s,s,y)\big) \quad \hbox{for }\phi(t,s,y,p,q)\in\inf_{a}\mathcal{H}(t,s,y,a,p,q)
\end{equation}
as we allow the control variate $\alpha(\cdot)$ entering both of the drift $b$ and volatility $\sigma$ of the controlled state process $X(\cdot)$. Hence, it suffices to solve for the second equation of \eqref{HJB system} with $e(s,y)$ of the form \eqref{Closed form of Equilibrium controls}: for $(t,s)\in\nabla[0,T]$ and $y\in\mathbb{R}^d$,
\begin{equation} \label{Equilibrium HJB equation}
	u_s(t,s,y)+\mathcal{H}\Big(t,s,y,\phi\big(s,s,y,u_y(s,s,y),u_{yy}(s,s,y)\big),u_y(t,s,y),u_{yy}(t,s,y)\Big)=0
\end{equation}
with boundary condition $u(t,T,y)=g(t,y)$ for $t\in[0,T]$ and $y\in\mathbb{R}^d$, while the verification theorem in \cite{Bjoerk2017} validates that the classical solution to \eqref{HJB system} or equivalently \eqref{Equilibrium HJB equation} is the subgame perfect equilibrium solution. By a discretization method and taking the continuous-time limit, \cite{Yong2012,Wei2017,Yan2019} also yield the same equation of \eqref{Equilibrium HJB equation} for the problem \eqref{eq:value} and the authors call it as an \textit{equilibrium HJB equation}. In contrast with the conventional HJB equation, which is a \textbf{local} fully nonlinear PDE, the equilibrium HJB equation \eqref{Equilibrium HJB equation} is a backward \textbf{nonlocal} fully nonlinear PDE as a special case of \eqref{Nonlocal fully nonlinear equation}. Heuristically speaking, the decision-makers/players in TIC stochastic control problems are supposed to follow the principle of ``globally thinking, locally acting." ``Thinking globally" induces the $u$ terms evaluated at $(t,s,y)$, while ``acting locally" implies them evaluated at $(s,s,y)$ in the equation \eqref{Equilibrium HJB equation}.

In addition to the originations and inspirations from TIC stochastic control problems, nonlocal PDEs are in close connection with the SDE theory, especially for a flow of FBSDEs or backward stochastic Volterra integral equation (BSVIE). As the well-known Feynman--Kac formula (see \cite{Kac1949}) indicates, an adapted solution $(X(\cdot),Y(\cdot),Z(\cdot))$ of a system of FBSDE gives a stochastic representation of the solution $u(s,y)$ of local parabolic PDEs. In the classical setting, the Feynman--Kac formula connects two sides of the adapted solution $Y(\cdot)$ and $Z(\cdot)$ of a BSDE and the solution $u(s,y)$ of the corresponding parabolic PDE in the fashion that $Y(\cdot)=u(\cdot,X(\cdot))$ and $Z(\cdot)=\sigma^\top(\cdot,X(\cdot))u_y(\cdot,X(\cdot))$, where $X(\cdot)$ solves a forward SDE. Depending on whether the FBSDE is coupled or not, the parabolic PDEs are either semi-linear or quasi-linear. By introducing the dependence of the current time $t$ into the generator and the terminal condition of BSDEs, a family of FBSDEs parameterized by $t$ forms a flow of FBSDEs, which is linked to the system of HJB equations. Moreover, a parallel extension of BSDE leads to a BSVIE, which is linked to the equilibrium HJB equation. In the context of SDE theory, BSVIE is \textbf{lack of flow property} that a normal BSDE possesses. Analogous to the FBSDE theory, we can also anticipate a Feynman--Kac-type representations of adapted solutions of FBSVIEs, which are characterized by the solution of a nonlocal parabolic PDE. As Section \ref{sec:stochrep} elaborates, all the nonlocal parabolic PDEs in the existing Feynman--Kac-type formulas of FBSVIEs are special cases of our general nonlocal fully nonlinear PDE \eqref{Nonlocal fully nonlinear equation}. It is noteworthy that a nonlocal fully nonlinear parabolic PDEs are corresponded to a flow of second-order FBSDEs (2BSDEs) or second-order FBSVIEs (2FBSVIEs), which creates new knowledge to the literature. A further exploration between nonlocal parabolic PDEs and a flow of FBSDEs (or 2FBSDEs) will be detailed in Section \ref{sec:stochrep}.


\subsection{Limitation of the Existing Studies} \label{sec:limitation}
In general, the existence and uniqueness problems of the equilibrium HJB equation \eqref{Equilibrium HJB equation} or the system of HJB equations \eqref{HJB system} still remain open, although there are several attempts on their existence and uniqueness in the past decade. To the best of our knowledge, the best result in this regard so far is obtained in \cite{Wei2017}, where the authors presented a existence and uniqueness result for the case that the nonlocal PDEs are restricted to the linear dependence on the second-order derivative at local point $u_{yy}(t,s,y)$ and the removal of the second-order derivative at diagonal $u_{yy}(s,s,y)$. The removal was justified by prohibiting the control from entering the diffusion of the state process $\eqref{State equation}$, i.e. $\sigma(s,y,a)=\sigma(s,y)$, which would lead to the $u_{yy}(s,s,y)$-independence of the equilibrium control policy $e(s,y)$ \eqref{Closed form of Equilibrium controls} and the equilibrium HJB equation \eqref{Equilibrium HJB equation}. The similar restriction is inherited to the subsequent works, e.g., \cite{Hernandez2020,Mei2020,Wang2021}. However, without controls in the diffusion of the state, we can hardly control the risk (noises) from the stochastic systems, which are crucial in many problems such as portfolio management, inventory control, etc. In fact, only when the controls could or would take effect on the magnitude of uncertainty, the stochastic problems differ from the deterministic ones. In the related works of \cite{Yong2012,Wei2017,Hamaguchi2020,Hernandez2020,Mei2020,Wang2020}, all the authors admit that the existence and uniqueness problems of nonlocal fully nonlinear parabolic PDEs with the term $u_{yy}(s,s,y)$ is a complicated open problem. Hence, if we can establish the well-posedness of nonlocal fully nonlinear parabolic PDEs \eqref{Nonlocal fully nonlinear equation}, the well-posedness of the subgame perfect equilibrium is resolved, so are some open problems listed in the discussion of \cite{Bjoerk2017}.

From the perspective of FBSDEs, hindered by the limitation of the nonlocal PDE theory, there is lack of unified and general treatment for the PDE-side in Feynman--Kac formulas of BSVIEs. In fact, as the Feynman--Kac formula indicates, advances of the well-posedness results of nonlocal PDEs allow us to solve directly the corresponding FBSVIE $(X(\cdot),Y(\cdot,\cdot),Z(\cdot,\cdot))$ by combining solutions of nonlocal PDEs $u(\cdot,\cdot,\cdot)$ and of the coupled forward SDE $X(\cdot)$. The well-posedness of nonlocal PDEs can be utilized to study the solvability of FBSVIEs in a similar manner of the classical four-step numerical scheme for studying the solvability of FBSDEs; see \cite{Ma1994}. Likewise, limited by the development of nonlocal PDEs, the FBSVIEs in the existing literature only correspond to some special cases of \eqref{Nonlocal fully nonlinear equation}, such as nonlocal \textit{semi-linear} or \textit{quasi-linear} PDEs as in \cite{Wei2017,Wang2019,Hamaguchi2020}. Therefore, the study of nonlocal \textit{fully nonlinear} PDEs \eqref{Nonlocal fully nonlinear equation} provides a solid foundation of studying a more general form of FBSVIEs.


\subsection{Insights into Analysing Nonlocal Fully Nonlinear PDEs} \label{sec:insights}
Given any suitable pair $(F,g)$ of nonlinearity and initial condition, the well-posedness of the nonlocal differential equations \eqref{Nonlocal fully nonlinear equation} requires
\begin{enumerate}
	\item a solution exists in some sense;
	\item the solution is unique in some space;
	\item the map from $(F,g)$ to solutions is continuous in some topology. 
\end{enumerate}

First of all, this paper adopts the concept of classical solutions with sufficient regularities, i.e. the smoothness of first- and second-order derivatives of $u$ with respect to $s$ and $y$. Within a classical solution framework, we can take advantage of methodologies for local parabolic PDEs in \cite{Friedman1964,Ladyzhanskaya1968,Lunardi1995}, such as various regularity results and prior estimates of solutions, as well as represent the equilibrium controls in a practical closed form \eqref{Closed form of Equilibrium controls}. Moreover, for such a new class of nonlocal PDEs \eqref{Nonlocal fully nonlinear equation}, both dependence on diagonal terms and the order relation between $t$ and $s$ distinguish themselves from the local PDEs. Hence, we need to identify suitable norms and Banach spaces for the nonlocal PDEs and even for all TIC problems for practicality. Under a well-defined norm and the induced Banach space, the non-linearity $F$ is required to be closed in some sense.

In this paper, we make use of fixed-point arguments to show the well-posedness of \eqref{Nonlocal fully nonlinear equation}. The key steps of which are to construct suitable mappings and to prove their contractions under the norms introduced. However, the presence of the second-order diagonal term $u_{yy}(s,s,y)$ causes an essential difficulty for constructing a desired contraction. As discussed in the previous section, this difficulty has been the major obstacle for many studies on TIC problems.

To understand the essence of our problem, we illustrate the difficulty with a simplified problem of nonlocal linear PDE without the low-order terms. For the existence and uniqueness problems of such an equation, it is natural to consider a mapping from $u$ to $w$ that satisfies
\begin{equation} \label{Mapping 1 from u to w} 
	w_s(t,s,y)=a(t,s,y)w_{yy}(t,s,y)+\overline{a}(t,s,y)u_{yy}(s,s,y)
\end{equation}
It is easy to see that the mapping is well-defined, thanks to the classical PDE theory. By replacing the unmanageable diagonal term with a known function $u$, the well-posedness of the PDEs \eqref{Mapping 1 from u to w} parameterized by $t$ promises the existence and uniqueness of the solution $w$. Moreover, it is also clear that the fixed point solves the original nonlocal linear PDE. However, since the input $u$ is of the same order of the output $w$ by simple analysis, it is not feasible to prove that they are contractive.

We overcome the major difficulties by developing a range of techniques and results, three essential ingredients of which are listed as follows:  
\begin{enumerate}[label=(\alph*)] 
	\item studying the well-posedness of a linearized version of \eqref{Nonlocal fully nonlinear equation}, i.e. \eqref{Nonlocal linear PDE with t}, where the mapping $F$ is linear with respect to all its arguments related to $u$. Giving up the kind of mappings directly from $u$ to $w$ like \eqref{Mapping 1 from u to w}, we propose a coupled nonlocal PDE system of $(u,u_t)$, which consists of the unknown function $u$ and its derivative $u_t$ with respective to $t$, and the system is mathematically equivalent to the linearized nonlocal PDE via the linkage between their solutions ($u$). In this way, the smoothness of $u$ in $t$ can help us achieve a desired contraction such that the system of $(u,u_t)$ admits a unique solution. Consequently, the well-posedness of original linearized nonlocal PDEs is established. The technical details are discussed in Theorem \ref{Well-posedness of (u,u_t)};
	\item establishing a Schauder-type prior estimate for solutions of linear nonlocal PDEs studied in (a). In general, a successful treatment of nonlinear PDEs depends on a prior estimate. Based on the norms defined by us for nonlocal PDEs (and TIC problems), the prior estimate not only controls the behavior of the solutions and provides quantitative information on the regularity of solutions, but also gives a certain compactness to the class of possible solutions. Such a compactness is necessary for Picard's iterative method and Banach's fixed-point theorem in the study of nonlocal fully nonlinear PDEs \eqref{Nonlocal fully nonlinear equation}. The Schauder prior estimate is presented in Corollary \ref{Schauder estimates};  
	\item linearizing nonlocal fully nonlinear PDEs into nonlocal linear equations. The established results in (a) and (b) are utilized to show the well-posedness of nonlocal fully nonlinear PDEs of the form \eqref{Nonlocal fully nonlinear equation} with the linearization. We first construct a mapping from $u$ to $w$  $$w_s=\mathcal{L}w+F(t,s,y,u,u_y,u_{yy},u|_{t=s},u_y|_{t=s},u_{yy}|_{t=s})-\mathcal{L}u$$ 
	where $\mathcal{L}$ is a nonlocal linear differential operator defined in \eqref{Nonlocal linear differential operator L}. Given $u$ in a suitable space, the solvability of nonlocal linear PDEs proven in (a) ensures the operator well-defined. Furthermore, the prior estimate obtained in (b) provides a bound of the solution $w$ with the nonhomogeneous term $F-\mathcal{L}u$ and a given initial condition $g$. Subsequently, we can prove that the mapping is a contraction in a suitable Banach space. Consequently, the unique fixed point solves the nonlocal fully nonlinear PDEs \eqref{Nonlocal fully nonlinear equation}.  
\end{enumerate}

Finally, the stability of solutions with respect to the data $(F,g)$ is also important for theoretical analysis and practical applications. The map from data to solutions is continuous in the topology induced by our norms. In summary, the essential difficulties of studying the nonlocal fully nonlinear PDEs (or TIC problems) is not only to establish suitable Banach spaces and well-defined mappings over them, but also to ensure
that the mappings are contractive over the norm-induced spaces.


\subsection{Main Contributions} \label{sec:contributions}
This paper aims to address some open problems in the theories of PDE, SDE, and stochastic controls for the cases where there is nonlocality (for PDEs), lack of flow property (for FBSDEs), or time-inconsistency (for optimal stochastic controls). By noting the connections among these three fields, the aforementioned three scenarios are haunted by the similar mathematical difficulty. Hence, our study of nonlocality in PDEs is beneficial to understanding time-inconsistency in FBSDE and control problems.

The main contributions of the this paper is threefold. First, we propose suitable norms and Banach spaces for studying nonlocal PDEs and TIC problems. Under them, the well-posedness results of nonlocal linear PDEs are first proven. After establishing a Schauder-type prior estimate, the method of linearization is adopted to show the existence and uniqueness of nonlocal fully nonlinear PDEs. Second, thanks to these regularity results of nonlocal PDEs \eqref{Nonlocal fully nonlinear equation}, the well-posedness of solutions of the equilibrium HJB equation \eqref{Equilibrium HJB equation} and the system of HJB equations \eqref{HJB system}, both of which identify the equilibrium policy and the equilibrium value functions in TIC problems, can be also established even when the control $\alpha(\cdot)$ enters the diffusion term $\sigma$ of state processes (the bottleneck of the existing studies). Third, a unified and general treatment of PDE representations of adapted solutions of a flow of second-order FBSDEs is initiated, which enhances the feasibility of studying FBSDEs or FBSVIEs from the modelling perspective of PDEs.


\section{Nonlocal Linear Parabolic PDEs} \label{sec:linear}
In this section, we first propose suitable norms and Banach spaces for nonlocal PDEs. Based on them, we will prove the solvability of nonlocal linear PDEs and estimate their solutions. 


\subsection{Preliminaries: Norms and Banach Spaces} \label{sec:normspace}
For the conventional local parabolic PDEs, the corresponding differential operator works well over H\"{o}lder spaces of sufficiently smooth functions. It is expected that some revision of the H\"{o}lder spaces are desired and effective for our nonlocal parabolic equations. Hence, we review some definitions of spaces of H\"{o}lder continuous functions here. 

Given $0\leq a\leq b\leq T$, we denote by $C([a,b]\times\mathbb{R}^d)$ the set of all the continuous and bounded functions in $[a,b]\times\mathbb{R}^d$ endowed with the sup norm $|\cdot|^{\infty}_{[a,b]\times\mathbb{R}^d}$. Wherever no confusion arises, we write $|\cdot|^\infty$ instead of $|\cdot|^{\infty}_{[a,b]\times\mathbb{R}^d}$. Next, for $\alpha\in(0,1)$, we introduce H\"{o}lder spaces:
\begin{equation*}
	\begin{split}
		C^{\alpha,0}([a,b]\times\mathbb{R}^d)=\Bigg\{\varphi\in C([a,b]\times\mathbb{R}^d):\langle\varphi\rangle^{(\alpha)}_s:=\sup\limits_{\begin{subarray}{c} a\leq s< s^\prime\leq b \\ y\in\mathbb{R}^d\end{subarray}}\frac{|\varphi(s,y)-\varphi(s^\prime,y)|}{|s-s^\prime|^\alpha}<\infty, \\
		|\varphi|_{C^{\alpha,0}([a,b]\times\mathbb{R}^d)}:=|\varphi|^\infty+\langle\varphi\rangle^{(\alpha)}_s<\infty\Bigg\}
	\end{split}
\end{equation*}
and similarly, 
\begin{equation*}
	\begin{split}
		C^{0,\alpha}([a,b]\times\mathbb{R}^d)=\Bigg\{\varphi\in C([a,b]\times\mathbb{R}^d):\langle\varphi\rangle^{(\alpha)}_y:=\sup\limits_{\begin{subarray}{c} a\leq s\leq b \\ y,y^\prime\in\mathbb{R}^d\end{subarray}}\frac{|\varphi(s,y)-\varphi(s,y^\prime)|}{|y-y^\prime|^\alpha}<\infty, \\
		|\varphi|_{C^{0,\alpha}([a,b]\times\mathbb{R}^d)}:=|\varphi|^\infty+\langle\varphi\rangle^{(\alpha)}_y<\infty\Bigg\}.
	\end{split}
\end{equation*}
Moreover, the classical set $C^{1,2}([a,b]\times\mathbb{R}^d)$ is defined as 
\begin{equation*}
	\begin{split}
		C^{1,2}([a,b]\times\mathbb{R}^d)=\Big\{\varphi\in C([a,b]\times\mathbb{R}^d):\exists~\varphi_s,~\varphi_{y_iy_j}\in C([a,b]\times\mathbb{R}^d),~i,j=1,\ldots,d\Big\}. 
	\end{split}
\end{equation*}
which is endowed with 
\begin{equation*}
	|\varphi|_{C^{1,2}([a,b]\times\mathbb{R}^d)}=|\varphi|^\infty+\sum^d_{i=1}|\varphi_{y_i}|^\infty+|\varphi_s|^\infty+\sum^d_{i,j=1}|\varphi_{y_iy_j}|^\infty. 
\end{equation*}

Now, we define ``parabolic" H\"{o}lder spaces, which are very common in the study of local parabolic equations. We set 
\begin{equation*}
	C^{\frac{\alpha}{2},\alpha}([a,b]\times\mathbb{R}^d)= C^{\frac{\alpha}{2},0}([a,b]\times\mathbb{R}^d)\bigcap C^{0,\alpha}([a,b]\times\mathbb{R}^d),
\end{equation*}
the norm of which is given by 
\begin{equation*}
	|\varphi|_{C^{\frac{\alpha}{2},\alpha}([a,b]\times\mathbb{R}^d)}= |\varphi|_{C^{\frac{\alpha}{2},0}([a,b]\times\mathbb{R}^d)}+ |\varphi|_{C^{0,\alpha}([a,b]\times\mathbb{R}^d)}. 
\end{equation*}
Furthermore, we have 
\begin{equation*}
	\begin{split}
		C^{1+\frac{\alpha}{2},2+\alpha}([a,b]\times\mathbb{R}^d)=\Big\{\varphi\in C^{1,2}([a,b]\times\mathbb{R}^d):~\varphi_s,~\varphi_{y_iy_j}\in C^{\frac{\alpha}{2},\alpha}([a,b]\times\mathbb{R}^d), \\
		i,j=1,\ldots,d\Big\}  
	\end{split}
\end{equation*} 
with the following norm 
\begin{equation*}
	\begin{split}
		|\varphi|_{C^{1+\frac{\alpha}{2},2+\alpha}([a,b]\times\mathbb{R}^d)}=&~|\varphi|^\infty+|\varphi_s|_{C^{\frac{\alpha}{2},\alpha}([a,b]\times\mathbb{R}^d)} \\
		&+\sum^d_{i=1}|\varphi_{y_i}|^\infty+\sum^d_{i,j=1}|\varphi_{y_iy_j}|_{C^{\frac{\alpha}{2},\alpha}([a,b]\times\mathbb{R}^d)}.
	\end{split} 
\end{equation*}
For convenience, $|\cdot|_{C^{\frac{\alpha}{2},\alpha}([a,b]\times\mathbb{R}^d)}$ and $|\cdot|_{C^{1+\frac{\alpha}{2},2+\alpha}([a,b]\times\mathbb{R}^d)}$ are usually denoted as $|\cdot|^{(\alpha)}_{[a,b]\times\mathbb{R}^d}$ and $|\cdot|^{(2+\alpha)}_{[a,b]\times\mathbb{R}^d}$. In addition, wherever no confusion arises, we do not distinguish between $|\varphi(\cdot)|^{(\alpha)}_{[a,b]\times\mathbb{R}^d}$ and $|\varphi(\cdot)|^{(\alpha)}_{\mathbb{R}^d}$ and between $|\varphi(\cdot)|^{(2+\alpha)}_{[a,b]\times\mathbb{R}^d}$ and $|\varphi(\cdot)|^{(2+\alpha)}_{\mathbb{R}^d}$ for functions $\varphi(y)$ independent of $s$.  

After reviewing the classical H\"{o}lder spaces, we revise on top of them to fulfil requirements of nonlocal PDEs and TIC problems. Before that, we first investigate the following two essential features of nonlocal PDEs and TIC problems. 
\begin{enumerate}
	\item (\textbf{The order relation between $t$ and $s$})
	Since both TIC control problems and FBSVIEs are backward problems, $t$ and $s$ usually represent the initial time and the running time, respectively, where $0\leq t\leq s\leq T$. Considering a symmetry between forward and backward problems, we can formulate our nonlocal PDEs as forward problems and consider a reverse order relation, i.e. $0\leq s\leq t\leq T$. Given the relation between $t$ and $s$, we are concerned only with functions defined over the triangle $\Delta[0,\delta]$ as illustrated in Figure \ref{fig:triangledomain} instead of a rectangle $[0,\delta]^2$;
	\begin{figure}[!ht]
		\centering
		\includegraphics[width=0.5\textwidth]{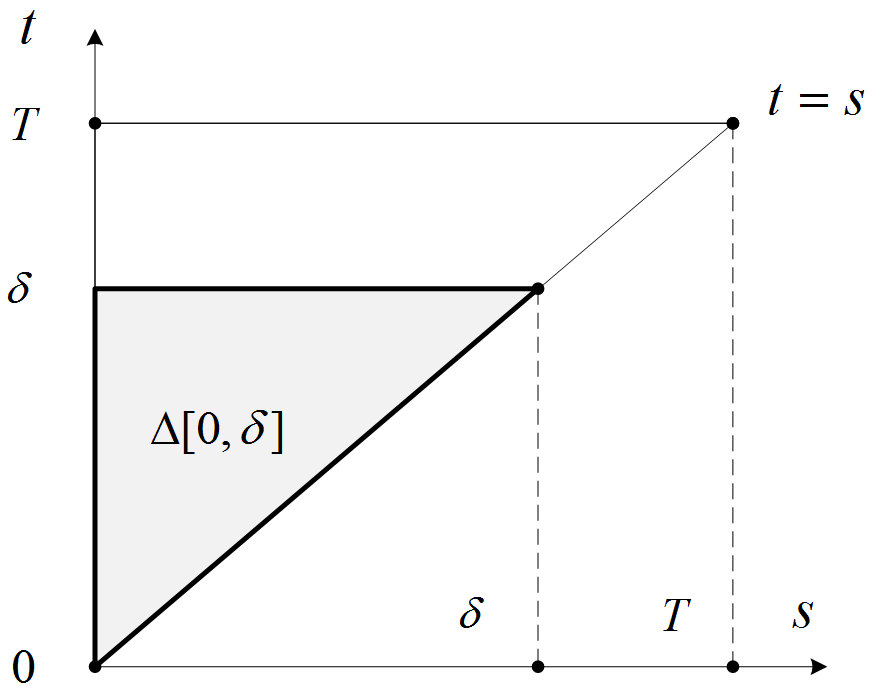}
		\caption{Time domain $\Delta[0,\delta]$ of solutions of nonlocal PDEs}
		\label{fig:triangledomain}
	\end{figure}
	\item (\textbf{Regularity and smoothness}) Inspired by parabolic H\"{o}lder spaces for local PDEs, it is natural to preserve sufficient smoothness of the functions with respect to the variables $s$ and $y$. Specifically, we will study the functions in space $C^{1+\frac{\alpha}{2},2+\alpha}$. Next, it is essential to determine regularity of the external parameter $t$ of $u(t,\cdot,\cdot)$. In this paper, we require that $u(t,s,y)$ is $C^1$ continuous with respect to $t$. Two main reasons behind the requirement are as follows: (a) when constructing the contractions and estimating solutions of nonlocal PDEs, we often encounter the evaluation of the difference $|\varphi(s,s,y)-\varphi(s^\prime,s^\prime,y)|$. The differentiability of $\varphi$ in $t$ makes some theorems like mean-value theorem available; (b) More importantly, the regularity of $u$ with respect to $t$ can help us convert nonlocal linear PDEs \eqref{Nonlocal linear PDE with t} into a system of coupled nonlocal PDEs \eqref{Nonlocal linear PDE system with t} for $(u,u_t)$. Consequently, the difficulty of studying the original nonlocal linear PDEs by fixed-point arguments can be resolved by investigating the induced system. Informally speaking, our successful settlement depends substantially on the contribution of the differentiability in $t$ to the establishment of contractions.
\end{enumerate}

After the preliminary analysis, we begin to define the norms and Banach spaces for nonlocal PDEs and TIC problems. For $0\leq\delta\leq T$, we introduce norms 
\begin{eqnarray*}
	[\phi]^{(k+\alpha)}_{[0,\delta]} & := & \sup\limits_{t\in[0,\delta]}\left\{|\phi(t,\cdot,\cdot)|^{(k+\alpha)}_{[0,t]\times\mathbb{R}^d}\right\}, \\
	\lVert\phi\rVert^{(k+\alpha)}_{[0,\delta]} & := & \sup\limits_{t\in[0,\delta]}\left\{|\phi(t,\cdot,\cdot)|^{(k+\alpha)}_{[0,t]\times\mathbb{R}^d}+|\phi_t(t,\cdot,\cdot)|^{(k+\alpha)}_{[0,t]\times\mathbb{R}^d}\right\}, 
\end{eqnarray*}
where $k=0,1,2$. Then they induce the following normed spaces, respectively,  
\begin{eqnarray*}
	\Theta^{(k+\alpha)}_{[0,\delta]} & := & \left\{\theta(\cdot,\cdot,\cdot)\in C(\Delta[0,\delta]\times\mathbb{R}^d):[\theta]^{(k+\alpha)}_{[0,\delta]}<\infty\right\}, \\
	\Omega^{(k+\alpha)}_{[0,\delta]} & := & \left\{\omega(\cdot,\cdot,\cdot)\in C(\Delta[0,\delta]\times\mathbb{R}^d): \Vert\omega\rVert^{(k+\alpha)}_{[0,\delta]}<\infty\right\}. 
\end{eqnarray*}
It is easy to see that both $[\phi]^{(k+\alpha)}_{[0,\delta]}$ and $\lVert\phi\rVert^{(k+\alpha)}_{[0,\delta]}$ are norms under which $\Theta^{(k+\alpha)}_{[0,\delta]}$ and $\Omega^{(k+\alpha)}_{[0,\delta]}$ are Banach spaces, respectively. The definitions above leverage not only the order relation between $t$ and $s$ but also the sufficient regularities in all arguments.


\subsection{Wellposedness of Nonlocal Linear Parabolic PDEs} \label{sec:linearresult} 
Now, we are ready to study the solvability of nonlocal linear parabolic PDEs with the form 
\begin{equation} \label{Nonlocal linear PDE with t}  
	\left\{
	\begin{array}{lr}
		u_s(t,s,y)=\sum^d\limits_{i,j=1} a_{ij}(t,s,y)u_{y_iy_j}(t,s,y)+\sum^d\limits_{i=1} b_i(t,s,y)u_{y_i}(t,s,y)+c(t,s,y)u(t,s,y) \\
		\qquad\qquad+\sum^d\limits_{i,j=1}\overline{a}_{ij}(t,s,y)u_{y_iy_j}(s,s,y)+\sum^d\limits_{i=1}\overline{b}_i(t,s,y)u_{y_i}(s,s,y)+\overline{c}(t,s,y)u(s,s,y) \\
		\qquad\qquad+f(t,s,y), \\
		
		\quad \\
		
		u(t,0,y)=g(t,y), \quad 0\leq s\leq t\leq T, \quad y\in\mathbb{R}^d.  
	\end{array}
	\right.
\end{equation} 
where all coefficients $a$, $\overline{a}$, $b$, $\overline{b}$, $c$, and $\overline{c}$ belong to $\Omega^{{(\alpha)}}_{[0,T]}$ and satisfy the uniform ellipticity condition, i.e., there exists some $\lambda>0$ such that 
\begin{eqnarray}
	\sum^d\limits_{i,j=1}a_{ij}(t,s,y)\xi_i\xi_j & \geq & \lambda|\xi|^2, \label{Uniform ellipticity condition 1} \\
	\sum^d\limits_{i,j=1}\left(a_{ij}(t,s,y)+\overline{a}_{ij}(t,s,y)\right)\xi_i\xi_j & \geq & \lambda|\xi|^2, \label{Uniform ellipticity condition 2}
\end{eqnarray}
for any $(t,s)\in\Delta[0,T]$ and $y,\xi\in\mathbb{R}^d$. Moreover, the nonhomogeneous term $f\in\Omega^{(\alpha)}_{[0,T]}$ and the initial condition $g\in\Omega^{(2+\alpha)}_{[0,T]}$.

Suppose that $u$ is a solution of \eqref{Nonlocal linear PDE with t}, then its first-order derivative $\frac{\partial u}{\partial t}$ with respect to $t$ should satisfy the following differential equation: 
\begin{equation} \label{Differentiate Nonlocal linear PDE with t with respect to t}   
	\left\{
	\begin{array}{lr}
		\left(\frac{\partial u}{\partial t}\right)_s(t,s,y)=\sum^d\limits_{i,j=1} a_{ij}(\cdot)\left(\frac{\partial u}{\partial t}\right)_{y_iy_j}(t,s,y)+\sum^d\limits_{i=1}b_i(\cdot)\left(\frac{\partial u}{\partial t}\right)_{y_i}(t,s,y)+c(\cdot)\left(\frac{\partial u}{\partial t}\right)(t,s,y) \\
		\qquad\qquad\qquad +\sum^d\limits_{i,j=1}\frac{\partial a_{ij}(\cdot)}{\partial t}u_{y_iy_j}(t,s,y)+\sum^d\limits_{i=1}\frac{\partial b_i(\cdot)}{\partial t}u_{y_i}(t,s,y)+c_t(\cdot)u(t,s,y) \\
		\qquad\qquad\qquad +\sum^d\limits_{i,j=1}\frac{\partial\overline{a}_{ij}(\cdot)}{\partial t}u_{y_iy_j}(s,s,y)+\sum^d\limits_{i=1}\frac{\partial\overline{b}_{i}(\cdot)}{\partial t}u_{y_i}(s,s,y)+\overline{c}_t(\cdot)u(s,s,y)+f_t(\cdot), \\
		
		u_t(t,0,y)=g_t(t,y), \quad 0\leq s\leq t\leq T, \quad y\in\mathbb{R}^d.  
	\end{array}
	\right.
\end{equation} 
where we use a convention that $\phi(\cdot)=\phi(t,s,y)$. Next, for any $i,j=1,\ldots,d$, we notice the following integral representations 
\begin{equation} \label{from u(s) to u(t)}
	\left\{
	\begin{aligned}
		u(t,s,y)-u(s,s,y)&=\int^t_s \frac{\partial u}{\partial t}(\theta,s,y) d\theta, \\
		u_{y_i}(t,s,y)-u_{y_i}(s,s,y)&=\int^t_s\left(\frac{\partial u}{\partial t}\right)_{y_i}(\theta,s,y)d\theta, \\
		u_{y_iy_j}(t,s,y)-u_{y_iy_j}(s,s,y)&=\int^t_s\left(\frac{\partial u}{\partial t}\right)_{y_iy_j}(\theta,s,y)d\theta. 
	\end{aligned}
	\right.
\end{equation}

\eqref{from u(s) to u(t)} allows us to rewrite the diagonal terms in \eqref{Nonlocal linear PDE with t} and \eqref{Differentiate Nonlocal linear PDE with t with respect to t} as integrals of $\frac{\partial u}{\partial t}$ or its derivatives.
Hence, by coupling equations \eqref{Nonlocal linear PDE with t} and \eqref{Differentiate Nonlocal linear PDE with t with respect to t} with \eqref{from u(s) to u(t)}, we obtain a nonlocal linear PDE system for $\left(u,v:=\frac{\partial u}{\partial t}\right)$:
\begin{equation} \label{Nonlocal linear PDE system with t}  
	\left\{
	\begin{array}{lr}
		u_s(t,s,y)=\sum^d\limits_{i,j=1}\left(a_{ij}(\cdot)+\overline{a}_{ij}(\cdot)\right)u_{y_iy_j}(t,s,y)+\sum^d\limits_{i=1}\left(b_i(\cdot)+\overline{b}_i(\cdot)\right)u_{y_i}(t,s,y) \\
		\qquad\qquad -\sum^d\limits_{i,j=1}\overline{a}_{ij}(\cdot)\int^t_s v_{y_iy_j}(\theta,s,y)d\theta -\sum^d\limits_{i=1}\overline{b}_i(\cdot)\int^t_s v_{y_i}(\theta,s,y)d\theta \\ \qquad\qquad+\left(c(\cdot)+\overline{c}(\cdot)\right)u(t,s,y)-\overline{c}(\cdot)\int^t_s v(\theta,s,y)d\theta +f(\cdot), \\
		
		v_s(t,s,y)=\sum^d\limits_{i,j=1}a_{ij}(\cdot)v_{y_iy_j}(t,s,y)+\sum^d_{i=1}b_i(\cdot)v_{y_i}(t,s,y)+c(\cdot)v(t,s,y) \\
		\qquad\qquad
		+\sum^d\limits_{i,j=1}\left(\frac{\partial a_{ij}(\cdot)}{\partial t}+\frac{\partial\overline{a}_{ij}(\cdot)}{\partial t}\right)u_{y_iy_j}(t,s,y)+\sum^d\limits_{i=1}\left(\frac{\partial b_{i}(\cdot)}{\partial t}+\frac{\partial\overline{b}_i(\cdot)}{\partial t}\right)u_{y_i}(t,s,y) \\
		\qquad\qquad -\sum^d\limits_{i,j=1}\frac{\partial\overline{a}_{ij}(\cdot)}{\partial t}\int^t_s v_{y_iy_j}(\theta,s,y)d\theta -\sum^d\limits_{i=1}\frac{\partial\overline{b}_i(\cdot)}{\partial t}\int^t_s v_{y_i}(\theta,s,y)d\theta \\
		\qquad\qquad+\left(c_t(\cdot)+\overline{c}_t(\cdot)\right)u(t,s,y)-\overline{c}_t(\cdot)\int^t_s v(\theta,s,y)d\theta+f_t(\cdot), \\
		
		\left(u,v\right)(t,0,y)=\left(g,g_t\right)(t,y), \quad 0\leq s\leq t\leq T, \quad y\in\mathbb{R}^d.  
	\end{array}
	\right.
\end{equation} 

Before studying the well-posedness of the system \eqref{Nonlocal linear PDE system with t} for $(u,v)$, we first prove the equivalence between \eqref{Nonlocal linear PDE with t} and \eqref{Nonlocal linear PDE system with t} in the following lemma.

\begin{lemma}\label{Equivalence between equation and system} The nonlocal equation \eqref{Nonlocal linear PDE with t} is equivalent to the system \eqref{Nonlocal linear PDE system with t}, i.e.,
	\begin{enumerate}
		\item If $u$ is a solution of \eqref{Nonlocal linear PDE with t}, then $(u,u_t)$ solves \eqref{Nonlocal linear PDE system with t}.
		\item Conversely, if \eqref{Nonlocal linear PDE system with t} admits a solution pair $(u,v)$, then $u$ solves \eqref{Nonlocal linear PDE with t} and $v=u_t$.
	\end{enumerate}
\end{lemma}

\begin{proof}
	The first claim is straightforward as it follows how we derived the equations \eqref{Nonlocal linear PDE system with t} before. Next, we are focused on proving the second claim. 
	
	Let $(u,v)$ be a solution of \eqref{Nonlocal linear PDE system with t}. Differentiating the first equation of \eqref{Nonlocal linear PDE system with t} with respect to $t$ gives
	\begin{equation} \label{u_t in lemma}   
		\left\{
		\begin{array}{lr}
			\left(\frac{\partial u}{\partial t}\right)_s(t,s,y)=\sum^d\limits_{i,j=1}\left(a_{ij}(\cdot)+\overline{a}_{ij}(\cdot)\right) \left(\frac{\partial u}{\partial t}\right)_{y_iy_j}(t,s,y)+\sum^d\limits_{i,j=1}\left(\frac{\partial a_{ij}(\cdot)}{\partial t}+\frac{\partial\overline{a}_{ij}(\cdot)}{\partial t}\right)u_{y_iy_j}(t,s,y) \\
			\qquad\qquad\qquad+\sum^d\limits_{i=1}\left(b_i(\cdot)+\overline{b}_i(\cdot)\right) \left(\frac{\partial u}{\partial t}\right)_{y_i}(t,s,y)+\sum^d\limits_{i=1}\left(\frac{\partial b_i(\cdot)}{\partial t}+\frac{\partial\overline{b}_i(\cdot)}{\partial t}\right)u_{y_i}(t,s,y) \\
			\qquad\qquad\qquad+\left(c(\cdot)+\overline{c}(\cdot)\right)\left(\frac{\partial u}{\partial t}\right)(t,s,y)+\left(\frac{\partial c(\cdot)}{\partial t}+\frac{\partial\overline{c}(\cdot)}{\partial t}\right)u(t,s,y) \\
			\qquad\qquad\qquad -\sum^d\limits_{i,j=1}\overline{a}_{ij}(\cdot)v_{y_iy_j}(t,s,y)-\sum^d\limits_{i,j=1}\frac{\partial\overline{a}_{ij}(\cdot)}{\partial t}\int^t_s v_{y_iy_j}(\theta,s,y)d\theta \\ \qquad\qquad\qquad-\sum^d\limits_{i=1}\overline{b}_i(\cdot)v_{y_i}(t,s,y)-\sum^d\limits_{i=1}\frac{\partial\overline{b}_i(\cdot)}{\partial t}\int^t_s v_{y_i}(\theta,s,y)d\theta \\ \qquad\qquad\qquad-\overline{c}(\cdot)v(t,s,y)-\frac{\partial\overline{c}(\cdot)}{\partial t}\int^t_s v(\theta,s,y)d\theta +f_t(\cdot), \\
			
			\left(\frac{\partial u}{\partial t}\right)(t,0,y)=g_t(t,y), \quad 0\leq s\leq t\leq T, \quad y\in\mathbb{R}^d.  
		\end{array}
		\right.
	\end{equation} 
	
	From \eqref{u_t in lemma} and \eqref{Nonlocal linear PDE system with t}, we can find that $\left(\frac{\partial u}{\partial t}-v\right)(t,s,y)$ satisfies:  
	\begin{equation} \label{u_t-v in lemma}   
		\left\{
		\begin{array}{lr}
			\left(\frac{\partial u}{\partial t}-v\right)_s(t,s,y)=\sum^d\limits_{i,j=1}\left(a_{ij}(\cdot)+\overline{a}_{ij}(\cdot)\right) \left(\frac{\partial u}{\partial t}-v\right)_{y_iy_j}(t,s,y) \\
			\qquad \qquad\qquad\qquad+\sum^d\limits_{i=1}\left(b_i(\cdot)+\overline{b}_i(\cdot)\right) \left(\frac{\partial u}{\partial t}-v\right)_{y_i}(t,s,y) \\
			\qquad \qquad\qquad\qquad+\left(c(\cdot)+\overline{c}(\cdot)\right)\left(\frac{\partial u}{\partial t}-v\right)(t,s,y) \\
			
			\left(\frac{\partial u}{\partial t}-v\right)(t,0,y)=0, \quad 0\leq s\leq t\leq T, \quad y\in\mathbb{R}^d.  
		\end{array}
		\right.
	\end{equation} 
	which implies that $\left(\frac{\partial u}{\partial t}-v\right)(t,s,y)\equiv 0$, thanks to the classical PDE theory. 
	
	Next, we replace $v$ in the first equation of \eqref{Nonlocal linear PDE system with t} with $\frac{\partial u}{\partial t}$. Then we have 
	\begin{align*}
		& u_s(t,s,y)=  \sum^d\limits_{i,j=1}\left(a_{ij}(\cdot)+\overline{a}_{ij}(\cdot)\right)u_{y_iy_j}(t,s,y)+\sum^d\limits_{i=1}\left(b_i(\cdot)+\overline{b}_i(\cdot)\right)u_{y_i}(t,s,y) \\
		&\qquad\qquad\qquad -\sum^d\limits_{i,j=1}\overline{a}_{ij}(\cdot)\int^t_s \left(\frac{\partial u}{\partial t}\right)_{y_iy_j}(\theta,s,y)d\theta -\sum^d\limits_{i=1}\overline{b}_i(\cdot)\int^t_s \left(\frac{\partial u}{\partial t}\right)_{y_i}(\theta,s,y)d\theta \\
		&\qquad\qquad\qquad+\left(c(\cdot)+\overline{c}(\cdot)\right)u(t,s,y)-\overline{c}(\cdot)\int^t_s \left(\frac{\partial u}{\partial t}\right)(\theta,s,y)d\theta +f(\cdot) \\
		&\qquad=\sum^d\limits_{i,j=1} a_{ij}(\cdot)u_{y_iy_j}(t,s,y)+\sum^d\limits_{i=1} b_i(\cdot)u_{y_i}(t,s,y)+c(\cdot)u(t,s,y)\qquad\qquad\qquad \\
		&\qquad\qquad\qquad+\sum^d\limits_{i,j=1}\overline{a}_{ij}(\cdot)u_{y_iy_j}(s,s,y)+\sum^d\limits_{i=1}\overline{b}_i(\cdot)u_{y_i}(s,s,y)+\overline{c}(\cdot)u(s,s,y)+f(\cdot)
	\end{align*}
	Therefore, the result follows.
\end{proof}

With the equivalence between \eqref{Nonlocal linear PDE with t} and \eqref{Nonlocal linear PDE system with t}, the well-posedness of \eqref{Nonlocal linear PDE system with t} guarantees the solvability of \eqref{Nonlocal linear PDE with t}. Next, for the coupled system for $(u,v)$ induced from \eqref{Nonlocal linear PDE with t}, we have the following conclusions.

\begin{theorem} \label{Well-posedness of (u,u_t)}
	Suppose that all coefficient functions and $f$ of \eqref{Nonlocal linear PDE with t} belong to $\Omega^{{(\alpha)}}_{[0,T]}$ and assume that $g\in\Omega^{{(2+\alpha)}}_{[0,T]}$. Then there exist $\delta>0$ and a unique solution pair $(u,v)\in\Theta^{{(2+\alpha)}}_{[0,\delta]}\times\Theta^{{(2+\alpha)}}_{[0,\delta]}$ satisfying \eqref{Nonlocal linear PDE system with t} in $\Delta[0,\delta]\times\mathbb{R}^d$. 
\end{theorem}
\begin{proof}
	The theorem is proven with fixed-point arguments and the overall idea is as follows. We first establish a contraction $\Gamma$ defined over a closed subset $\mathcal{V}$ with a radius $R$ in $\Theta^{{(2+\alpha)}}_{[0,\delta]}$ and then show that it admits a unique fixed point in $\mathcal{V}$. Finally, a contradiction is constructed to argue for the uniqueness of the solution to the system \eqref{Nonlocal linear PDE system with t} in $\Theta^{{(2+\alpha)}}_{[0,\delta]}\times\Theta^{{(2+\alpha)}}_{[0,\delta]}$. \\
	
	\noindent (\textbf{Existence}) We make use of \eqref{Nonlocal linear PDE system with t} to construct a mapping from a conservative vector field $v$ to a conservative vector field $V$ by $\Gamma(v)=V$, where $V$ is part of the solution pair $(u,V)$ to a system of equations:
	\begin{equation} \label{Contraction from v to V}  
		\left\{
		\begin{array}{lr}
			u_s(t,s,y)=\sum^d\limits_{i,j=1}\left(a_{ij}(\cdot)+\overline{a}_{ij}(\cdot)\right)u_{y_iy_j}(t,s,y)+\sum^d\limits_{i=1}\left(b_i(\cdot)+\overline{b}_i(\cdot)\right)u_{y_i}(t,s,y) \\
			\qquad\qquad -\sum^d\limits_{i,j=1}\overline{a}_{ij}(\cdot)\int^t_s v_{y_iy_j}(\theta,s,y)d\theta -\sum^d\limits_{i=1}\overline{b}_i(\cdot)\int^t_s v_{y_i}(\theta,s,y)d\theta \\ \qquad\qquad+\left(c(\cdot)+\overline{c}(\cdot)\right)u(t,s,y)-\overline{c}(\cdot)\int^t_s v(\theta,s,y)d\theta +f(\cdot), \\
			
			V_s(t,s,y)=\sum^d\limits_{i,j=1}a_{ij}(\cdot)V_{y_iy_j}(t,s,y)+\sum^d_{i=1}b_i(\cdot)V_{y_i}(t,s,y)+c(\cdot)V(t,s,y) \\
			\qquad\qquad 
			+\sum^d\limits_{i,j=1}\left(\frac{\partial a_{ij}(\cdot)}{\partial t}+\frac{\partial\overline{a}_{ij}(\cdot)}{\partial t}\right)u_{y_iy_j}(t,s,y)+\sum^d\limits_{i=1}\left(\frac{\partial b_{i}(\cdot)}{\partial t}+\frac{\partial\overline{b}_i(\cdot)}{\partial t}\right)u_{y_i}(t,s,y) \\
			\qquad\qquad -\sum^d\limits_{i,j=1}\frac{\partial\overline{a}_{ij}(\cdot)}{\partial t}\int^t_s v_{y_iy_j}(\theta,s,y)d\theta -\sum^d\limits_{i=1}\frac{\partial\overline{b}_i(\cdot)}{\partial t}\int^t_s v_{y_i}(\theta,s,y)d\theta \\
			\qquad\qquad+\left(c_t(\cdot)+\overline{c}_t(\cdot)\right)u(t,s,y)-\overline{c}_t(\cdot)\int^t_s v(\theta,s,y)d\theta+f_t(\cdot), \\
			
			\left(u,V\right)(t,0,y)=\left(g,g_t\right)(t,y), \quad 0\leq s\leq t\leq \delta, \quad y\in\mathbb{R}^d.  
		\end{array}
		\right.
	\end{equation} 
	The operator $\Gamma(v)=V$ is defined in the set 
	\begin{equation*}
		\mathcal{V}=\left\{v\in\Theta^{(2+\alpha)}_{[0,\delta]}:~v(t,0,y)=g_t(t,y),~[v-g]^{(2+\alpha)}_{[0,\delta]}\leq R\right\}.
	\end{equation*}
	
	It is clear that given a conservative vector field $v$ satisfying the initial condition $g_t(t,y)$, the system \eqref{Contraction from v to V} admits a unique solution $V$ and thus $\Gamma(\cdot)$ is well-defined.
	
	Next, we shall show that the mapping $\Gamma$ has a unique fixed point. Let $v$, $\widehat{v}\in \mathcal{V}$. Then, $w:=\Gamma(v)-\Gamma(\widehat{v})$ is part of the solution pair $(u-\widehat{u},w)$ to a system of equations:
	\begin{equation} \label{Difference between v to v hat}
		\left\{
		\begin{array}{lr}
			\left(u-\widehat{u}\right)_s(t,s,y)=\sum^d\limits_{i,j=1}\left(a_{ij}(\cdot)+\overline{a}_{ij}(\cdot)\right)\left(u-\widehat{u}\right)_{y_iy_j}(t,s,y) \\
			\qquad\qquad\qquad+\sum^d\limits_{i=1}\left(b_i(\cdot)+\overline{b}_i(\cdot)\right)\left(u-\widehat{u}\right)_{y_i}(t,s,y) \\
			\qquad\qquad\qquad -\sum^d\limits_{i,j=1}\overline{a}_{ij}(\cdot)\int^t_s \left(v-\widehat{v}\right)_{y_iy_j}(\theta,s,y)d\theta -\sum^d\limits_{i=1}\overline{b}_i(\cdot)\int^t_s \left(v-\widehat{v}\right)_{y_i}(\theta,s,y)d\theta \\
			\qquad\qquad\qquad+\left(c(\cdot)+\overline{c}(\cdot)\right)\left(u-\widehat{u}\right)(t,s,y)-\overline{c}(\cdot)\int^t_s \left(v-\widehat{v}\right)(\theta,s,y)d\theta, \\
			
			w_s(t,s,y)=\sum^d\limits_{i,j=1}a_{ij}(\cdot)w_{y_iy_j}(t,s,y)+\sum^d_{i=1}b_i(\cdot)w_{y_i}(t,s,y)+c(\cdot)w(t,s,y) \\
			\qquad\qquad\qquad 
			+\sum^d\limits_{i,j=1}\left(\frac{\partial a_{ij}(\cdot)}{\partial t}+\frac{\partial\overline{a}_{ij}(\cdot)}{\partial t}\right)\left(u-\widehat{u}\right)_{y_iy_j}(t,s,y) \\
			\qquad\qquad\qquad+\sum^d\limits_{i=1}\left(\frac{\partial b_{i}(\cdot)}{\partial t}+\frac{\partial\overline{b}_i(\cdot)}{\partial t}\right)\left(u-\widehat{u}\right)_{y_i}(t,s,y) \\
			\qquad\qquad\qquad -\sum^d\limits_{i,j=1}\frac{\partial\overline{a}_{ij}(\cdot)}{\partial t}\int^t_s \left(v-\widehat{v}\right)_{y_iy_j}(\theta,s,y)d\theta -\sum^d\limits_{i=1}\frac{\partial\overline{b}_i(\cdot)}{\partial t}\int^t_s \left(v-\widehat{v}\right)_{y_i}(\theta,s,y)d\theta \\
			\qquad\qquad\qquad+\left(c_t(\cdot)+\overline{c}_t(\cdot)\right)\left(u-\widehat{u}\right)(t,s,y)-\overline{c}_t(\cdot)\int^t_s \left(v-\widehat{v}\right)(\theta,s,y)d\theta, \\
			
			\left(u-\widehat{u},w\right)(t,0,y)=\left(0,0\right), \quad 0\leq s\leq t\leq \delta, \quad y\in\mathbb{R}^d.  
		\end{array}
		\right.
	\end{equation} 
	
	By the classical theory of parabolic PDEs \cite{Friedman1964,Ladyzhanskaya1968,Lunardi1995}, for a fix $t\in[0,\delta]$, we have the following estimates for the solution to the first equation of \eqref{Difference between v to v hat}: 
	\begin{equation} \label{estimates of (u-u hat)(t,s,y)} 
		\begin{split}
			&\left|\left(u-\widehat{u}\right)(t,\cdot,\cdot)\right|^{(2+\alpha)}_{[0,t]\times\mathbb{R}^d} \\
			\leq &~ c\left(\sum^d_{i,j=1}\left|\int^t_\cdot\left(v-\widehat{v}\right)_{y_iy_j}(\theta,\cdot,\cdot)d\theta\right|^{(\alpha)}_{[0,t]\times\mathbb{R}^d}+\sum^d_{i=1}\left|\int^t_\cdot\left(v-\widehat{v}\right)_{y_i}(\theta,\cdot,\cdot)d\theta\right|^{(\alpha)}_{[0,t]\times\mathbb{R}^d}\right. \\
			&\qquad 
			\left.+\left|\int^t_\cdot\left(v-\widehat{v}\right)(\theta,\cdot,\cdot)d\theta\right|^{(\alpha)}_{[0,t]\times\mathbb{R}^d}\right), 
		\end{split}
	\end{equation}
	where the constant $c$ depends on $\lVert a\rVert^{(\alpha)}_{[0,T]}$, $\lVert \overline{a}\rVert^{(\alpha)}_{[0,T]}$, $\lVert b\rVert^{(\alpha)}_{[0,T]}$, $\lVert \overline{b}\rVert^{(\alpha)}_{[0,T]}$, $\lVert c\rVert^{(\alpha)}_{[0,T]}$ and $\lVert \overline{c}\rVert^{(\alpha)}_{[0,T]}$. To further estimate the upper bound of \eqref{estimates of (u-u hat)(t,s,y)}, we investigate the H\"{o}lder continuity of $\int^t_s \left(v-\widehat{v}\right)_{y_iy_j}(\theta,s,y)d\theta$ with respect to $s$ and $y$ for any $y_i$ and $y_j$, while the similar analyses can be conducted for other two integral terms in \eqref{estimates of (u-u hat)(t,s,y)}.
	
	Let $\delta\leq 1$ and $0\leq s\leq s^\prime\leq t\leq \delta\leq T$, we have 
	\begin{equation} \label{Holder estimate with respect to s}
		\begin{split}
			&\left|\int^t_s \left(v-\widehat{v}\right)_{y_iy_j}(\theta,s,y)d\theta-\int^t_{s^\prime}\left(v-\widehat{v}\right)_{y_iy_j}(\theta,s^\prime,y)d\theta\right| \\
			\leq&\left|\int^t_s \left(v-\widehat{v}\right)_{y_iy_j}(\theta,s,y)d\theta-\int^t_{s^\prime} \left(v-\widehat{v}\right)_{y_iy_j}(\theta,s,y)d\theta\right| \\
			&\qquad\qquad\qquad+\left|\int^t_{s^\prime} \left(v-\widehat{v}\right)_{y_iy_j}(\theta,s,y)d\theta-\int^t_{s^\prime} \left(v-\widehat{v}\right)_{y_iy_j}(\theta,s^\prime,y)d\theta\right| \\ 
			\leq& \int^{s^\prime}_s\left|\left(v-\widehat{v}\right)_{y_iy_j}(\theta,s,y)\right|d\theta+\int^t_{s^\prime}\left|\left(v-\widehat{v}\right)_{y_iy_j}(\theta,s,y)-\left(v-\widehat{v}\right)_{y_iy_j}(\theta,s^\prime,y)\right|d\theta \\
			\leq&\left(\sup\limits_{t\in[0,\delta]}\left\{\left|\left(v-\widehat{v}\right)(t,\cdot,\cdot)\right|^{(2+\alpha)}_{[0,t]\times\mathbb{R}^d}\right\}(s^\prime-s)^{1-\frac{\alpha}{2}}\right. \\
			&\qquad\qquad\qquad\left.+\sup\limits_{t\in[0,\delta]}\left\{\left|\left(v-\widehat{v}\right)(t,\cdot,\cdot)\right|^{(2+\alpha)}_{[0,t]\times\mathbb{R}^d}\right\}(t-s^\prime)\right)\cdot\left|s^\prime-s\right|^{\frac{\alpha}{2}} \\
			\leq&\sup\limits_{t\in[0,\delta]}\left\{\left|\left(v-\widehat{v}\right)(t,\cdot,\cdot)\right|^{(2+\alpha)}_{[0,t]\times\mathbb{R}^d}\right\}\delta^{1-\frac{\alpha}{2}}\cdot\left|s^\prime-s\right|^{\frac{\alpha}{2}} 
		\end{split}
	\end{equation} 
	which implies also, by noting $\left(v-\widehat{v}\right)(\cdot,0,\cdot)\equiv 0$, that
	\begin{equation} \label{Boundness of v-vhat}
		\left|\int^t_\cdot \left(v-\widehat{v}\right)_{y_iy_j}(\theta,\cdot,\cdot)d\theta\right|^\infty_{[0,t]\times\mathbb{R}^d}\leq\sup\limits_{t\in[0,\delta]}\left\{\left|\left(v-\widehat{v}\right)(t,\cdot,\cdot)\right|^{(2+\alpha)}_{[0,t]\times\mathbb{R}^d}\right\}\delta 
	\end{equation} 
	
	Moreover, let $y$, $y^\prime\in\mathbb{R}^d$, then we have
	\begin{equation} \label{Holder estimate with respect to y}
		\begin{split}
			&\left|\int^t_s \left(v-\widehat{v}\right)_{y_iy_j}(\theta,s,y^\prime)d\theta-\int^t_s\left(v-\widehat{v}\right)_{y_iy_j}(\theta,s,y)d\theta\right| \\
			\leq&\int^t_s\left|\left(v-\widehat{v}\right)_{y_iy_j}(\theta,s,y^\prime)-\left(v-\widehat{v}\right)_{y_iy_j}(\theta,s,y)\right|d\theta \\
			\leq& \sup\limits_{t\in[0,\delta]}\left\{\left|\left(v-\widehat{v}\right)(t,\cdot,\cdot)\right|^{(2+\alpha)}_{[0,t]\times\mathbb{R}^d}\right\}(t-s)\cdot\left|y^\prime-y\right|^\alpha \\
			\leq& \sup\limits_{t\in[0,\delta]}\left\{\left|\left(v-\widehat{v}\right)(t,\cdot,\cdot)\right|^{(2+\alpha)}_{[0,t]\times\mathbb{R}^d}\right\}\delta\cdot\left|y^\prime-y\right|^\alpha. 
		\end{split}
	\end{equation}
	
	Consequently, from \eqref{Holder estimate with respect to s}, \eqref{Boundness of v-vhat}, and \eqref{Holder estimate with respect to y}, we have 
	\begin{equation*}
		\left|\int^t_\cdot\left(v-\widehat{v}\right)_{y_iy_j}(\theta,\cdot,\cdot)d\theta\right|^{(\alpha)}_{[0,t]\times\mathbb{R}^d}\leq \delta^{1-\frac{\alpha}{2}}\cdot\sup\limits_{t\in[0,\delta]}\left\{\left|\left(v-\widehat{v}\right)(t,\cdot,\cdot)\right|^{(2+\alpha)}_{[0,t]\times\mathbb{R}^d}\right\}
	\end{equation*}
	
	Similarly, we can also estimate other lower-order terms of \eqref{estimates of (u-u hat)(t,s,y)}. Then, 
	\begin{equation} \label{Holder estimates of (u-u hat)(t,s,y)} 
		\left|\left(u-\widehat{u}\right)(t,\cdot,\cdot)\right|^{(2+\alpha)}_{[0,t]\times\mathbb{R}^d}\leq C_1\delta^{1-\frac{\alpha}{2}}\cdot\sup\limits_{t\in[0,\delta]}\left\{\left|\left(v-\widehat{v}\right)(t,\cdot,\cdot)\right|^{(2+\alpha)}_{[0,t]\times\mathbb{R}^d}\right\}. 
	\end{equation} 
	
	By the classical theory of parabolic PDEs and \eqref{Holder estimates of (u-u hat)(t,s,y)}, we can also acquire the inequality from the second equation of \eqref{Difference between v to v hat}:   
	\begin{equation*}  
		\left|w(t,\cdot,\cdot)\right|^{(2+\alpha)}_{[0,t]\times\mathbb{R}^d}\leq C_2\delta^{1-\frac{\alpha}{2}}\cdot\sup\limits_{t\in[0,\delta]}\left\{\left|\left(v-\widehat{v}\right)(t,\cdot,\cdot)\right|^{(2+\alpha)}_{[0,t]\times\mathbb{R}^d}\right\}. 
	\end{equation*} 
	
	Consequently, under the norm $[\phi]^{(2+\alpha)}_{[0,\delta]}:=\sup\limits_{t\in[0,\delta]}\left\{|\phi(t,\cdot,\cdot)|^{(2+\alpha)}_{[0,t]\times\mathbb{R}^d}\right\}$, choosing a suitably small $\delta$ yields that 
	\begin{equation*} \label{Contraction of Gamma} 
		\left[\Gamma(v)-\Gamma(\widehat{v})\right]^{(2+\alpha)}_{[0,\delta]} =\sup\limits_{t\in[0,\delta]}\left\{\left|w(t,\cdot,\cdot)\right|^{(2+\alpha)}_{[0,t]\times\mathbb{R}^d}\right\}\leq \frac{1}{2}\left[ v-\widehat{v}\right]^{(2+\alpha)}_{[0,\delta]}.
	\end{equation*} 
	
	\noindent (\textbf{A contraction $\Gamma$ mapping $\mathcal{V}$ into itself.}) On the other hand, we also have 
	\begin{equation*} 
		\begin{split}
			\left[\Gamma(v)-g_t\right]^{(2+\alpha)}_{[0,\delta]} &\leq\left[\Gamma(v)-\Gamma(g_t)\right]^{(2+\alpha)}_{[0,\delta]}+\left[\Gamma(g_t)-g_t\right]^{(2+\alpha)}_{[0,\delta]} \\
			&\leq \frac{1}{2}\left[ v-g_t\right]^{(2+\alpha)}_{[0,\delta]}+C^\prime\leq \frac{R}{2}+C^\prime
		\end{split}
	\end{equation*} 
	Therefore, for a suitably large $R$, $\Gamma$ is a contraction mapping $\mathcal{V}$ into itself and thus it has
	a unique fixed point $v$ in $\mathcal{V}$ such that $\Gamma(v)=v$. \\
	
	\noindent (\textbf{Uniqueness}) To complete the proof, we ought to show that the conservative vector field $v$ is the unique fixed point of \eqref{Contraction from v to V} in $\Theta^{(2+\alpha)}_{[0,\delta]}$. It can be done with some standard arguments. If \eqref{Contraction from v to V} admits two fixed points $v^1$ and $v^2$, let
	$$
	t_0=\sup\left\{t\in[0,\delta]:~v^1(t,s,y)=v^2(t,s,y),~(t,s,y)\in\Delta[0,t]\times\mathbb{R}^d\right\}.
	$$
	We shall focus only on the case when $t_0<\delta$ because if $t_0=\delta$, then $v^1=v^2$ in the whole $\Delta[0,\delta]\times\mathbb{R}^d$ and the proof is completed. 
	
	According to the definition of $t_0(<\delta)$, we know that $v^1(t,s,y)=v^2(t,s,y)$ in $\Delta_1\times\mathbb{R}^d$ in Figure \ref{fig:uniqueness}. Furthermore, their corresponding $u^1(t,s,y)$ and $u^2(t,s,y)$ are also equal in $\Delta_1\times\mathbb{R}^d$ according to the first equation of \eqref{Contraction from v to V}. Hence, we obtain diagonal conditions, namely $u^1(s,s,y)=u^2(s,s,y)$, $u^1_y(s,s,y)=u^2_y(s,s,y)$ and $u^1_{yy}(s,s,y)=u^2_{yy}(s,s,y)$ for any $s\in[0,t_0]$ and $y\in\mathbb{R}^d$. By observing \eqref{Nonlocal linear PDE with t} and \eqref{Differentiate Nonlocal linear PDE with t with respect to t} provided that the same initial and diagonal conditions, i.e. the initial condition 1 and the diagonal condition in Figure \ref{fig:uniqueness}, the classical PDE theory promises that $(u^1,v^1)$ and $(u^2,v^2)$ coincide in $\left(R\cup\Delta_1\right)\times\mathbb{R}^d$.
	\begin{figure}[!ht]
		\centering
		\includegraphics[width=0.6\textwidth]{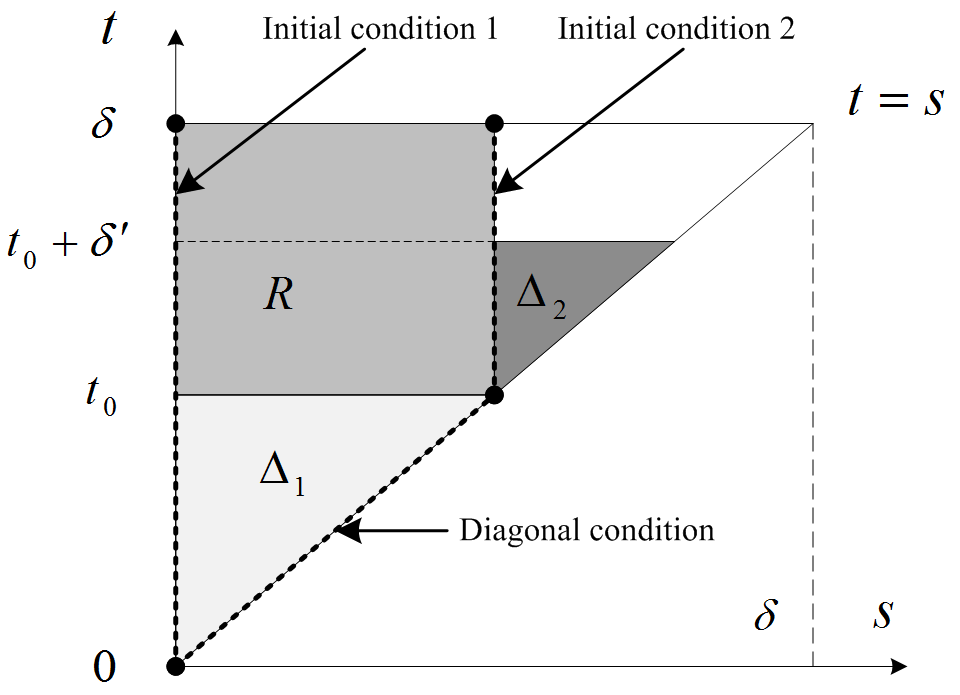}
		\caption{Uniqueness of the solution in $\Theta[0,\delta]$}
		\label{fig:uniqueness}
	\end{figure}
	
	Next, let $(u^1(t,t_0,y),v^1(t,t_0,y))=(u^2(t,t_0,y),v^2(t,t_0,y))=(g^\prime(t,y),g^\prime_t(t,y))$. Based on the new initial condition, i.e. the initial condition 2 in Figure \ref{fig:uniqueness}, we consider the following initial value problem: 
	\begin{equation} \label{Nonlocal linear PDE system with a new initial value condition}   
		\left\{
		\begin{array}{lr}
			u_s(t,s,y)=\left(a(\cdot)+\overline{a}(\cdot)\right)u_{yy}(t,s,y)+\left(b(\cdot)+\overline{b}(\cdot)\right)u_y(t,s,y) \\
			\qquad\qquad\qquad -\overline{a}(\cdot)\int^t_s v_{yy}(\theta,s,y)d\theta -\overline{b}(\cdot)\int^t_s v_y(\theta,s,y)d\theta \\
			\qquad\qquad\qquad +\left(c(\cdot)+\overline{c}(\cdot)\right)u(t,s,y)-\overline{c}(\cdot)\int^t_s v(\theta,s,y)d\theta+f(t,s,y), \\
			
			v_s(t,s,y)=a(\cdot)v_{yy}(t,s,y)+b(\cdot)v_y(t,s,y)+c(\cdot)v(t,s,y) \\
			\qquad\qquad\qquad +\left(a_t(\cdot)+\overline{a}_t(\cdot)\right)u_{yy}(t,s,y)+\left(b_t(\cdot)+\overline{b}_t(\cdot)\right)u_{y}(t,s,y) \\
			\qquad\qquad\qquad -\overline{a}_t(\cdot)\int^t_s v_{yy}(\theta,s,y)d\theta -\overline{b}_t(\cdot)\int^t_s v_y(\theta,s,y)d\theta \\
			\qquad\qquad\qquad+\left(c_t(\cdot)+\overline{c}_t(\cdot)\right)u(t,s,y)-\overline{c}_t(\cdot)\int^t_s v(\theta,s,y)d\theta+f_t(t,s,y), \\
			
			\left(u,v\right)(t,t_0,y)=\left(g^\prime,g^\prime_t\right)(t,y), \quad t_0\leq s\leq t\leq T, \quad y\in\mathbb{R}^d.  
		\end{array}
		\right.
	\end{equation} 
	Our previous proof shows that \eqref{Nonlocal linear PDE system with a new initial value condition} admits a unique $v$ in the set $$\mathcal{V}^\prime=\left\{v\in\Theta^{(2+\alpha)}_{[t_0,t_0+\delta^\prime]}:v(t,t_0,y)=g^\prime_t(t,y),[v-g^\prime]^{(2+\alpha)}_{[t_0,t_0+\delta^\prime]}\leq R^\prime\right\}$$ 
	provided that $R^\prime$ is large enough and $\delta^\prime$ is small enough. Considering $R^\prime$ larger than $[v^1-g^\prime]^{(2+\alpha)}_{[t_0,t_0+\delta^\prime]}$ and $[v^2-g^\prime]^{(2+\alpha)}_{[t_0,t_0+\delta^\prime]}$, we have $v^1=v^2$ in $\Delta_2\times\mathbb{R}^d$. Hence, for any $y\in\mathbb{R}^d$, $v^1$ equals to $v^2$ in $\{(t,s):t_0\leq t\leq t_0+\delta^\prime,0\leq s\leq t_0])\}\cup\Delta_2$. This contradicts the definition of $t_0$. Consequently, $t_0=\delta$ and $v^1=v^2$. 
	
	Finally, the unique fixed point $v$ determines uniquely a function $u(t,s,y)$ via \eqref{Contraction from v to V}. Moreover, given $v\in\Theta^{{(2+\alpha)}}_{[0,\delta]}$, it is clear that   
	\begin{equation*}
		[u]^{(2+\alpha)}_{[0,\delta]\times\mathbb{R}^d}\leq c\left([v]^{(2+\alpha)}_{[0,\delta]\times\mathbb{R}^d}+[f]^{(\alpha)}_{[0,\delta]\times\mathbb{R}^d}+[g]^{(2+\alpha)}_{[0,\delta]\times\mathbb{R}^d}\right)<\infty. 
	\end{equation*}
	Therefore, there exists a unique solution pair $(u,v)\in\Theta^{{(2+\alpha)}}_{[0,\delta]}\times\Theta^{{(2+\alpha)}}_{[0,\delta]}$ to \eqref{Nonlocal linear PDE system with t} in $\Delta[0,\delta]\times\mathbb{R}^d$.
\end{proof}

\begin{remark}[\textbf{Extension of solutions to a larger time interval}] \label{Extension of solutions to a larger time interval}
	Until now, we have proven the local existence of $(u,v)$ in $R_1=\{0\le s\le t\le \delta\}$ in Figure \ref{fig:extension}.
	\begin{figure}[!ht]
		\centering
		\includegraphics[width=0.5\textwidth]{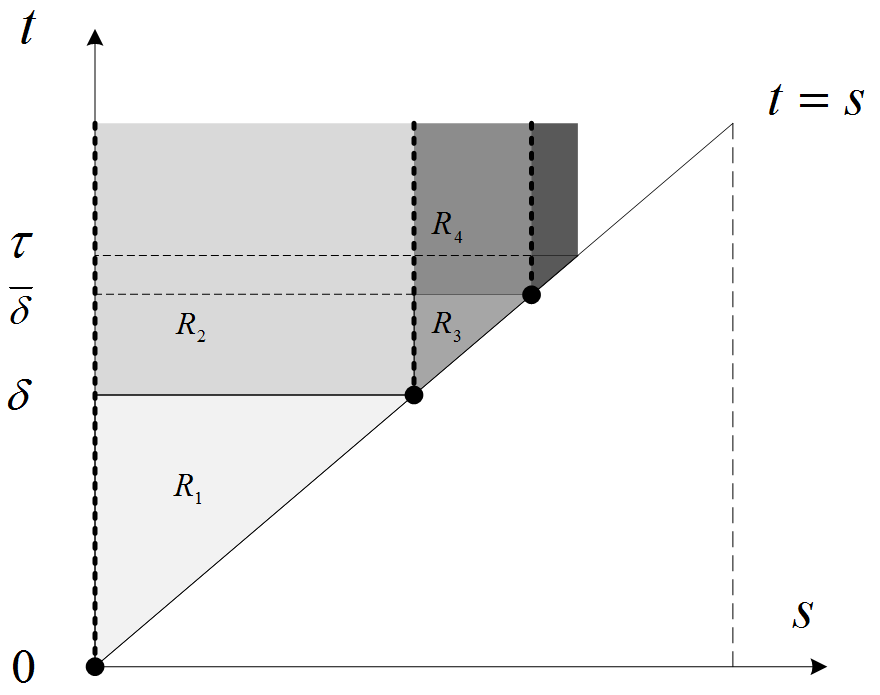}
		\caption{Extension from $\Delta[0,\delta]$ to a larger time interval}
		\label{fig:extension}
	\end{figure}
	Hence, we can determine the diagonal condition for $s\in[0,\delta]$. Then the nonlocal equations \eqref{Nonlocal linear PDE with t} and \eqref{Differentiate Nonlocal linear PDE with t with respect to t} reduce to classical PDEs with a parameter $t$. Therefore, we can extend uniquely our solution $(u,v)$ from $R_1$ to $R_1\cup R_2=\{0\le s\le \delta,~s\le t\le T\}$ in Figure \ref{fig:extension}. Subsequently, we acquire a new initial condition at $s=\delta$ for $t\in[s,T]$. Taking $\delta$ as an initial time and $(u(t,\delta,y),v(t,\delta,y))$ as initial datum, one can extend the solution to a larger time intervals $R_1\cup R_2\cup R_3$ and then $R_1\cup R_2\cup R_3\cup R_4$ as illustrated in Figure \ref{fig:extension}. Hence, we can extend uniquely the solution from $\Delta[0,\delta]$ to $\Delta[0,\overline{\delta}]$. The procedure could be repeated indefinitely, up to a \textit{maximally defined solution} $(u,v):\Delta[0,\sigma]\times\mathbb{R}^d\to\mathbb{R}$, belonging to $\Theta^{(2+\alpha)}_{[0,\sigma]}\times\Theta^{(2+\alpha)}_{[0,\sigma]}$ for any $\sigma<\tau$. The time region $\Delta[0,\tau]$ is maximal in the sense that if $\tau<\infty$, then there does not exist any solution of \eqref{Nonlocal linear PDE system with t} belonging to $\Theta^{(2+\alpha)}_{[0,\tau]}\times\Theta^{(2+\alpha)}_{[0,\tau]}$.
\end{remark} 

Next, for the nonlocal system \eqref{Nonlocal linear PDE system with t} of $(u,v)$ and a small enough $\delta\in[0,T]$, we give a Schauder-type estimate. It not only shows the stability of the solutions to \eqref{Nonlocal linear PDE with t} with respect to the data $(f,g)$, but also establishes a foundation for the further analysis of nonlocal fully nonlinear equations in the next section.

\begin{corollary}\label{Schauder estimates}
	Suppose that all coefficient functions and $f$ of \eqref{Nonlocal linear PDE with t} belong to $\Omega^{{(\alpha)}}_{[0,T]}$ and assume that $g\in\Omega^{{(2+\alpha)}}_{[0,T]}$. Then there exist $\delta>0$ and a unique $u\in\Omega^{{(2+\alpha)}}_{[0,\delta]}$ satisfying \eqref{Nonlocal linear PDE with t} in $\Delta[0,\delta]\times\mathbb{R}^d$. Furthermore, the following holds 
	\begin{equation} \label{Estimates of solutions of nonlocal linear PDE} 
		\begin{split}
			&\sup\limits_{t\in[0,\delta]}\left\{\left|u(t,\cdot,\cdot)\right|^{(2+\alpha)}_{[0,t]\times\mathbb{R}^d}+\left|u_t(t,\cdot,\cdot)\right|^{(2+\alpha)}_{[0,t]\times\mathbb{R}^d}\right\} \\
			&\leq c\left(\sup\limits_{t\in[0,\delta]}\left\{\left|f(t,\cdot,\cdot)\right|^{(\alpha)}_{[0,t]\times\mathbb{R}^d}+\left|f_t(t,\cdot,\cdot)\right|^{(\alpha)}_{[0,t]\times\mathbb{R}^d}\right\}+\sup\limits_{t\in[0,\delta]}\left\{\left|g(t,\cdot)\right|^{(2+\alpha)}_{\mathbb{R}^d}+\left|g_t(t,\cdot)\right|^{(2+\alpha)}_{\mathbb{R}^d}\right\}\right)
		\end{split}
	\end{equation}
	which implies that 
	\begin{equation} \label{Estimates of solutions of nonlocal PDEs} 
		\lVert u\rVert^{(2+\alpha)}_{[0,\delta]}\leq c\left(\lVert f\rVert^{(\alpha)}_{[0,\delta]}+\lVert g\rVert^{(2+\alpha)}_{[0,\delta]}\right). 
	\end{equation}
\end{corollary}
\begin{proof}
	As Lemma \ref{Equivalence between equation and system} shows, the first component $u$ of the unique solution $(u,v)$ of \eqref{Nonlocal linear PDE system with t} solves the nonlocal linear equation \eqref{Nonlocal linear PDE with t}. By noting of $v=u_t$ in \eqref{Nonlocal linear PDE system with t}, it is clear that $u\in\Omega^{(2+\alpha)}_{[0,\delta]}$ thanks to $(u,u_t)\in\Theta^{{(2+\alpha)}}_{[0,\delta]}\times\Theta^{{(2+\alpha)}}_{[0,\delta]}$. Moreover, by \eqref{Nonlocal linear PDE system with t}, we have 
	\begin{equation*}
		\begin{split}
			& \left|u(t,\cdot,\cdot)\right|^{(2+\alpha)}_{[0,t]\times\mathbb{R}^d} \\
			\leq &~ c\left(\delta^{1-\frac{\alpha}{2}}\sup\limits_{t\in[0,\delta]}\left|u_t(t,\cdot,\cdot)\right|^{(2+\alpha)}_{[0,t]\times\mathbb{R}^d}+\sup\limits_{t\in[0,\delta]}\left|f(t,\cdot,\cdot)\right|^{(\alpha)}_{[0,t]\times\mathbb{R}^d}+\sup\limits_{t\in[0,\delta]}\left|g(t,\cdot)\right|^{(2+\alpha)}_{\mathbb{R}^d}\right)
		\end{split}
	\end{equation*}
	and 
	\begin{equation*}
		\begin{split}
			\left|u_t(t,\cdot,\cdot)\right|^{(2+\alpha)}_{[0,t]\times\mathbb{R}^d}&\leq c\left(\left|u(t,\cdot,\cdot)\right|^{(2+\alpha)}_{[0,t]\times\mathbb{R}^d}+\delta^{1-\frac{\alpha}{2}}\sup\limits_{t\in[0,\delta]}\left|u_t(t,\cdot,\cdot)\right|^{(2+\alpha)}_{[0,t]\times\mathbb{R}^d}\right. \\
			&\qquad\left.+\sup\limits_{t\in[0,\delta]}\left|f_t(t,\cdot,\cdot)\right|^{(\alpha)}_{[0,t]\times\mathbb{R}^d}+\sup\limits_{t\in[0,\delta]}\left|g_t(t,\cdot)\right|^{(2+\alpha)}_{\mathbb{R}^d}\right).
		\end{split}
	\end{equation*}
	
	Consequently, for a small enough $\delta$, it holds 
	\begin{equation*}
		\begin{split}
			& \left|u(t,\cdot,\cdot)\right|^{(2+\alpha)}_{[0,t]\times\mathbb{R}^d}+\left|u_t(t,\cdot,\cdot)\right|^{(2+\alpha)}_{[0,t]\times\mathbb{R}^d}\leq  \sup\limits_{t\in[0,\delta]}\left|u(t,\cdot,\cdot)\right|^{(2+\alpha)}_{[0,t]\times\mathbb{R}^d}+\sup\limits_{t\in[0,\delta]}\left|u_t(t,\cdot,\cdot)\right|^{(2+\alpha)}_{[0,t]\times\mathbb{R}^d} \\
			\leq &~ c\left(\sup\limits_{t\in[0,\delta]}\left\{\left|f(t,\cdot,\cdot)\right|^{(\alpha)}_{[0,t]\times\mathbb{R}^d}+\left|f_t(t,\cdot,\cdot)\right|^{(\alpha)}_{[0,t]\times\mathbb{R}^d}\right\}+\sup\limits_{t\in[0,\delta]}\left\{\left|g(t,\cdot)\right|^{(2+\alpha)}_{\mathbb{R}^d}+\left|g_t(t,\cdot)\right|^{(2+\alpha)}_{\mathbb{R}^d}\right\}\right)
		\end{split}
	\end{equation*} 
	which leads to \eqref{Estimates of solutions of nonlocal linear PDE} and \eqref{Estimates of solutions of nonlocal PDEs}.
\end{proof}

\begin{remark}[\textbf{Stability of solutions with respect to data}]
	The inequalities in Corollary \ref{Schauder estimates} can be used to show that the map from data $(f,g)$ to solutions of \eqref{Nonlocal linear PDE with t} is continuous in the $\lVert\cdot\rVert^{(2+\alpha)}_{[0,\delta]}$-topology. Specifically, let $u$ and $\widehat{u}$ correspond to $(f,g)$ and $(\widehat{f},\widehat{g})$ satisfying these assumptions of Theorem \ref{Well-posedness of (u,u_t)}, respectively. Then, we have 
	\begin{equation*}
		\lVert u-\widehat{u}\rVert^{(2+\alpha)}_{[0,\delta]}\leq c\left(\lVert f-\widehat{f}\rVert^{(\alpha)}_{[0,\delta]}+\lVert g-\widehat{g}\rVert^{(2+\alpha)}_{[0,\delta]}\right). 
	\end{equation*}
\end{remark}

\ \ 


\section{Nonlocal Fully Nonlinear Parabolic PDEs} \label{sec:nonlinear}
After studying nonlocal linear PDEs \eqref{Nonlocal linear PDE with t} and acquiring a prior estimate of solutions \eqref{Estimates of solutions of nonlocal PDEs}, we further study nonlocal fully nonlinear PDEs: 
\begin{equation} \label{Nonlocal fully nonlinear equation in Section 3}  
	\left\{
	\begin{array}{rcl}
		u_s(t,s,y) & = & F\big(t,s,y,u(t,s,y),u_y(t,s,y),u_{yy}(t,s,y), \\
		& & \qquad \qquad u(s,s,y),u_y(s,s,y),u_{yy}(s,s,y)\big), \\
		u(t,0,y) & = & g(t,y),\quad 0\leq s\leq t\leq T,\quad y\in\mathbb{R}^d.
	\end{array}
	\right.
\end{equation} 
where the mapping $F$ is nonlinear with respect to all its arguments, $s$ and $y$ are temporal and spatial variables, respectively, while $t$ could be considered as an external temporal parameter, and the temporal variables $(s,t)$ are defined in a triangular region $\Delta[0,T]$.

To leverage the results of nonlocal linear PDEs in Section \ref{sec:linear}, we require certain regularity assumptions on $F$ and $g$. Generally speaking, we need that the initial condition $g\in\Omega^{(2+\alpha)}_{[0,T]}$ and the nonlinear $F$  maps the functions (about $u$) from $\Omega^{(2+\alpha)}_{[0,T]}$ to $\Omega^{(\alpha)}_{[0,T]}$. More specifically, we suppose that the real-valued function $\pi\mapsto F(\pi)$ is defined in $\Pi=\Delta[0,T]\times\mathbb{R}^d\times\mathbb{R}\times\mathbb{R}^d\times\mathbb{S}^d\times\mathbb{R}\times\mathbb{R}^d\times\mathbb{S}^d$ and satisfies that
\begin{enumerate}[label=(\roman*)]
	\item \textbf{(Ellipticity condition)} for any $\pi=(t,s,y,u,p,q,l,m,n)\in\Pi$ and $\xi=(\xi_1,\xi_2,\dots,\xi_n)\in\mathbb{R}^n\setminus\{0\}$,   
	\begin{eqnarray}
		\sum^d_{i,j=1}\frac{\partial F}{\partial q_{ij}}(t,s,y,u,p,q,l,m,n)\xi_i\xi_j & > & 0, \label{Uniform ellipticity condition of F 1} \\
		\sum^d_{i,j=1}\left(\frac{\partial F}{\partial q_{ij}}+\frac{\partial F}{\partial n_{ij}}\right)(t,s,y,u,p,q,l,m,n)\xi_i\xi_j & > & 0; \label{Uniform ellipticity condition of F 2} 
	\end{eqnarray} 
	\item \textbf{(H\"{o}lder continuity)} for every $\delta\geq 0$ and $\overline{\pi}=(u,p,q,l,m,n)$, there exists a positive constant $K>0$ such that
	\begin{equation} \label{Holder continuity of F}
		\sup_{(t,\overline{\pi})}\left\{\left| \mathcal{F}(t,\cdot,\cdot,\overline{\pi})\right|^{(\alpha)}_{[0,\delta]\times\mathbb{R}^d}\right\}=K;
	\end{equation}
	\item \textbf{(Lipschitz continuity)} for any $(t_1,s,y,\overline{\pi}_1),(t_2,s,y,\overline{\pi}_2)\in\Pi$, there exists a positive constant $L>0$ such that
	\begin{equation} \label{Lipschitz continuity of F} 
		|\mathcal{F}(t_1,s,y,\overline{\pi}_1)-\mathcal{F}(t_2,s,y,\overline{\pi}_2)|\leq L\left(|t_1-t_2|+|\overline{\pi}_1-\overline{\pi}_2|\right), 
	\end{equation}
\end{enumerate}
where $\mathcal{F}$ represents $F$ itself and some of its first- and second-order derivatives ($F_\mathcal{X}$ and $F_{\mathcal{X}\mathcal{Y}}$), which are denoted by ``$\surd$" in Table \ref{tab:table1} and Table \ref{tab:table2}. 

\begin{table}[!ht] 
	\centering
	\begin{tabular}{c| c c c c c c c c c}
		\hline
		$\mathcal{X}$ & $t$ & $s$ & $y$ & $u$ & $p$ & $q$ & $l$ & $m$ & $n$ \\ 
		\hline 
		$F_\mathcal{X}$ & $\surd$ & & & $\surd$ & $\surd$ & $\surd$ & $\surd$ & $\surd$ & $\surd$ \\ 
		\hline 
	\end{tabular}
	\caption{First-order derivatives of $F$ required to be H\"{o}lder and Lipschitz continuous}
	\label{tab:table1}
\end{table} 

\begin{table}[!ht] 
	\centering
	\begin{tabular}{c| c c c c c c c c c}
		\hline
		\diagbox{$\mathcal{X}$}{$F_{\mathcal{X}\mathcal{Y}}$}{$\mathcal{Y}$} & $t$ & $s$ & $y$ & $u$ & $p$ & $q$ & $l$ & $m$ & $n$\\ 
		\hline 
		$t$ &  &  &  & $\surd$ & $\surd$ & $\surd$ & $\surd$ & $\surd$ & $\surd$ \\ 
		$s$ &  &  &  &  &  &  &  &  & \\
		$y$ &  &  &  &  &  &  &  &  & \\
		$u$ & $\surd$ &  &  & $\surd$ & $\surd$ & $\surd$ & $\surd$ & $\surd$ & $\surd$ \\
		$p$ & $\surd$ &  &  & $\surd$ & $\surd$ & $\surd$ & $\surd$ & $\surd$ & $\surd$ \\
		$q$ & $\surd$ &  &  & $\surd$ & $\surd$ & $\surd$ & $\surd$ & $\surd$ & $\surd$ \\
		$l$ & $\surd$ &  &  & $\surd$ & $\surd$ & $\surd$ &  &  &  \\
		$m$ & $\surd$ &  &  & $\surd$ & $\surd$ & $\surd$ &  &  &  \\
		$n$ & $\surd$ &  &  & $\surd$ & $\surd$ & $\surd$ &  &  &  \\
		\hline 
	\end{tabular}
	\caption{Second-order derivatives of $F$ required to be H\"{o}lder and Lipschitz continuous}
	\label{tab:table2}
\end{table} 

Before we present our main result, we stress that the standard linearization methods are not applicable for the nonlocal case. In the setting of local parabolic PDEs, \cite{Eidelman1969} introduced a so-called ``quasi-linearization method" and studied local existence for fully nonlinear parabolic problems by transforming fully nonlinear equations into quasi-linear systems. Furthermore, \cite{Khudyaev1963,Sopolov1970} utilized a variant of this method to investigate fully nonlinear PDEs. The linearization method, which we have used in Theorem \ref{Well-posedness of fully nonlinear equations} below, is substantially inspired by \cite{Kruzhkov1975,Lunardi1995}. Although there are some previous works on how to linearize nonlinear equations, it is still difficult to extend their methods from local to nonlocal cases, mainly due to the limited progress on studying the nonlocal linear PDEs \eqref{Nonlocal linear PDE with t}.

Next, we prove the well-posedness of nonlocal fully nonlinear PDEs \eqref{Nonlocal fully nonlinear equation in Section 3}.

\begin{theorem}\label{Well-posedness of fully nonlinear equations}
	Suppose that \eqref{Uniform ellipticity condition of F 1}-\eqref{Lipschitz continuity of F} hold and assume that $g\in\Omega^{{(2+\alpha)}}_{[0,T]}$. Then, there exist $\delta>0$ and a unique $u\in\Omega^{{(2+\alpha)}}_{[0,\delta]}$ satisfying \eqref{Nonlocal fully nonlinear equation in Section 3} in $\Delta[0,\delta]\times\mathbb{R}^d$. 
\end{theorem} 

\begin{proof}
	We prove the theorem again with fixed-point arguments. \\
	
	\noindent (\textbf{Existence}) The solution of \eqref{Nonlocal fully nonlinear equation in Section 3} is found via a fixed point of the operator $\Lambda$ defined in the set
	\begin{equation*}
		\mathcal{U}=\left\{u\in\Omega^{(2+\alpha)}_{[0,\delta]}:u(t,0,y)=g(t,y),\lVert u-g\rVert^{(2+\alpha)}_{[0,\delta]}\leq R\right\} 
	\end{equation*} 
	and $\Lambda(u)=U$, where $U$ is the solution of  
	\begin{equation} \label{Mapping from u to U}
		\left\{
		\begin{array}{lr}
			U_s(t,s,y)=\mathcal{L}U+F\big(t,s,y,u(t,s,y),u_y(t,s,y),u_{yy}(t,s,y), \\
			\qquad\qquad\qquad\qquad\qquad\qquad  u(s,s,y),u_y(s,s,y),u_{yy}(s,s,y)\big)-\mathcal{L}u, \\
			U(t,0,y)=g(t,y),\quad 0\leq s\leq t\leq \delta,\quad y\in\mathbb{R}^d.
		\end{array}
		\right.
	\end{equation}
	where $\mathcal{L}$ is a linear operator given by
	\begin{equation} \label{Nonlocal linear differential operator L} 
		\begin{split}
			\mathcal{L}u=&\sum^d_{i,j=1}F_{q_{ij}}\frac{\partial^2u}{\partial y_i\partial y_j}+\sum^d_{i=1}F_{p_{i}}\frac{\partial u}{\partial y_i}+F_u u \\
			&+\sum^d_{i,j=1}F_{n_{ij}}\left(\left.\frac{\partial^2u}{\partial y_i\partial y_j}\right|_{t=s}\right)+\sum^d_{i=1}F_{m_i}\left(\left.\frac{\partial u}{\partial y_i}\right|_{t=s}\right)+F_l \Big(\left.u\right|_{t=s}\Big)
		\end{split}
	\end{equation}
	with the partial derivatives $F_u$, $F_p$, $F_q$, $F_l$, $F_m$ and $F_n$ of $F(t,s,y,u,p,q,l,m,n)$ evaluated at the point $\big(t,0,y,g(t,y),g_y(t,y),g_{yy}(t,y),g(0,y),g_y(0,y),g_{yy}(0,y)\big)$. It is clear that all coefficients of $\mathcal{L}$ belong to $\Omega^{(\alpha)}_{[0,\delta]}$ and $F-\mathcal{L}u$ in \eqref{Mapping from u to U} also belongs to $\Omega^{(\alpha)}_{[0,\delta]}$. Hence, by the well-posedness result of nonlocal linear PDEs in Section \ref{sec:linearresult}, the mapping $\Lambda(u)$ from $\Omega^{{(2+\alpha)}}_{[0,\delta]}$ to $\Omega^{{(2+\alpha)}}_{[0,\delta]}$ is well-defined. 
	
	Generally speaking, we shall show that for every $u,\widehat{u}\in\mathcal{U}$,
	\begin{equation*}
		\lVert\Lambda(u)-\Lambda(\widehat{u})\rVert^{(2+\alpha)}_{[0,\delta]}\leq C(R)\delta^{\frac{\alpha}{2}}\lVert u-\widehat{u}\rVert^{(2+\alpha)}_{[0,\delta]}.
	\end{equation*}
	
	Let $u,\widehat{u}\in\mathcal{U}$. Then $w=\Lambda(u)-\Lambda(\widehat{u})$ satisfies 
	\begin{equation*}
		\left\{
		\begin{array}{rcl}
			w_s(t,s,y)& = & \mathcal{L}w+F\big(t,s,y,u(t,s,y),u_y(t,s,y),u_{yy}(t,s,y), \\
			&& \qquad \qquad \qquad \quad u(s,s,y),u_y(s,s,y),u_{yy}(s,s,y)\big)\\
			&& \qquad-F\big(t,s,y,\widehat{u}(t,s,y),\widehat{u}_y(t,s,y),\widehat{u}_{yy}(t,s,y), \\
			&& \qquad \qquad \qquad \quad \widehat{u}(s,s,y),\widehat{u}_y(s,s,y),\widehat{u}_{yy}(s,s,y)\big)-\mathcal{L}(u-\widehat{u}), \\
			w(t,0,y) & = & 0,\quad 0\leq s\leq t\leq \delta,\quad y\in\mathbb{R}^d.
		\end{array}
		\right.
	\end{equation*}
	
	By \eqref{Estimates of solutions of nonlocal linear PDE} and \eqref{Estimates of solutions of nonlocal PDEs}, there is $c>0$, independent of $\delta$, such that 
	\begin{equation*}  
		\begin{split}
			&\sup\limits_{t\in[0,\delta]}\left\{\left|w(t,\cdot,\cdot)\right|^{(2+\alpha)}_{[0,t]\times\mathbb{R}^d}+\left|w_t(t,\cdot,\cdot)\right|^{(2+\alpha)}_{[0,t]\times\mathbb{R}^d}\right\} \\
			\leq &~ c\sup\limits_{t\in[0,\delta]}\left\{\left|\varphi(t,\cdot,\cdot)\right|^{(\alpha)}_{[0,t]\times\mathbb{R}^d}+\left|\varphi_t(t,\cdot,\cdot)\right|^{(\alpha)}_{[0,t]\times\mathbb{R}^d}\right\},
		\end{split}
	\end{equation*} 
	where 
	\begin{eqnarray*}
		\varphi(t,s,y) & = & F\big(t,s,y,u(t,s,y),u_y(t,s,y),u_{yy}(t,s,y), \\
		&& \qquad \qquad u(s,s,y),u_y(s,s,y),u_{yy}(s,s,y)\big)\\
		&&-F\big(t,s,y,\widehat{u}(t,s,y),\widehat{u}_y(t,s,y),\widehat{u}_{yy}(t,s,y), \\
		&& \qquad \qquad \quad \widehat{u}(s,s,y),\widehat{u}_y(s,s,y),\widehat{u}_{yy}(s,s,y)\big)-\mathcal{L}(u-\widehat{u}).
	\end{eqnarray*}
	
	In order to estimate $|\varphi(t,\cdot,\cdot)|^{(\alpha)}_{[0,t]\times\mathbb{R}^d}$ and $|\varphi_t(t,\cdot,\cdot)|^{(\alpha)}_{[0,t]\times\mathbb{R}^d}$, it is convenient to write $\varphi(t,s,y)$ as an integral representation 
	\begin{align}
		&\varphi(t,s,y) \label{varphi}  \\
		=&\int^1_0\left(F_u(t,s,y,\theta_\sigma(t,s,y))-F_u(t,0,y,\theta_0(t,y))\right)(u(t,s,y)-\widehat{u}(t,s,y))d\sigma \cr
		&+\int^1_0\sum^d_{i=1}\left(F_{p_i}(t,s,y,\theta_\sigma(t,s,y))-F_{p_i}(t,0,y,\theta_0(t,y))\right)(u_{y_i}(t,s,y)-\widehat{u}_{y_i}(t,s,y))d\sigma \cr
		&+\int^1_0\sum^d_{i,j=1}\left(F_{q_{ij}}(t,s,y,\theta_\sigma(t,s,y))-F_{q_{ij}}(t,0,y,\theta_0(t,y))\right)(u_{y_iy_j}(t,s,y)-\widehat{u}_{y_iy_j}(t,s,y))d\sigma \cr
		&+\int^1_0\left(F_l(t,s,y,\theta_\sigma(t,s,y))-F_l(t,0,y,\theta_0(t,y))\right)(u(s,s,y)-\widehat{u}(s,s,y))d\sigma \cr
		&+\int^1_0\sum^d_{i=1}\left(F_{m_i}(t,s,y,\theta_\sigma(t,s,y))-F_{m_i}(t,0,y,\theta_0(t,y))\right)(u_{y_i}(s,s,y)-\widehat{u}_{y_i}(s,s,y))d\sigma \cr
		&+\int^1_0\sum^d_{i,j=1}\left(F_{n_{ij}}(t,s,y,\theta_\sigma(t,s,y))-F_{n_{ij}}(t,0,y,\theta_0(t,y))\right)(u_{y_iy_j}(s,s,y)-\widehat{u}_{y_iy_j}(s,s,y))d\sigma \nonumber
	\end{align}
	with $\theta_0(t,y)=\big(g(t,y),g_y(t,y),g_{yy}(t,y),g(0,y),g_y(0,y),g_{yy}(0,y)\big)$ and
	\begin{eqnarray*}
		\theta_\sigma(t,s,y)&=&\sigma\big(u(t,s,y),u_y(t,s,y),u_{yy}(t,s,y),u(s,s,y),u_y(s,s,y),u_{yy}(s,s,y)\big) \\
		&&+(1-\sigma)\big(\widehat{u}(t,s,y),\widehat{u}_y(t,s,y),\widehat{u}_{yy}(t,s,y),\widehat{u}(s,s,y),\widehat{u}_y(s,s,y),\widehat{u}_{yy}(s,s,y)\big). \\
	\end{eqnarray*} 
	
	\noindent (\textbf{Estimates of $|\varphi(t,\cdot,\cdot)|^{(\alpha)}_{[0,t]\times\mathbb{R}^d}$}) We first consider $|\varphi(t,s,y)-\varphi(t,s^\prime,y)|$ for any $0\leq s\leq s^\prime\leq t\leq\delta\leq T$ and $y\in\mathbb{R}^d$. It is convenient to add and subtract the following to $\varphi(t,s,y)-\varphi(t,s^\prime,y)$:
	\begin{equation*} 
		\begin{split}
			&\int^1_0 F_u(t,s^\prime,y,\theta_\sigma(t,s^\prime,y))\times(u(t,s,y)-\widehat{u}(t,s,y))d\sigma \\
			+&\int^1_0\sum^d_{i=1} F_p(t,s^\prime,y,\theta_\sigma(t,s^\prime,y))\times(u_{y_i}(t,s,y)-\widehat{u}_{y_i}(t,s,y))d\sigma \\
			+&\int^1_0\sum^d_{i,j=1} F_{q_{ij}}(t,s^\prime,y,\theta_\sigma(t,s^\prime,y))\times(u_{y_iy_j}(t,s,y)-\widehat{u}_{y_iy_j}(t,s,y))d\sigma \\
			+&\int^1_0 F_l(t,s^\prime,y,\theta_\sigma(t,s^\prime,y))\times(u(s,s,y)-\widehat{u}(s,s,y))d\sigma \\
			+&\int^1_0\sum^d_{i=1} F_{m_i}(t,s^\prime,y,\theta_\sigma(t,s^\prime,y))\times(u_{y_i}(s,s,y)-\widehat{u}_{y_i}(s,s,y))d\sigma \\
			+&\int^1_0\sum^d_{i,j=1} F_{n_{ij}}(t,s^\prime,y,\theta_\sigma(t,s^\prime,y))\times(u_{y_iy_j}(s,s,y)-\widehat{u}_{y_iy_j}(s,s,y))d\sigma     
		\end{split}
	\end{equation*}
	Let $F_\chi$ be any first order derivatives of $F$ with respect to $(u,p,q,l,m,n)$. We have 
	\begin{align*}
		&\left|F_\chi(t,s,y,\theta_\sigma(t,s,y))-F_\chi(t,s^\prime,y,\theta_\sigma(t,s^\prime,y))\right| \leq L\left|\theta_\sigma(t,s,y)-\theta_\sigma(t,s^\prime,y)\right|+K(s^\prime-s)^\frac{\alpha}{2} \\
		\leq &~L\left(\sum^d\limits_{i,j=1}\left|u_{y_iy_j}(t,s,y)-u_{y_iy_j}(t,s^\prime,y)\right|+\sum^d_{i=1}\left|u_{y_i}(t,s,y)-u_{y_i}(t,s^\prime,y)\right|\right.\\
		&\qquad+\left|u(t,s,y)-u(t,s^\prime,y)\right|+\sum^d\limits_{i,j=1}\left|u_{y_iy_j}(s,s,y)-u_{y_iy_j}(s^\prime,s^\prime,y)\right| \\
		&\qquad \left.+\sum^d_{i=1}\left|u_{y_i}(s,s,y)-u_{y_i}(s^\prime,s^\prime,y)\right|+\left|u(s,s,y)-u(s^\prime,s^\prime,y)\right|\right) \\
		&+L\left(\sum^d\limits_{i,j=1}\left|\widehat{u}_{y_iy_j}(t,s,y)-\widehat{u}_{y_iy_j}(t,s^\prime,y)\right|+\sum^d_{i=1}\left|\widehat{u}_{y_i}(t,s,y)-\widehat{u}_{y_i}(t,s^\prime,y)\right|\right.\\
		&\qquad+\left|\widehat{u}(t,s,y)-\widehat{u}(t,s^\prime,y)\right|+\sum^d\limits_{i,j=1}\left|\widehat{u}_{y_iy_j}(s,s,y)-\widehat{u}_{y_iy_j}(s^\prime,s^\prime,y)\right| \\
		&\qquad \left.+\sum^d_{i=1}\left|\widehat{u}_{y_i}(s,s,y)-\widehat{u}_{y_i}(s^\prime,s^\prime,y)\right|+\left|\widehat{u}(s,s,y)-\widehat{u}(s^\prime,s^\prime,y)\right|\right)+K(s^\prime-s)^\frac{\alpha}{2}
		\\
		\leq &~L\left(|u(t,\cdot,\cdot)|^{(2+\alpha)}_{[0,t]\times\mathbb{R}^d}(s^\prime-s)^\frac{\alpha}{2}+\sup\limits_{\overline{s}\in(s,s^\prime)}|u_t(\overline{s},\cdot,\cdot)|^{(2+\alpha)}_{[0,\overline{s}]\times\mathbb{R}^d}(s^\prime-s)^1\right. \\
		&\qquad+|u(s^\prime,\cdot,\cdot)|^{(2+\alpha)}_{[0,s^\prime]\times\mathbb{R}^d}(s^\prime-s)^\frac{\alpha}{2} +|\widehat{u}(t,\cdot,\cdot)|^{(2+\alpha)}_{[0,t]\times\mathbb{R}^d}(s^\prime-s)^\frac{\alpha}{2} \\
		&\left.\qquad+\sup\limits_{\overline{s}\in(s,s^\prime)}|\widehat{u}_t(\overline{s},\cdot,\cdot)|^{(2+\alpha)}_{[0,\overline{s}]\times\mathbb{R}^d}(s^\prime-s)^1+|\widehat{u}(s^\prime,\cdot,\cdot)|^{(2+\alpha)}_{[0,s^\prime]\times\mathbb{R}^d}(s^\prime-s)^\frac{\alpha}{2}\right)+K(s^\prime-s)^\frac{\alpha}{2} \\
		\leq &\left(K+L\left(\lVert u\rVert^{(2+\alpha)}_{[0,\delta]}+\lVert \widehat{u}\rVert^{(2+\alpha)}_{[0,\delta]}\right)\right)(s^\prime-s)^\frac{\alpha}{2} \leq C_1(R)(s^\prime-s)^\frac{\alpha}{2}
	\end{align*}
	and 
	\begin{align*}
		&|F_\chi(t,s^\prime,y,\theta_\sigma(t,s^\prime,y))-F_\chi(t,0,y,\theta_0(t,y))| \le L\left|\theta_\sigma(t,s^\prime,y)-\theta_0(t,y)\right|+K(s^\prime-0)^\frac{\alpha}{2} \\
		\leq &~ L\left(\sum^d\limits_{i,j=1}\left|u_{y_iy_j}(t,s^\prime,y)-g_{y_iy_j}(t,y)\right|+\sum^d_{i=1}\left|u_{y_i}(t,s^\prime,y)-g_{y_i}(t,y)\right|\right.+\left|u(t,s^\prime,y)-g(t,y)\right|\\
		&\qquad+\sum^d\limits_{i,j=1}\left|u_{y_iy_j}(s^\prime,s^\prime,y)-g_{y_iy_j}(0,y)\right|\left.+\sum^d_{i=1}\left|u_{y_i}(s^\prime,s^\prime,y)-g_{y_i}(0,y)\right|+\left|u(s^\prime,s^\prime,y)-g(0,y)\right|\right) \\
		& +L\left(\sum^d\limits_{i,j=1}\left|\widehat{u}_{y_iy_j}(t,s^\prime,y)-g_{y_iy_j}(t,y)\right|+\sum^d_{i=1}\left|\widehat{u}_{y_i}(t,s^\prime,y)-g_{y_i}(t,y)\right|\right.+\left|\widehat{u}(t,s^\prime,y)-g(t,y)\right| \\
		&\qquad+\sum^d\limits_{i,j=1}\left|\widehat{u}_{y_iy_j}(s^\prime,s^\prime,y)-g_{y_iy_j}(0,y)\right| \left.+\sum^d_{i=1}\left|\widehat{u}_{y_i}(s^\prime,s^\prime,y)-g_{y_i}(0,y)\right|+\left|\widehat{u}(s^\prime,s^\prime,y)-g(0,y)\right|\right)\\
		&+K(s^\prime-0)^\frac{\alpha}{2} \\
	\end{align*}
	\begin{align*}
		\leq &~ L\left(|\left(u-g\right)(t,\cdot,\cdot)|^{(2+\alpha)}_{[0,t]\times\mathbb{R}^d}(s^\prime-0)^\frac{\alpha}{2}+\sup\limits_{\overline{s}\in(0,s^\prime)}|g_t(\overline{s},\cdot)|^{(2+\alpha)}_{\mathbb{R}^d}(s^\prime-0)^1\right. \\
		&\qquad+|u(s^\prime,\cdot,\cdot)|^{(2+\alpha)}_{[0,s^\prime]\times\mathbb{R}^d}(s^\prime-0)^\frac{\alpha}{2} +|\left(\widehat{u}-g\right)(t,\cdot,\cdot)|^{(2+\alpha)}_{[0,t]\times\mathbb{R}^d}(s^\prime-0)^\frac{\alpha}{2} \\
		&\left.\qquad+\sup\limits_{\overline{s}\in(0,s^\prime)}|g_t(\overline{s},\cdot)|^{(2+\alpha)}_{\mathbb{R}^d}(s^\prime-0)^1+|\widehat{u}(s^\prime,\cdot,\cdot)|^{(2+\alpha)}_{[0,s^\prime]\times\mathbb{R}^d}(s^\prime-0)^\frac{\alpha}{2}\right)+K(s^\prime-0)^\frac{\alpha}{2} \\
		\leq &~ \left(K+L\left(\lVert u-g\rVert^{(2+\alpha)}_{[0,\delta]}+\lVert \widehat{u}-g\rVert^{(2+\alpha)}_{[0,\delta]}+\lVert g\rVert^{(2+\alpha)}_{[0,\delta]}\right)\right)(s^\prime-0)^\frac{\alpha}{2} \leq C_2(R)\delta^\frac{\alpha}{2} 
	\end{align*}
	where $L>0$ is a constant which could vary from line to line. 
	
	Since $u$ and $\widehat{u}$ satisfy the same initial condition $g$, we have  
	\begin{equation} \label{Holder continuity of varphi with respect to s}
		\begin{split}
			&|\varphi(t,s,y)-\varphi(t,s^\prime,y)| \\
			\leq &~ C_1(R)(s^\prime-s)^\frac{\alpha}{2}\Bigg(\delta^1\left|(u-\widehat{u})_s(t,\cdot,\cdot)\right|^{(0)}_{[0,t]\times\mathbb{R}^d}+\delta^\frac{1+\alpha}{2}\sum^d\limits_{i=1}\left|(u-\widehat{u})_{y_i}(t,\cdot,\cdot)\right|^{(1+\alpha)}_{[0,t]\times\mathbb{R}^d} \\
			&\qquad+\delta^\frac{\alpha}{2}\sum^d\limits_{i,j=1}\left|(u-\widehat{u})_{y_iy_j}(t,\cdot,\cdot)\right|^{(\alpha)}_{[0,t]\times\mathbb{R}^d}+\delta^1\left|(u-\widehat{u})_s(s,\cdot,\cdot)\right|^{(0)}_{[0,s]\times\mathbb{R}^d} \\
			&\qquad+\delta^\frac{1+\alpha}{2}\sum^d\limits_{i=1}\left|(u-\widehat{u})_{y_i}(s,\cdot,\cdot)\right|^{(1+\alpha)}_{[0,s]\times\mathbb{R}^d}+\delta^\frac{\alpha}{2}\sum^d\limits_{i,j=1}\left|(u-\widehat{u})_{y_iy_j}(s,\cdot,\cdot)\right|^{(\alpha)}_{[0,s]\times\mathbb{R}^d}\Bigg) \\
			&+C_2(R)\delta^\frac{\alpha}{2}\Bigg((s^\prime-s)^1\left|(u-\widehat{u})_s(t,\cdot,\cdot)\right|^{(0)}_{[0,t]\times\mathbb{R}^d}+(s^\prime-s)^\frac{1+\alpha}{2}\sum^d\limits_{i=1}\left|(u-\widehat{u})_{y_i}(t,\cdot,\cdot)\right|^{(1+\alpha)}_{[0,t]\times\mathbb{R}^d} \\
			&\qquad+(s^\prime-s)^\frac{\alpha}{2}\sum^d\limits_{i,j=1}\left|(u-\widehat{u})_{y_iy_j}(t,\cdot,\cdot)\right|^{(\alpha)}_{[0,t]\times\mathbb{R}^d} \\
			&\qquad+(s^\prime-s)^1\sup\limits_{\overline{s}\in(s,s^\prime)}\left|(u-\widehat{u})_t(\overline{s},\cdot,\cdot)\right|^{(0)}_{[0,\overline{s}]\times\mathbb{R}^d}+(s^\prime-s)^1\left|(u-\widehat{u})_s(s^\prime,\cdot,\cdot)\right|^{(0)}_{[0,s^\prime]\times\mathbb{R}^d} \\
			&\qquad+(s^\prime-s)^1\sum^d\limits_{i=1}\sup\limits_{\overline{s}\in(s,s^\prime)}\left|(u-\widehat{u})_{ty_i}(\overline{s},\cdot,\cdot)\right|^{(0)}_{[0,\overline{s}]\times\mathbb{R}^d}+(s^\prime-s)^\frac{1+\alpha}{2}\sum^d_{i=1}\left|(u-\widehat{u})_{y_i}(s^\prime,\cdot,\cdot)\right|^{(1+\alpha)}_{[0,s^\prime]\times\mathbb{R}^d} \\
			&\qquad+(s^\prime-s)^1\sum^d\limits_{i,j=1}\sup\limits_{\overline{s}\in(s,s^\prime)}\left|(u-\widehat{u})_{ty_iy_j}(\overline{s},\cdot,\cdot)\right|^{(0)}_{[0,\overline{s}]\times\mathbb{R}^d}+(s^\prime-s)^\frac{\alpha}{2}\sum^d\limits_{i,j=1}\left|(u-\widehat{u})_{y_iy_j}(s^\prime,\cdot,\cdot)\right|^{(\alpha)}_{[0,s^\prime]\times\mathbb{R}^d}\Bigg) \\
			\leq &~ C_3(R)\delta^{\frac{\alpha}{2}}(s^\prime-s)^{\frac{\alpha}{2}}\lVert u-\widehat{u}\rVert^{(2+\alpha)}_{[0,\delta]}, 
		\end{split}
	\end{equation}
	which also implies that
	\begin{equation} \label{Boundness of varphi} 
		|\varphi(t,\cdot,\cdot)|^{\infty}_{[0,t]\times\mathbb{R}^d}\leq C_3(R)\delta^\alpha\lVert u-\widehat{u}\rVert^{(2+\alpha)}_{[0,\delta]}, 
	\end{equation}
	since $\varphi(t,0,y)\equiv 0$.
	
	To estimate $|\varphi(t,s,y)-\varphi(t,s,y^\prime)|$ it is convenient to add and subtract the following to $\varphi(t,s,y)-\varphi(t,s,y^\prime)$:
	\begin{equation*}
		\begin{split}
			&\int^1_0\left(F_u(t,s,y^\prime,\theta_\sigma(t,s,y^\prime))-F_u(t,0,y^\prime,\theta_0(t,y^\prime))\right)\times(u(t,s,y)-\widehat{u}(t,s,y))d\sigma \\
			+&\int^1_0\sum^d_{i=1}\left(F_{p_i}(t,s,y^\prime,\theta_\sigma(t,s,y^\prime))-F_{p_i}(t,0,y^\prime,\theta_0(t,y^\prime))\right)\times(u_{y_i}(t,s,y)-\widehat{u}_{y_i}(t,s,y))d\sigma \\
			+&\int^1_0\sum^d_{i,j=1}\left(F_{q_{ij}}(t,s,y^\prime,\theta_\sigma(t,s,y^\prime))-F_{q_{ij}}(t,0,y^\prime,\theta_0(t,y^\prime))\right)\times(u_{y_iy_j}(t,s,y)-\widehat{u}_{y_iy_j}(t,s,y))d\sigma \\
			+&\int^1_0\left(F_l(t,s,y^\prime,\theta_\sigma(t,s,y^\prime))-F_l(t,0,y^\prime,\theta_0(t,y^\prime))\right)\times(u(s,s,y)-\widehat{u}(s,s,y))d\sigma \\
			+&\int^1_0\sum^d_{i=1}\left(F_{m_i}(t,s,y^\prime,\theta_\sigma(t,s,y^\prime))-F_{m_i}(t,0,y^\prime,\theta_0(t,y^\prime))\right)\times(u_{y_i}(s,s,y)-\widehat{u}_{y_i}(s,s,y))d\sigma \\
			+&\int^1_0\sum^d_{i,j=1}\left(F_{n_{ij}}(t,s,y^\prime,\theta_\sigma(t,s,y^\prime))-F_{n_{ij}}(t,0,y^\prime,\theta_0(t,y^\prime))\right)\times(u_{y_iy_j}(s,s,y)-\widehat{u}_{y_iy_j}(s,s,y))d\sigma
		\end{split}
	\end{equation*}
	
	For any first order derivatives $F_\chi$ of $F$ with respect to $(u,p,q,l,m,n)$, we have
	\begin{align*}
		&|F_\chi(t,s,y,\theta_\sigma(t,s,y))-F_\chi(t,s,y^\prime,\theta_\sigma(t,s,y^\prime))|+|F_\chi(t,0,y,\theta_0(t,y))-F_\chi(t,0,y^\prime,\theta_0(t,y^\prime))| \\ 
		\leq &~ L\left|\theta_\sigma(t,s,y)-\theta_\sigma(t,s,y^\prime)\right|+L\left|\theta_0(t,y)-\theta_0(t,y^\prime)\right|+2K|y-y^\prime|^\alpha \\
		\leq &~ L\left(\sum^d\limits_{i,j=1}\left|u_{y_iy_j}(t,s,y)-u_{y_iy_j}(t,s,y^\prime)\right|+\sum^d_{i=1}\left|u_{y_i}(t,s,y)-u_{y_i}(t,s,y^\prime)\right|\right.+\left|u(t,s,y)-u(t,s,y^\prime)\right| \\
		&\quad+\sum^d\limits_{i,j=1}\left|u_{y_iy_j}(s,s,y)-u_{y_iy_j}(s,s,y^\prime)\right|\left.+\sum^d_{i=1}\left|u_{y_i}(s,s,y)-u_{y_i}(s,s,y^\prime)\right|+\left|u(s,s,y)-u(s,s,y^\prime)\right|\right) \\
		&+L\left(\sum^d\limits_{i,j=1}\left|\widehat{u}_{y_iy_j}(t,s,y)-\widehat{u}_{y_iy_j}(t,s,y^\prime)\right|+\sum^d_{i=1}\left|\widehat{u}_{y_i}(t,s,y)-\widehat{u}_{y_i}(t,s,y^\prime)\right|\right.+\left|\widehat{u}(t,s,y)-\widehat{u}(t,s,y^\prime)\right|\\
		&\quad+\sum^d\limits_{i,j=1}\left|\widehat{u}_{y_iy_j}(s,s,y)-\widehat{u}_{y_iy_j}(s,s,y^\prime)\right|\left.+\sum^d_{i=1}\left|\widehat{u}_{y_i}(s,s,y)-\widehat{u}_{y_i}(s,s,y^\prime)\right|+\left|\widehat{u}(s,s,y)-\widehat{u}(s,s,y^\prime)\right|\right) \\
		&+L\left(\sum^d\limits_{i,j=1}\left|g_{y_iy_j}(t,y)-g_{y_iy_j}(t,y^\prime)\right|+\sum^d_{i=1}\left|g_{y_i}(t,y)-g_{y_i}(t,y^\prime)\right|\right.+\left|g(t,y)-g(t,y^\prime)\right| \\
		&\quad+\sum^d\limits_{i,j=1}\left|g_{y_iy_j}(0,y)-g_{y_iy_j}(0,y^\prime)\right|\left.+\sum^d_{i=1}\left|g_{y_i}(0,y)-g_{y_i}(0,y^\prime)\right|+\left|g(0,y)-g(0,y^\prime)\right|\right)+2K|y-y^\prime|^\alpha \\
		\leq &~ L\Bigg(|u(t,\cdot,\cdot)|^{(2+\alpha)}_{[0,t]\times\mathbb{R}^d}|y-y^\prime|^\alpha+\left|u(s,\cdot,\cdot)\right|^{(2+\alpha)}_{[0,s]\times\mathbb{R}^d}|y-y^\prime|^\alpha \\
		&\qquad+|\widehat{u}(t,\cdot,\cdot)|^{(2+\alpha)}_{[0,t]\times\mathbb{R}^d}|y-y^\prime|^\alpha+\left|\widehat{u}(s,\cdot,\cdot)\right|^{(2+\alpha)}_{[0,s]\times\mathbb{R}^d}|y-y^\prime|^\alpha \\
		&\qquad+|g(t,\cdot)|^{(2+\alpha)}_{\mathbb{R}^d}|y-y^\prime|^\alpha+\left|g(0,\cdot)\right|^{(2+\alpha)}_{\mathbb{R}^d}|y-y^\prime|^\alpha\Bigg)+2K|y-y^\prime|^\alpha \\
		\leq & \left(2K+L\left(\lVert u\rVert^{(2+\alpha)}_{[0,\delta]}+\lVert \widehat{u}\rVert^{(2+\alpha)}_{[0,\delta]}+\lVert g\rVert^{(2+\alpha)}_{[0,\delta]}\right)\right)|y-y^\prime|^\alpha \\
		\leq &~C_4(R)|y-y^\prime|^\alpha, 
	\end{align*}
	and for every $y\in\mathbb{R}^d$, 
	\begin{equation*}
		\begin{split}
			&|F_\chi(t,s,y,\theta_\sigma(t,s,y))-F_\chi(t,0,y,\theta_0(t,y))| \\
			\leq &~L\left|\theta_\sigma(t,s,y)-\theta_0(t,y)\right|+K(s-0)^\frac{\alpha}{2} \\
			\leq &~L\left(\sum^d\limits_{i,j=1}\left|u_{y_iy_j}(t,s,y)-g_{y_iy_j}(t,y)\right|+\sum^d_{i=1}\left|u_{y_i}(t,s,y)-g_{y_i}(t,y)\right|\right.\\
			&\qquad+\left|u(t,s,y)-g(t,y)\right|+\sum^d\limits_{i,j=1}\left|u_{y_iy_j}(s,s,y)-g_{y_iy_j}(0,y)\right| \\
			&\qquad \left.+\sum^d_{i=1}\left|u_{y_i}(s,s,y)-g_{y_i}(0,y)\right|+\left|u(s,s,y)-g(0,y)\right|\right) \\
			&+L\left(\sum^d\limits_{i,j=1}\left|\widehat{u}_{y_iy_j}(t,s,y)-g_{y_iy_j}(t,y)\right|+\sum^d_{i=1}\left|\widehat{u}_{y_i}(t,s,y)-g_{y_i}(t,y)\right|\right.\\
			&\qquad+\left|\widehat{u}(t,s,y)-g(t,y)\right|+\sum^d\limits_{i,j=1}\left|\widehat{u}_{y_iy_j}(s,s,y)-g_{y_iy_j}(0,y)\right| \\
			&\qquad \left.+\sum^d_{i=1}\left|\widehat{u}_{y_i}(s,s,y)-g_{y_i}(0,y)\right|+\left|\widehat{u}(s,s,y)-g(0,y)\right|\right)+K(s-0)^\frac{\alpha}{2} \\
			\leq &~ L\left(|\left(u-g\right)(t,\cdot,\cdot)|^{(2+\alpha)}_{[0,t]\times\mathbb{R}^d}(s-0)^\frac{\alpha}{2}+\sup\limits_{\overline{s}\in(0,s)}|g_t(\overline{s},\cdot)|^{(2+\alpha)}_{\mathbb{R}^d}(s-0)^1\right. \\
			&\qquad+|u(s,\cdot,\cdot)|^{(2+\alpha)}_{[0,s]\times\mathbb{R}^d}(s-0)^\frac{\alpha}{2} +|\left(\widehat{u}-g\right)(t,\cdot,\cdot)|^{(2+\alpha)}_{[0,t]\times\mathbb{R}^d}(s-0)^\frac{\alpha}{2} \\
			&\left.\qquad+\sup\limits_{\overline{s}\in(0,s)}|g_t(\overline{s},\cdot)|^{(2+\alpha)}_{\mathbb{R}^d}(s-0)^1+|\widehat{u}(s,\cdot,\cdot)|^{(2+\alpha)}_{[0,s]\times\mathbb{R}^d}(s-0)^\frac{\alpha}{2}\right)+K(s-0)^\frac{\alpha}{2} \\
			\leq & \left(K+L\left(\lVert u-g\rVert^{(2+\alpha)}_{[0,\delta]}+\lVert \widehat{u}-g\rVert^{(2+\alpha)}_{[0,\delta]}+\lVert g\rVert^{(2+\alpha)}_{[0,\delta]}\right)\right)(s-0)^\frac{\alpha}{2} \\
			\leq &~C_5(R)\delta^\frac{\alpha}{2}.
		\end{split}
	\end{equation*}

	Hence, we have
	\begin{equation} \label{Holder continuity of varphi with respect to y}
		\begin{split}
			&|\varphi(t,s,y)-\varphi(t,s,y^\prime)| \\
			\leq &~C_4(R)\Bigg(\delta^1\left|(u-\widehat{u})_s(t,\cdot,\cdot)\right|^{(0)}_{[0,t]\times\mathbb{R}^d}+\delta^\frac{1+\alpha}{2}\sum^d\limits_{i=1}\left|(u-\widehat{u})_{y_i}(t,\cdot,\cdot)\right|^{(1+\alpha)}_{[0,t]\times\mathbb{R}^d} \\
			&\qquad\qquad+\delta^\frac{\alpha}{2}\sum^d\limits_{i,j=1}\left|(u-\widehat{u})_{y_iy_j}(t,\cdot,\cdot)\right|^{(\alpha)}_{[0,t]\times\mathbb{R}^d}+\delta^1\left|(u-\widehat{u})_s(s,\cdot,\cdot)\right|^{(0)}_{[0,s]\times\mathbb{R}^d} \\
			&\qquad\qquad+\delta^\frac{1+\alpha}{2}\sum^d\limits_{i=1}\left|(u-\widehat{u})_{y_i}(s,\cdot,\cdot)\right|^{(1+\alpha)}_{[0,s]\times\mathbb{R}^d}+\delta^\frac{\alpha}{2}\sum^d\limits_{i,j=1}\left|(u-\widehat{u})_{y_iy_j}(s,\cdot,\cdot)\right|^{(\alpha)}_{[0,s]\times\mathbb{R}^d}\Bigg)|y-y^\prime|^\alpha \\
			&\qquad+C_5(R)\delta^\frac{\alpha}{2}\Bigg(\left|(u-\widehat{u})(t,\cdot,\cdot)\right|^{(\alpha)}_{[0,t]\times\mathbb{R}^d}+\sum^d\limits_{i=1}\left|(u-\widehat{u})_{y_i}(t,\cdot,\cdot)\right|^{(\alpha)}_{[0,t]\times\mathbb{R}^d} \\
			&\qquad\qquad\qquad+\sum^d\limits_{i,j=1}\left|(u-\widehat{u})_{y_iy_j}(t,\cdot,\cdot)\right|^{(\alpha)}_{[0,t]\times\mathbb{R}^d}+\left|(u-\widehat{u})(s,\cdot,\cdot)\right|^{(\alpha)}_{[0,s]\times\mathbb{R}^d} \\
			&\qquad\qquad\qquad+\sum^d\limits_{i=1}\left|(u-\widehat{u})_{y_i}(s,\cdot,\cdot)\right|^{(\alpha)}_{[0,s]\times\mathbb{R}^d}+\sum^d\limits_{i,j=1}\left|(u-\widehat{u})_{y_iy_j}(s,\cdot,\cdot)\right|^{(\alpha)}_{[0,s]\times\mathbb{R}^d}\Bigg)|y-y^\prime|^\alpha \\
			\leq &~C_6(R)\delta^\frac{\alpha}{2}|y-y^\prime|^\alpha\lVert u-\widehat{u}\rVert^{(2+\alpha)}_{[0,\delta]}.  
		\end{split}
	\end{equation}
	
	By \eqref{Holder continuity of varphi with respect to s}, \eqref{Boundness of varphi} and \eqref{Holder continuity of varphi with respect to y}, for any $t\in[0,\delta]$, we conclude that
	\begin{equation} \label{Estimate of varphi}
		|\varphi(t,\cdot,\cdot)|^{(\alpha)}_{[0,t]\times\mathbb{R}^d}\leq C_7(R)\delta^{\frac{\alpha}{2}}\lVert u-\widehat{u}\rVert^{(2+\alpha)}_{[0,\delta]}.
	\end{equation} 
	
	\ \ 
	
	\noindent (\textbf{Estimates of $|\varphi_t(t,\cdot,\cdot)|^{(\alpha)}_{[0,t]\times\mathbb{R}^d}$.}) Next, we analyze the H\"{o}lder continuity of $\varphi_t(t,s,y)$ with respect to $s$ and $y$ in $[0,t]\times\mathbb{R}^d$. By the integral representation of $\varphi(t,s,y)$ in \eqref{varphi}, we write $\varphi_t(t,s,y)$ as
	\begin{equation*}
		\begin{split}
			\varphi_t(t,s,y)=\sum_{\mu}\left(\frac{\partial}{\partial t}I_\mu(t,s,y)\right),
		\end{split}
	\end{equation*}
	where ``$I$" represents the integrals in \eqref{varphi} and the subscript $\mu\in\{u,p_i,q_{ij},l,m_i,n_{ij};~i,j=1,\ldots,d\}$ indicate the variable, of which $F$ is differentiated. 
	
	Then, we analyze $\frac{\partial}{\partial t}I_u$ in the case that $\mu=u$ as follows: 
	\begin{equation*}
		\begin{split}
			&\frac{\partial}{\partial t}I_u(t,s,y) \\
			=&\int^1_0\frac{\partial \big(F_u(t,s,y,\theta_\sigma(t,s,y))-F_u(t,0,y,\theta_0(t,y))\big)}{\partial t}\times(u(t,s,y)-\widehat{u}(t,s,y))d\sigma \\
			&\qquad+\int^1_0\big(F_u(t,s,y,\theta_\sigma(t,s,y))-F_u(t,0,y,\theta_0(t,y))\big)\times\frac{\partial (u(t,s,y)-\widehat{u}(t,s,y))}{\partial t}d\sigma \\
			=&\int^1_0\Bigg\{\bigg[F_{ut}(t,s,y,\theta_\sigma(t,s,y))-F_{ut}(t,0,y,\theta_0(t,y))\bigg] \\
			&\qquad+\bigg[F_{uu}(t,s,y,\theta_\sigma(t,s,y))\cdot\Big(\sigma u_t(t,s,y)+(1-\sigma)\widehat{u}_t(t,s,y)\Big) -F_{uu}(t,0,y,\theta_0(t,y))\cdot g_t(t,y)\bigg] \\
			&\qquad+\sum^n_{i=1}\bigg[F_{up_i}(t,s,y,\theta_\sigma(t,s,y))\cdot\Big(\sigma u_{y_it}(t,s,y)+(1-\sigma)\widehat{u}_{y_it}(t,s,y)\Big)-F_{up_i}(t,0,y,\theta_0(t,y))\cdot g_{y_it}(t,y)\bigg] \\
			&\qquad+\sum^n_{i,j=1}\Big[F_{uq_{ij}}(t,s,y,\theta_\sigma(t,s,y))\cdot\bigg(\sigma u_{y_iy_jt}(t,s,y)+(1-\sigma)\widehat{u}_{y_iy_jt}(t,s,y)\Big) \\
			&\qquad\qquad \qquad -F_{uq_{ij}}(t,0,y,\theta_0(t,y))\cdot g_{y_iy_jt}(t,y)\bigg]\Bigg\}\times(u(t,s,y)-\widehat{u}(t,s,y))d\sigma \\
			&+\int^1_0\big(F_u(t,s,y,\theta_\sigma(t,s,y))-F_u(t,0,y,\theta_0(t,y))\big)\times(u_t(t,s,y)-\widehat{u}_t(t,s,y))d\sigma \\
			=&\int^1_0\big(F_{ut}(t,s,y,\theta_\sigma(t,s,y))-F_{ut}(t,0,y,\theta_0(t,y))\big)\times(u(t,s,y)-\widehat{u}(t,s,y))d\sigma \\
			&+\int^1_0\big(F_u(t,s,y,\theta_\sigma(t,s,y))-F_u(t,0,y,\theta_0(t,y))\big)\times(u_t(t,s,y)-\widehat{u}_t(t,s,y))d\sigma \\
			&+\sum_\nu J_\nu(t,s,y) \\
			=&J_1(t,s,y)+J_2(t,s,y)+\sum_\nu J_\nu(t,s,y),
		\end{split}
	\end{equation*}
	where the index $\nu\in \{uu,up_i,uq_{ij};~i,j=1,\ldots,d\}$ in the summations. 
	
	We are to estimate the $|\frac{\partial}{\partial t}I_u(t,s,y)-\frac{\partial}{\partial t}I_u(t,s^\prime,y)|$ in $|\varphi_t(t,s,y)-\varphi_t(t,s^\prime,y)|$ for any $0\leq s\leq s^\prime\leq t\leq \delta\leq T$ and $y\in\mathbb{R}^d$. It can be observed that the estimates of $J_1$ and $J_2$ are the same as the ones of $\varphi$. We focus on those $J_\nu$ terms below. In order to estimate $|J_{uu}(t,s,y)-J_{uu}(t,s^\prime,y)|$, it is convenient to add and subtract the following to $J_{uu}(t,s,y)-J_{uu}(t,s^\prime,y)$:
	\begin{equation*}
		\int^1_0 F_{uu}(t,s^\prime,y,\theta_\sigma(t,s^\prime,y))\cdot\big(\sigma u_t(t,s^\prime,y)+(1-\sigma)\widehat{u}_t(t,s^\prime,y)\big)\times(u(t,s,y)-\widehat{u}(t,s,y))d\sigma. 
	\end{equation*}
	
	We note that
	\begin{equation*}
		\begin{split}
			&\big|F_{uu}(t,s,y,\theta_\sigma(t,s,y))\cdot\big(\sigma u_t(t,s,y)+(1-\sigma)\widehat{u}_t(t,s,y)\big) \\
			&\qquad\qquad\qquad-F_{uu}(t,s^\prime,y,\theta_\sigma(t,s^\prime,y))\cdot\big(\sigma u_t(t,s^\prime,y)+(1-\sigma)\widehat{u}_t(t,s^\prime,y)\big)\big| \\
			\leq &~\big|F_{uu}(t,s,y,\theta_\sigma(t,s,y))\cdot\big(\sigma u_t(t,s,y)+(1-\sigma)\widehat{u}_t(t,s,y)\big) \\
			&\qquad\qquad\qquad-F_{uu}(t,s^\prime,y,\theta_\sigma(t,s^\prime,y))\cdot\big(\sigma u_t(t,s,y)+(1-\sigma)\widehat{u}_t(t,s,y)\big)\big| \\
			&+\big|F_{uu}(t,s^\prime,y,\theta_\sigma(t,s^\prime,y))\cdot\big(\sigma u_t(t,s,y)+(1-\sigma)\widehat{u}_t(t,s,y)\big) \\
			&\qquad\qquad\qquad-F_{uu}(t,s^\prime,y,\theta_\sigma(t,s^\prime,y))\cdot\big(\sigma u_t(t,s^\prime,y)+(1-\sigma)\widehat{u}_t(t,s^\prime,y)\big)\big| \\
			\leq &~\Big(K(s^\prime-s)^\frac{\alpha}{2}+L\left|\theta_\sigma(t,s,y)-\theta_\sigma(t,s^\prime,y)\right|\Big)\cdot\left(\left|u_t(t,\cdot,\cdot)\right|^{(0)}_{[0,t]\times\mathbb{R}^d}+\left|\widehat{u}_t(t,\cdot,\cdot)\right|^{(0)}_{[0,t]\times\mathbb{R}^d}\right) \\
			&+K\left(\left|u_{ts}(t,\cdot,\cdot)\right|^{(0)}_{[0,t]\times\mathbb{R}^d}+\left|\widehat{u}_{ts}(t,\cdot,\cdot)\right|^{(0)}_{[0,t]\times\mathbb{R}^d}\right)(s^\prime-s)^1 \\
			\leq &~\left[K(s^\prime-s)^\frac{\alpha}{2}+L\left(|u(t,\cdot,\cdot)|^{(2+\alpha)}_{[0,t]\times\mathbb{R}^d}(s^\prime-s)^\frac{\alpha}{2}+\sup\limits_{\overline{s}\in(s,s^\prime)}|u_t(\overline{s},\cdot,\cdot)|^{(2+\alpha)}_{[0,\overline{s}]\times\mathbb{R}^d}(s^\prime-s)^1\right.\right. \\
			&\qquad+|u(s^\prime,\cdot,\cdot)|^{(2+\alpha)}_{[0,s^\prime]\times\mathbb{R}^d}(s^\prime-s)^\frac{\alpha}{2} +|\widehat{u}(t,\cdot,\cdot)|^{(2+\alpha)}_{[0,t]\times\mathbb{R}^d}(s^\prime-s)^\frac{\alpha}{2}+\sup\limits_{\overline{s}\in(s,s^\prime)}|\widehat{u}_t(\overline{s},\cdot,\cdot)|^{(2+\alpha)}_{[0,\overline{s}]\times\mathbb{R}^d}(s^\prime-s)^1 \\
			&\qquad+|\widehat{u}(s^\prime,\cdot,\cdot)|^{(2+\alpha)}_{[0,s^\prime]\times\mathbb{R}^d}(s^\prime-s)^\frac{\alpha}{2}\Bigg)\Bigg]\cdot\left(\left|u_t(t,\cdot,\cdot)\right|^{(0)}_{[0,t]\times\mathbb{R}^d}+\left|\widehat{u}_t(t,\cdot,\cdot)\right|^{(0)}_{[0,t]\times\mathbb{R}^d}\right) \\
			&+K\left(\left|u_{ts}(t,\cdot,\cdot)\right|^{(0)}_{[0,t]\times\mathbb{R}^d}+\left|\widehat{u}_{ts}(t,\cdot,\cdot)\right|^{(0)}_{[0,t]\times\mathbb{R}^d}\right)(s^\prime-s)^1 \\
			\leq &~C_8(R)(s^\prime-s)^\frac{\alpha}{2}, 
		\end{split}
	\end{equation*}
	and 
	\begin{equation*}
		\begin{split}
			&\big|F_{uu}(t,s^\prime,y,\theta_\sigma(t,s^\prime,y))\cdot\big(\sigma u_t(t,s^\prime,y)+(1-\sigma)\widehat{u}_t(t,s^\prime,y)\big) \\
			&\qquad\qquad\qquad\qquad\qquad\qquad\qquad\qquad -F_{uu}(t,0,y,\theta_0(t,y))\cdot g_t(t,y)\big| \\
			\leq &~ \big|F_{uu}(t,s^\prime,y,\theta_\sigma(t,s^\prime,y))\cdot\big(\sigma u_t(t,s^\prime,y)+(1-\sigma)\widehat{u}_t(t,s^\prime,y)\big) \\
			&\qquad\qquad\qquad-F_{uu}(t,s^\prime,y,\theta_\sigma(t,s^\prime,y))\cdot\big(\sigma g_t(t,y)+(1-\sigma)g_t(t,y)\big)\big| \\
			&+\big|F_{uu}(t,s^\prime,y,\theta_\sigma(t,s^\prime,y))\cdot\big(\sigma g_t(t,y)+(1-\sigma)g_t(t,y)\big)-F_{uu}(t,0,y,\theta_0(t,y))\cdot g_t(t,y)\big| \\
			\leq &~ K\left(\left|(u-g)_{ts}(t,\cdot,\cdot)\right|^{(0)}_{[0,t]\times\mathbb{R}^d}+\left|(\widehat{u}-g)_{ts}(t,\cdot,\cdot)\right|^{(0)}_{[0,t]\times\mathbb{R}^d}\right)(s^\prime-0)^1 \\
			&+\Bigg[K(s^\prime-0)^\frac{\alpha}{2}+L\left(|\left(u-g\right)(t,\cdot,\cdot)|^{(2+\alpha)}_{[0,t]\times\mathbb{R}^d}(s^\prime-0)^\frac{\alpha}{2}+\sup\limits_{\overline{s}\in(0,s^\prime)}|g_t(\overline{s},\cdot)|^{(2+\alpha)}_{\mathbb{R}^d}(s^\prime-0)^1\right. \\
			&\qquad+|u(s^\prime,\cdot,\cdot)|^{(2+\alpha)}_{[0,s^\prime]\times\mathbb{R}^d}(s^\prime-0)^\frac{\alpha}{2} +|\left(\widehat{u}-g\right)(t,\cdot,\cdot)|^{(2+\alpha)}_{[0,t]\times\mathbb{R}^d}(s^\prime-0)^\frac{\alpha}{2} \\
			&\left.\qquad+\sup\limits_{\overline{s}\in(0,s^\prime)}|g_t(\overline{s},\cdot)|^{(2+\alpha)}_{\mathbb{R}^d}(s^\prime-0)^1+|\widehat{u}(s^\prime,\cdot,\cdot)|^{(2+\alpha)}_{[0,s^\prime]\times\mathbb{R}^d}(s^\prime-0)^\frac{\alpha}{2}\right)\Bigg]\cdot\left|g_t(t,\cdot)\right|^{(0)}_{\mathbb{R}^d} \\
			\leq &~C_9(R)\delta^\frac{\alpha}{2}.  
		\end{split}
	\end{equation*} 
	
	Hence, we have
	\begin{equation} \label{Holder continuity of J2 with respect to s}
		\begin{split}
			|J_{uu}(t,s,y)-J_{uu}(t,s^\prime,y)| 
			\leq &~ C_8(R)\delta^1(s^\prime-s)^{\frac{\alpha}{2}}\sup\limits_{t\in[0,\delta]}|(u-\widehat{u})_s(t,\cdot,\cdot)|^{(0)}_{[0,t]\times\mathbb{R}^d} \\
			&\qquad+C_9(R)\delta^{\frac{\alpha}{2}}(s^\prime-s)^1\sup\limits_{t\in[0,\delta]}|(u-\widehat{u})_s(t,\cdot,\cdot)|^{(0)}_{[0,t]\times\mathbb{R}^d} \\
			\leq &~ C_{10}(R)\delta^\frac{\alpha}{2}(s^\prime-s)^\frac{\alpha}{2}\lVert u-\widehat{u}\rVert^{(2+\alpha)}_{[0,\delta]}
		\end{split}
	\end{equation}
	which implies that (since $J_{uu}(t,0,y)\equiv 0$)
	\begin{equation} \label{Boundness of J2}
		|J_{uu}(t,\cdot,\cdot)|^{\infty}_{[0,t]\times\mathbb{R}^d}\leq C_{11}(R)\delta^\alpha\lVert u-\widehat{u}\rVert^{(2+\alpha)}_{[0,\delta]}.
	\end{equation} 
	
	Next, to estimate $|J_{uu}(t,s,y)-J_{uu}(t,s,y^\prime)|$, it is convenient to add and subtract     
	\begin{equation*}
		\begin{split}
			&\int^1_0\bigg[F_{uu}(t,s,y^\prime,\theta_\sigma(t,s,y^\prime))\cdot\Big(\sigma u_t(t,s,y^\prime)+(1-\sigma)\widehat{u}_t(t,s,y^\prime)\Big) \\
			&\qquad\qquad\qquad-F_{uu}(t,0,y^\prime,\theta_0(t,y^\prime))\cdot g_t(t,y^\prime)\bigg]\times(u_t(t,s,y)-\widehat{u}_t(t,s,y))d\sigma
		\end{split}
	\end{equation*} 
	Then we note that
	\begin{equation*}
		\begin{split}
			&\Big|F_{uu}(t,s,y,\theta_\sigma(t,s,y))\cdot\Big(\sigma u_t(t,s,y)+(1-\sigma)\widehat{u}_t(t,s,y)\Big) \\
			&\qquad\qquad\qquad\qquad\qquad\qquad\qquad\qquad-F_{uu}(t,0,y,\theta_0(t,y))\cdot g_t(t,y) \\
			&\quad -F_{uu}(t,s,y^\prime,\theta_\sigma(t,s,y^\prime))\cdot\Big(\sigma u_t(t,s,y^\prime)+(1-\sigma)\widehat{u}_t(t,s,y^\prime)\Big) \\
			&\qquad\qquad\qquad\qquad\qquad\qquad\qquad\qquad +F_{uu}(t,0,y^\prime,\theta_0(t,y^\prime))\cdot g_t(t,y^\prime)\Big| \\
			\leq &~ \Big|F_{uu}(t,s,y,\theta_\sigma(t,s,y))\cdot\Big(\sigma u_t(t,s,y)+(1-\sigma)\widehat{u}_t(t,s,y)\Big) \\
			&\qquad\qquad\qquad-F_{uu}(t,s,y^\prime,\theta_\sigma(t,s,y^\prime))\cdot\Big(\sigma u_t(t,s,y^\prime)+(1-\sigma)\widehat{u}_t(t,s,y^\prime)\Big)\Big| \\
			&+\Big|F_{uu}(t,0,y^\prime,\theta_0(t,y^\prime))\cdot g_t(t,y^\prime)-F_{uu}(t,0,y,\theta_0(t,y))\cdot g_t(t,y)\Big| \\
			=:&~ N_1+N_2
		\end{split}
	\end{equation*}
	
	For $N_1$, we have 
	\begin{equation*}
		\begin{split}
			&\Big|F_{uu}(t,s,y,\theta_\sigma(t,s,y))\cdot\Big(\sigma u_t(t,s,y)+(1-\sigma)\widehat{u}_t(t,s,y)\Big) \\
			&\qquad\qquad\qquad-F_{uu}(t,s,y^\prime,\theta_\sigma(t,s,y^\prime))\cdot\Big(\sigma u_t(t,s,y^\prime)+(1-\sigma)\widehat{u}_t(t,s,y^\prime)\Big)\Big| \\
			\leq &~ \Big|F_{uu}(t,s,y,\theta_\sigma(t,s,y))\cdot\Big(\sigma u_t(t,s,y)+(1-\sigma)\widehat{u}_t(t,s,y)\Big) \\
			&\qquad\qquad\qquad-F_{uu}(t,s,y,\theta_\sigma(t,s,y))\cdot\Big(\sigma u_t(t,s,y^\prime)+(1-\sigma)\widehat{u}_t(t,s,y^\prime)\Big)\Big| \\
			&+\Big|F_{uu}(t,s,y,\theta_\sigma(t,s,y))\cdot\Big(\sigma u_t(t,s,y^\prime)+(1-\sigma)\widehat{u}_t(t,s,y^\prime)\Big) \\
			&\qquad\qquad\qquad-F_{uu}(t,s,y^\prime,\theta_\sigma(t,s,y^\prime))\cdot\Big(\sigma u_t(t,s,y^\prime)+(1-\sigma)\widehat{u}_t(t,s,y^\prime)\Big)\Big| \\
			\leq &~K\left(\left|u_t(t,\cdot,\cdot)\right|^{(\alpha)}_{[0,t]\times\mathbb{R}^d}+\left|\widehat{u}_t(t,\cdot,\cdot)\right|^{(\alpha)}_{[0,t]\times\mathbb{R}^d}\right)|y-y^\prime|^\alpha \\
			&+\Bigg[K|y-y^\prime|^\alpha+L\Bigg(|u(t,\cdot,\cdot)|^{(2+\alpha)}_{[0,t]\times\mathbb{R}^d}|y-y^\prime|^\alpha+\left|u(s,\cdot,\cdot)\right|^{(2+\alpha)}_{[0,s]\times\mathbb{R}^d}|y-y^\prime|^\alpha \\
			&\quad +|\widehat{u}(t,\cdot,\cdot)|^{(2+\alpha)}_{[0,t]\times\mathbb{R}^d}|y-y^\prime|^\alpha+\left|\widehat{u}(s,\cdot,\cdot)\right|^{(2+\alpha)}_{[0,s]\times\mathbb{R}^d}|y-y^\prime|^\alpha\Bigg)\Bigg] \\
			& \qquad\qquad\qquad\qquad\qquad\qquad\times \left(\left|u_t(t,\cdot,\cdot)\right|^{(0)}_{[0,t]\times\mathbb{R}^d}+\left|\widehat{u}_t(t,\cdot,\cdot)\right|^{(0)}_{[0,t]\times\mathbb{R}^d}\right)  \\
			\leq &~C_{12}(R)|y-y^\prime|^\alpha, 
		\end{split}
	\end{equation*}
	
	For $N_2$, we have 
	\begin{equation*}
		\begin{split}
			&\Big|F_{uu}(t,0,y^\prime,\theta_0(t,y^\prime))\cdot g_t(t,y^\prime)-F_{uu}(t,0,y,\theta_0(t,y))\cdot g_t(t,y)\Big| \\
			\leq &~\Big|F_{uu}(t,0,y^\prime,\theta_0(t,y^\prime))\cdot g_t(t,y^\prime)-F_{uu}(t,0,y,\theta_0(t,y))\cdot g_t(t,y^\prime)\Big| \\
			&\qquad+\Big|F_{uu}(t,0,y,\theta_0(t,y))\cdot g_t(t,y^\prime)-F_{uu}(t,0,y,\theta_0(t,y))\cdot g_t(t,y)\Big| \\
			\leq & \bigg(K|y-y^\prime|^\alpha+|g(t,\cdot)|^{(2+\alpha)}_{\mathbb{R}^d}|y-y^\prime|^\alpha+\left|g(0,\cdot)\right|^{(2+\alpha)}_{\mathbb{R}^d}|y-y^\prime|^\alpha\bigg)\cdot\left|g_t(t,\cdot)\right|^{(0)}_{\mathbb{R}^d} \\
			&\qquad+K\left|g_t(t,\cdot)\right|^{(\alpha)}_{\mathbb{R}^d}|y-y^\prime|^\alpha  \\
			\leq &~C_{13}(R)|y-y^\prime|^\alpha. 
		\end{split}
	\end{equation*}
	
	So we have 
	\begin{equation*}
		\begin{split}
			&\Big|F_{uu}(t,s,y,\theta_\sigma(t,s,y))\cdot\Big(\sigma u_t(t,s,y)+(1-\sigma)\widehat{u}_t(t,s,y)\Big) \\
			&\qquad\qquad\qquad\qquad\qquad\qquad\qquad\qquad -F_{uu}(t,0,y,\theta_0(t,y))\cdot g_t(t,y) \\
			&\quad -F_{uu}(t,s,y^\prime,\theta_\sigma(t,s,y^\prime))\cdot\Big(\sigma u_t(t,s,y^\prime)+(1-\sigma)\widehat{u}_t(t,s,y^\prime)\Big) \\
			&\qquad\qquad\qquad\qquad\qquad\qquad\qquad\qquad +F_{uu}(t,0,y^\prime,\theta_0(t,y^\prime))\cdot g_t(t,y^\prime)\Big| \\
			\leq &~ N_1+N_2\leq C_{14}(R)|y-y^\prime|^\alpha
		\end{split}
	\end{equation*}
	
	Moreover, we need to estimate 
	\begin{equation*}
		\begin{split}
			&\Big|F_{uu}(t,s,y,\theta_\sigma(t,s,y))\cdot\Big(\sigma u_t(t,s,y)+(1-\sigma)\widehat{u}_t(t,s,y)\Big) \\
			&\qquad\qquad\qquad\qquad\qquad\qquad\qquad\qquad 
			-F_{uu}(t,0,y,\theta_0(t,y))\cdot g_t(t,y)\Big| \\
			\leq &~ \Big|F_{uu}(t,s,y,\theta_\sigma(t,s,y))\cdot\Big(\sigma u_t(t,s,y)+(1-\sigma)\widehat{u}_t(t,s,y)\Big) \\
			&\qquad\qquad\qquad\qquad\qquad\qquad\qquad\qquad -F_{uu}(t,s,y,\theta_\sigma(t,s,y))\cdot g_t(t,y)\Big| \\
			&+\Big|F_{uu}(t,s,y,\theta_\sigma(t,s,y))\cdot g_t(t,y)-F_{uu}(t,0,y,\theta_0(t,y))\cdot g_t(t,y)\Big| \\
			\leq &~K\left(\left|(u-g)_{ts}(t,\cdot,\cdot)\right|^{(0)}_{[0,t]\times\mathbb{R}^d}+\left|(\widehat{u}-g)_{ts}(t,\cdot,\cdot)\right|^{(0)}_{[0,t]\times\mathbb{R}^d}\right)(s-0)^1 \\
			& ~+\Bigg[ K(s-0)^\frac{\alpha}{2}+L\left(|\left(u-g\right)(t,\cdot,\cdot)|^{(2+\alpha)}_{[0,t]\times\mathbb{R}^d}(s-0)^\frac{\alpha}{2}+\sup\limits_{\overline{s}\in(0,s)}|g_t(\overline{s},\cdot)|^{(2+\alpha)}_{\mathbb{R}^d}(s-0)^1\right. \\
			&\qquad+|u(s,\cdot,\cdot)|^{(2+\alpha)}_{[0,s]\times\mathbb{R}^d}(s-0)^\frac{\alpha}{2} +|\left(\widehat{u}-g\right)(t,\cdot,\cdot)|^{(2+\alpha)}_{[0,t]\times\mathbb{R}^d}(s-0)^\frac{\alpha}{2} \\
			&\left.\qquad+\sup\limits_{\overline{s}\in(0,s)}|g_t(\overline{s},\cdot)|^{(2+\alpha)}_{\mathbb{R}^d}(s-0)^1+|\widehat{u}(s,\cdot,\cdot)|^{(2+\alpha)}_{[0,s]\times\mathbb{R}^d}(s-0)^\frac{\alpha}{2}\right)\Bigg]\cdot\left|g_{t}(t,\cdot)\right|^{(0)}_{\mathbb{R}^d} \\
			\leq &~C_{15}(R)\delta^\frac{\alpha}{2}. 
		\end{split}
	\end{equation*}
	
	Consequently, we have   
	\begin{equation} \label{Holder continuity of J2 with respect to y}
		\begin{split}
			|J_{uu}(t,s,y)-J_{uu}(t,s,y^\prime)| 
			\leq &~ C_{14}(R)\delta^1|y-y^\prime|^\alpha\sup\limits_{t\in[0,\delta]}|(u-\widehat{u})_s(t,\cdot,\cdot)|^{(0)}_{[0,t]\times\mathbb{R}^d} \\
			&\qquad+C_{15}(R)\delta^\frac{\alpha}{2}\sup\limits_{t\in[0,\delta]}\left|(u-\widehat{u})(t,\cdot,\cdot)\right|^{(\alpha)}_{[0,t]\times\mathbb{R}^d}|y-y^\prime|^\alpha \\
			\leq &~ C_{16}(R)\delta^\frac{\alpha}{2}|y-y^\prime|^\alpha\lVert u-\widehat{u}\rVert^{(2+\alpha)}_{[0,\delta]}. 
		\end{split}
	\end{equation} 
	
	From \eqref{Holder continuity of J2 with respect to s}, \eqref{Boundness of J2} and \eqref{Holder continuity of J2 with respect to y}, for any $t\in[0,\delta]$, we obtain
	\begin{equation*}
		|J_{uu}(t,\cdot,\cdot)|^{(\alpha)}_{[0,t]\times\mathbb{R}^d}\leq C_{17}(R)\delta^{\frac{\alpha}{2}}\lVert u-\widehat{u}\rVert^{(2+\alpha)}_{[0,\delta]}. 
	\end{equation*} 
	Similarly, for any $\nu\in\{uu,up_i,uq_{ij}\}$, we also obtain
	\begin{equation*}
		|J_\nu(t,\cdot,\cdot)|^{(\alpha)}_{[0,t]\times\mathbb{R}^d}\leq C_{18}(R)\delta^{\frac{\alpha}{2}}\lVert u-\widehat{u}\rVert^{(2+\alpha)}_{[0,\delta]}. 
	\end{equation*} 
	Hence, we have
	\begin{equation*}
		\left|\frac{\partial}{\partial t}I_u(t,\cdot,\cdot)\right|^{(\alpha)}_{[0,t]\times\mathbb{R}^d}\leq C_{19}(R)\delta^\frac{\alpha}{2}\lVert u-\widehat{u}\rVert^{(2+\alpha)}_{[0,\delta]}. 
	\end{equation*}
	
	\ \ 
	
	By the same way, we can also estimate $|\frac{\partial}{\partial t}I_\mu(t,\cdot,\cdot)|^{(\alpha)}_{[0,t]\times\mathbb{R}^d}$ for $\mu\in\{u,p_i,q_{ij},l,m_i,n_{ij}\}$: 
	\begin{equation*}
		\left|\frac{\partial}{\partial t}I_\mu(t,\cdot,\cdot)\right|^{(\alpha)}_{[0,t]\times\mathbb{R}^d}\leq C_{20}(R)\delta^\frac{\alpha}{2}\lVert u-\widehat{u}\rVert^{(2+\alpha)}_{[0,\delta]}. 
	\end{equation*}
	Hence, for any $t\in[0,\delta]$, we obtain 
	\begin{equation*}
		|\varphi_t(t,\cdot,\cdot)|^{(\alpha)}_{[0,t]\times\mathbb{R}^d}\leq C_{21}(R)\delta^\frac{\alpha}{2}\lVert u-\widehat{u}\rVert^{(2+\alpha)}_{[0,\delta]}.  
	\end{equation*} 
	
	Together with the estimate of $|\varphi(t,\cdot,\cdot)|^{(\alpha)}_{[0,t]\times\mathbb{R}^d}$ in \eqref{Estimate of varphi}, for any $t\in[0,\delta]$, we have 
	\begin{equation*}
		|\varphi(t,\cdot,\cdot)|^{(\alpha)}_{[0,t]\times\mathbb{R}^d}+|\varphi_t(t,\cdot,\cdot)|^{(\alpha)}_{[0,t]\times\mathbb{R}^d}\leq C(R)\delta^\frac{\alpha}{2}\lVert u-\widehat{u}\rVert^{(2+\alpha)}_{[0,\delta]},
	\end{equation*} 
	which indicates that
	\begin{equation*}
		\sup\limits_{t\in[0,\delta]}\left\{|\varphi(t,\cdot,\cdot)|^{(\alpha)}_{[0,t]\times\mathbb{R}^d}+|\varphi_t(t,\cdot,\cdot)|^{(\alpha)}_{[0,t]\times\mathbb{R}^d}\right\}\leq C(R)\delta^\frac{\alpha}{2}\lVert u-\widehat{u}\rVert^{(2+\alpha)}_{[0,\delta]}.   
	\end{equation*} 
	Consequently, we derive a contraction mapping
	\begin{equation*}
		\lVert\Lambda(u)-\Lambda(\widehat{u})\rVert^{(2+\alpha)}_{[0,\delta]}\leq C(R)\delta^{\frac{\alpha}{2}}\lVert u-\widehat{u}\rVert^{(2+\alpha)}_{[0,\delta]}. 
	\end{equation*} 
	
	\ \ 
	
	\noindent (\textbf{A contraction $\Lambda$ mapping $\mathcal{U}$ into itself.}) Now, we need to choose a suitably large $R$ such that $\Lambda$ maps $\mathcal{U}$ into itself. If $\delta$ and $R$ satisfy 
	\begin{equation*}
		C(R)\delta^{\frac{\alpha}{2}}\leq \frac{1}{2} 
	\end{equation*} 
	then $\Lambda$ is a $\frac{1}{2}$-contraction, and for any $u\in\mathcal{U}$, we have
	\begin{equation*}
		\lVert\Lambda(u)-g\rVert^{(2+\alpha)}_{[0,\delta]}\leq \frac{R}{2}+\lVert \Lambda(g)-g\rVert^{(2+\alpha)}_{[0,\delta]}. 
	\end{equation*} 
	The function $G=\Lambda(g)-g$ is the solution of 
	\begin{equation*}
		\left\{
		\begin{array}{lr}
			G_s(t,s,y)=\mathcal{L}G+F\big(t,s,y,g(t,y),g_y(t,y),g_{yy}(t,y),g(s,y),g_y(s,y),g_{yy}(s,y)\big), \\
			G(t,0,y)=0,\quad 0\leq s\leq t\leq \delta,\quad y\in\mathbb{R}^n.
		\end{array} 
		\right.
	\end{equation*}
	It is clear that the term $F\in\Omega^{(\alpha)}_{[0,\delta]}$. By \eqref{Estimates of solutions of nonlocal linear PDE} and \eqref{Estimates of solutions of nonlocal PDEs}, there exists a constant $c>0$ independent of $\delta$, such that 
	\begin{equation*}
		\lVert G\rVert^{(2+\alpha)}_{[0,\delta]}\leq c\sup\limits_{t\in[0,\delta]}\left\{\left|\psi(t,\cdot,\cdot)\right|^{(\alpha)}_{[0,t]\times\mathbb{R}^d}+\left|\psi_t(t,\cdot,\cdot)\right|^{(\alpha)}_{[0,t]\times\mathbb{R}^d}\right\}\doteq C^\prime, 
	\end{equation*}
	where $\psi(t,s,y)=F\big(t,s,y,g(t,y),g_y(t,y),g_{yy}(t,y),g(s,y),g_y(s,y),g_{yy}(s,y)\big)$. 
	
	Hence, we have 
	\begin{equation*}
		\lVert\Lambda(u)-g\rVert^{(2+\alpha)}_{[0,\delta]}\leq \frac{R}{2}+C^\prime.
	\end{equation*}
	Therefore for a suitably large $R$, $\Lambda$ is a contraction mapping $\mathcal{U}$ into itself, and it has a unique fixed point $u$ in $\mathcal{U}$ satisfying
	\begin{equation*} 
		\left\{
		\begin{array}{rcl}
			u_s(t,s,y) & = & F\big(t,s,y,u(t,s,y),u_y(t,s,y),u_{yy}(t,s,y), \\
			&& \qquad\qquad   u(s,s,y),u_y(s,s,y),u_{yy}(s,s,y)\big), \\
			u(t,0,y) & = & g(t,y),\quad 0\leq s\leq t\leq \delta,\quad y\in\mathbb{R}^d.
		\end{array}
		\right. 
	\end{equation*}
	
	\ \ 
	
	\noindent \textbf{(Uniqueness)} To complete the proof, we ought to show that $u$ is the unique solution of \eqref{Nonlocal fully nonlinear equation in Section 3} in $\Omega^{(2+\alpha)}_{[0,\delta]}$. It follows with similar arguments as in the proof of Theorem \ref{Well-posedness of (u,u_t)} for nonlocal linear PDEs, or by studying directly the following problem satisfied by the difference of two solutions $u^1$ and $u^2$:  
	\begin{equation*}
		\left\{
		\begin{array}{rcl}
			\left(u^1-u^2\right)_s(t,s,y) & = & F\big(t,s,y,u^1(t,s,y),u^1_y(t,s,y),u^1_{yy}(t,s,y), \\
			&&\qquad\qquad u^1(s,s,y),u^1_y(s,s,y),u^1_{yy}(s,s,y)\big)
			\\
			&&-F\big(t,s,y,u^2(t,s,y),u^2_y(t,s,y),u^2_{yy}(t,s,y), \\
			&&\qquad\qquad\quad  u^2(s,s,y),u^2_y(s,s,y),u^2_{yy}(s,s,y)\big), \\
			\left(u^1-u^2\right)(t,0,y) & = & 0,\quad 0\leq s\leq t\leq \delta,\quad y\in\mathbb{R}^d.
		\end{array}
		\right.
	\end{equation*}  
	By using the mean value theorem, the right hand side of the equation above can be expressed as a linear combination of $u^1-u^2$ and its first- and second-order partial derivatives with respect to $y$. In this way, it is transformed into a nonlocal linear equation studied in the previous section. The uniqueness of solutions of nonlocal linear PDEs promises that $u^1=u^2$ in $\Delta[0,\delta]\times\mathbb{R}^d$.
\end{proof}

\begin{remark}[\textbf{Maximally defined solutions}]
	We have proven the local well-posedness of \eqref{Nonlocal fully nonlinear equation in Section 3} in $\Delta[0,\delta]\times\mathbb{R}^d$ and thus the diagonal condition can be determined for $s\in[0,\delta]$. After that, the nonlocal fully nonlinear equation \eqref{Nonlocal fully nonlinear equation in Section 3} reduces to a classical local fully nonlinear PDEs parameterized by $t$. Then we take $\delta$ as initial time and $u(t,\delta,y)$ as initial datum, we can extend the solution to a larger time interval up to the maximal interval. It is analogous to the definition of the maximally defined solution of nonlocal linear PDEs in Remark \ref{Extension of solutions to a larger time interval}. It is noteworthy that the problem of existence at large for arbitrary initial data is a difficult task even in the local fully nonlinear case. The difficulty is caused by the fact that a priori estimate in a very high norm $|\cdot|^{(2+\alpha)}_{[a,b]\times\mathbb{R}^d}$ is needed to establish the existence at large. To this end, there should be severe restrictions on the nonlinearities. More details are discussed in \cite{Krylov1987,Lieberman1996}. 
\end{remark}

\begin{remark}[\textbf{Stability analysis}]
	Suppose $u$ and $\widehat{u}$ correspond to $(F,g)$ and $(\widehat{F},\widehat{g})$, respectively. Then the following estimate holds: 
	\begin{equation*}
		\lVert u-\widehat{u}\rVert^{(2+\alpha)}_{[0,\delta]}\leq c\left(\lVert F-\widehat{F}\rVert^{(\alpha)}_{[0,\delta]}+\lVert g-\widehat{g}\rVert^{(2+\alpha)}_{[0,\delta]}\right) 
	\end{equation*} 
	where both $F$ and $\widehat{F}$ are evaluated at the point $$\big(t,s,y,u(t,s,y),u_y(t,s,y),u_{yy}(t,s,y),u(s,s,y),u_y(s,s,y),s_{yy}(s,s,y)\big).$$ 
	To see this, note that the difference between $u$ and $\widehat{u}$ satisfies	
	\begin{equation*}
		\left\{
		\begin{array}{rcl}
			\left(u-\widehat{u}\right)_s(t,s,y) & = & \mathcal{L}\left(u-\widehat{u}\right)+F\big(t,s,y,u(t,s,y),u_y(t,s,y),u_{yy}(t,s,y), \\
			&&\qquad\qquad\qquad\qquad\qquad u(s,s,y),u_y(s,s,y),u_{yy}(s,s,y)\big)
			\\
			&&-\widehat{F}\big(t,s,y,\widehat{u}(t,s,y),\widehat{u}_y(t,s,y),\widehat{u}_{yy}(t,s,y), \\
			&&\qquad\qquad\qquad\qquad\qquad  \widehat{u}(s,s,y),\widehat{u}_y(s,s,y),\widehat{u}_{yy}(s,s,y)\big), \\
			
			\qquad & & -\mathcal{L}\left(u-\widehat{u}\right)  \\
			\left(u-\widehat{u}\right)(t,0,y) & = & \left(g-\widehat{g}\right)(t,y),\quad 0\leq s\leq t\leq \delta,\quad y\in\mathbb{R}^d.
		\end{array}
		\right.
	\end{equation*} 	
	Similar to the proof of Theorem \ref{Well-posedness of fully nonlinear equations}, we have 
	\begin{equation*}
		\lVert u-\widehat{u}\rVert^{(2+\alpha)}_{[0,\delta]}\leq c\left(\lVert g-\widehat{g}\rVert^{(2+\alpha)}_{[0,\delta]}+\lVert \varphi\rVert^{(\alpha)}_{[0,\delta]}\right)+\frac{1}{2}\lVert u-\widehat{u}\rVert^{(2+\alpha)}_{[0,\delta]}, 
	\end{equation*} 
	where $\varphi$ represents $F-\widehat{F}$ at $u$.
\end{remark}


\section{Stochastic Representation of Solutions of Nonlocal PDEs} \label{sec:stochrep}
In this section, we illustrate that the nonlocal PDEs \eqref{Nonlocal fully nonlinear equation} is closely connected with the theory of stochastic differential equations (SDEs). Recall that with the classical Feynman--Kac formula, a solution of a (local) linear parabolic PDE can be formulated as an expectation of a terminal stochastic process driven by a forward SDE, while a solution of a (local) semi-linear (resp. quasi-linear) PDE can be represented by adapted solutions of decoupled (resp. coupled) FBSDEs. In other words, there is a PDE-representation for adapted solutions of coupled FBSDEs and conversely, the FBSDEs can also provide a probabilistic interpretation and a numerical scheme for parabolic PDEs.

Parallel to the elegant connection between local parabolic PDEs and FBSDEs, there is also a key connection between nonlocal PDEs and a flow of FBSDEs (or BSVIEs). By exploring this connection, we can construct computational methods to seek for solutions of nonlocal PDEs numerically.

\subsection{Existing Feynman--Kac type results}
We first briefly review the existing results that connect a flow of FBSDEs with nonlocal PDEs in the literature. Given a $k$-dimensional Brownian motion $W(\cdot)$ defined in the completed filtered probability space $\left(\Omega,\mathcal{F},\mathbb{F},\mathbb{P}\right)$, let us consider a coupled FBSDE of the form: 
\begin{align}
	X(s) & = y+\int^s_0b(\tau,X(\tau),Y(\tau,\tau))d\tau+\int^s_0\sigma(\tau,X(\tau),Y(\tau,\tau))dW(\tau), \label{Forward SDE} \\
	Y(t,s) & = g(t,X(T))+\int^T_sh(t,\tau,X(\tau),Y(t,\tau),Y(\tau,\tau),Z(t,\tau))d\tau \label{Backward BSVIE} \\
	& \qquad \qquad \quad -\int^T_sZ^\top(t,\tau)dW(\tau), \quad 0\leq t\leq s\leq T,\quad y\in\mathbb{R}^d, \nonumber
\end{align}
Getting rid of the diagonal term $Y(\tau,\tau)$ in \eqref{Forward SDE} and \eqref{Backward BSVIE}, the FBSDEs above reduce to classical ones parameterized by $t\in[0,T]$, which essentially form a flow of FBSDEs. However, the dependence on $t$ and $Y(\tau,\tau)$ violates the flow property of FBSDEs as well as results in nonlocality of its Feynman--Kac-type representations. Heuristically speaking, a family of $\mathbb{R}^d\times\mathbb{R}\times\mathbb{R}^{k\times 1}$-valued random fields 
$$\left(X(\cdot),Y(\cdot,\cdot),Z(\cdot,\cdot)\right)=\big\{\left(X(s),Y(t,s),Z(t,s)\right):0\leq t\leq s\leq T\big\}$$ 
is called an adapted solution of the FBSDE system \eqref{Forward SDE}-\eqref{Backward BSVIE} if 
\begin{enumerate}
	\item $X(\cdot)$ is $\mathbb{F}$-adapted and continuous; 
	\item for each fixed $t\in[0,T]$, $Y(t,\cdot)$ is $\mathbb{F}$-adapted and continuous; 
	\item for each fixed $t\in[0,T]$, $Z(t,\cdot)$ is $\mathbb{F}$-progressively measurable; 
	\item the system holds in the usual It\^{o}'s sense for almost every $t\in[0,T]$. 
\end{enumerate}
Some well-posedness results and applications of the system have been studied in the literature. Such a class of BSDEs \eqref{Backward BSVIE}, the generator $h$ and the terminal term $g$ of which depend on $t$ and/or diagonal terms $Y(\tau,\tau)$, is also usually called as backward stochastic Volterra integral equation (BSVIE). The concept of BSVIEs was initially introduced in \cite{Lin2002} as an extension of classical BSDEs developed in \cite{Bismut1976,Pardoux1990}. Then it was developed in several other studies, including \cite{Aman2005,Yong2006,Anh2010,Ren2010,Djordjevi2013,Djordjevi2015,Hu2019}. Among them, \cite{Yong2006} extends BSVIEs to incorporate with general terminal function motivated by the optimal control of forward-SVIEs (FSVIEs). The regularity of BSVIEs in Hilbert spaces is studied in \cite{Anh2010}. A timely survey on BSVIEs is documented in \cite{Yong2013}. Recently, \cite{Overbeck2018,Wang2020c} develop a theory of path-dependent BSVIEs in a non-Markovian setting. As for applications, BSVIEs generalize the theory of stochastic differential utility to incorporate with state dependence and solve the corresponding TIC stochastic recursive control problems; see \cite{Wei2017,Wang2021,Hamaguchi2021}. Moreover, \cite{Yong2007,Wang2021a} study the applications of BSVIE to dynamic risk measures. 

Next, we discuss about several special situations of the coupled system \eqref{Forward SDE}-\eqref{Backward BSVIE}. Under some conditions, they will lead to some Feynman--Kac-type results in the existing literature. 
\begin{enumerate}
	\item When $b$, $\sigma$ and $h$ are all independent of the diagonal term $Y(\tau,\tau)$, then the system of \eqref{Forward SDE}-\eqref{Backward BSVIE} is reduced to a family of classical FBSDEs parameterized by $t\in[0,T]$, whose solution is linked with that of the following parameterized semilinear PDEs: 
	$$
	\left\{
	\begin{array}{lr}
		u_s(t,s,y)+\frac{1}{2}\mathrm{tr}\left\{\sigma(s,y)\sigma(s,y)^\top u_{yy}(t,s,y)\right\}+b^\top(s,y)u_y(t,s,y) \\
		\qquad\qquad\qquad\qquad\qquad
		+h(t,s,y,u(t,s,y),\sigma^\top(s,y)u_y(t,s,y))=0, \\
		u(t,T,y)=g(t,y), \quad 0\leq t\leq s\leq T,\quad y\in\mathbb{R}^d. 
	\end{array}
	\right.
	$$
	via the relations $Y(t,s)=u(t,s,X(s))$ and $Z(t,s)=\sigma^\top(s,X(s))u_y(t,s,X(s))$. In particular, when the parameter $t$ is taken over a singleton, the result is proved in \cite{Peng1992,Ma2002};  
	
	\item When $b$ and $\sigma$ do not depend on $Y(\tau,\tau)$ and $h$ does not include $Y(t,\tau)$, by letting $s=t$ in \eqref{Backward BSVIE}, \cite{Wang2019} links the system with a nonlocal PDE: 
	$$
	\left\{
	\begin{array}{lr}
		u_s(t,s,y)+\frac{1}{2}\mathrm{tr}\left\{\sigma(s,y)\sigma(s,y)^\top u_{yy}(t,s,y)\right\}+b^\top(s,y)u_y(t,s,y) \\
		\qquad\qquad\qquad\qquad\qquad
		+h\big(t,s,y,u(s,s,y),\sigma^\top(s,y)u_y(t,s,y)\big)=0, \\
		u(t,T,y)=g(t,y), \quad  0\leq t\leq s\leq T, \quad y\in\mathbb{R}^d. 
	\end{array}
	\right.
	$$
	via the relations $Y(t,t)=u(t,t,X(t))$ and $Z(t,s)=\sigma^\top(s,X(s))u_y(t,s,X(s))$;  
	
	\item When $b$ and $\sigma$ are independent of $Y(\tau,\tau)$, \cite{Wang2020} shows that the system corresponds to the following nonlocal PDE:
	$$
	\left\{
	\begin{array}{lr}
		u_s(t,s,y)+\frac{1}{2}\mathrm{tr}\left\{\sigma(s,y)\sigma(s,y)^\top u_{yy}(t,s,y)\right\}+b^\top(s,y)u_y(t,s,y) \\
		\qquad\qquad\qquad
		+h\big(t,s,y,u(t,s,y),u(s,s,y),\sigma^\top(s,y)u_y(t,s,y)\big)=0, \\
		u(t,T,y)=g(t,y), \quad  0\leq t\leq s\leq T, \quad y\in\mathbb{R}^d. 
	\end{array}
	\right.
	$$    
	indicated by $Y(t,s)=u(t,s,X(s))$ and $Z(t,s)=\sigma^\top(s,X(s))u_y(t,s,X(s))$. 
	
	\item Inspired by the previous works, \cite{Hamaguchi2020} concludes that the investigation of a (coupled) flow of forward-backward SDEs \eqref{Forward SDE}-\eqref{Backward BSVIE} will be reduced to the study of the nonlocal PDE of the following form:
	$$
	\left\{
	\begin{array}{lr}
		u_s(t,s,y)+\frac{1}{2}\mathrm{tr}\left\{\sigma\left(s,y,u(s,s,y)\right)\sigma\left(s,y,u(s,s,y)\right)^\top u_{yy}(t,s,y)\right\} \\
		\quad
		+b^\top\left(s,y,u(s,s,y)\right)u_y(t,s,y) \\
		\quad+h\big(t,s,y,u(t,s,y),u(s,s,y),\sigma^\top(s,y,u(s,s,y))u_y(t,s,y)\big)=0, \\
		u(t,T,y)=g(t,y), \qquad  0\leq t\leq s\leq T, \quad y\in\mathbb{R}^d.
	\end{array}
	\right.
	$$
\end{enumerate}

One can observe that the examples (1)-(4) above contain neither a nonlinearity of the highest-order term $u(t,s,y)$ nor a dependence of diagonal terms $u_y(s,s,y)$ and $u_{yy}(s,s,y)$. It is clear that our study about nonlocal PDEs in Section \ref{sec:linear} and Section \ref{sec:nonlinear} provides a general and unified treatment of this class of Feynman--Kac formulas and presents the necessity and significance of FBSDEs in a more general form than \eqref{Forward SDE}-\eqref{Backward BSVIE}.

\subsection{Flow of second-order FBSDEs} \label{sec:2FBVIEs}
In fact, there has been some work connecting classical local fully nonlinear PDEs with FBSDEs. The pioneering work of \cite{Cheridito2007} introduces the concept of second-order BSDEs (2BSDEs) that extends the Feynman--Kac formula to a fully-nonlinear setting. The key ingredient of the method is to introduce a new process that identifies the Hessian matrix of the solution of the corresponding PDE. \cite{Soner2011} provides an alternative formulation of 2BSDEs under a non-dominated family of mutually singular probability measures. By parameterizing the 2FBSDEs in \cite{Cheridito2007} or \cite{Soner2011} with $t$ and introducing some diagonal terms to them, the parallel concept of 2FBSVIEs or a flow of 2FBSDEs is highly expected to provide a probabilistic interpretation and numerical scheme for solutions of nonlocal fully nonlinear PDEs \eqref{Nonlocal fully nonlinear equation}.

To see this, we further impose regularity of the initial condition and thus we can obtain regularity of solutions of \eqref{Nonlocal fully nonlinear equation} as in the following lemma.
\begin{lemma}\label{Thrice differentiability}
	For any $k=1,\ldots,d$, suppose that $F_{y_k}$, $F_\mathcal{X}$ in Table \ref{tab:table1}, and $F_\mathcal{XY}$ in Table \ref{tab:table2} satisfy conditions \eqref{Uniform ellipticity condition of F 1}-\eqref{Lipschitz continuity of F}. If $g\in\Omega^{{(3+\alpha)}}_{[0,T]}$, then the unique solution $u\in\Omega^{{(2+\alpha)}}_{[0,\delta]}$ of \eqref{Nonlocal fully nonlinear equation in Section 3} satisfies that $u_{y_k}\in\Omega^{{(2+\alpha)}}_{[0,\delta]}$. 
\end{lemma}
\begin{proof}
	As Theorem \ref{Well-posedness of fully nonlinear equations} presents, the nonlocal fully nonlinear equation \eqref{Nonlocal fully nonlinear equation in Section 3} admits a unique solution $u$ in $\Omega^{{(2+\alpha)}}_{[0,\delta]}$. Next, for any $k=1,2,\cdots,d$, we can find $u_{y_k}$ satisfying: 
	\begin{equation} \label{u_(y_k) equation}   
		\left\{
		\begin{array}{rcl}
			\left(\frac{\partial u}{\partial y_k}\right)_s(t,s,y)&=&\sum^d\limits_{i,j=1} F_{q_{ij}}(\cdot)\left(\frac{\partial u}{\partial y_k}\right)_{y_iy_j}(t,s,y)+\sum^d\limits_{i=1} F_{p_i}(\cdot)\left(\frac{\partial u}{\partial y_k}\right)_{y_i}(t,s,y) \\
			&&+\sum^d\limits_{i,j=1}F_{n_{ij}}(\cdot)\left(\frac{\partial u}{\partial y_k}\right)_{y_iy_j}(s,s,y)+\sum^d\limits_{i=1}F_{m_i}(\cdot)\left(\frac{\partial u}{\partial y_k}\right)_{y_i}(s,s,y) \\
			&&+F_u(\cdot)\left(\frac{\partial u}{\partial y_k}\right)(t,s,y)+F_l(\cdot)\left(\frac{\partial u}{\partial y_k}\right)(s,s,y) +F_{y_k}(\cdot) \\
			
			\quad \\
			
			\left(\frac{\partial u}{\partial y_k}\right)(t,0,y)&=&g_{y_k}(t,y), \quad 0\leq s\leq t\leq T, \quad y\in\mathbb{R}^d.  
		\end{array}
		\right.
	\end{equation}    
	where all $F_{y_k}$ and $F_{\mathcal{X}}(\cdot)$ are evaluated at $$(t,s,y,u(t,s,y),u_y(t,s,y),u_{yy}(t,s,y),u(s,s,y),u_y(s,s,y),u_{yy}(s,s,y)).$$    
	
	Given $u\in\Omega^{{(2+\alpha)}}_{[0,\delta]}$, all coefficient functions $F_\mathcal{X}(\cdot)$ in \eqref{u_(y_k) equation} and the nonhomogeneous term $F_{y_k}(\cdot)$ belong to $\Omega^{(\alpha)}_{[0,\delta]}$. Moreover, since the initial condition $g_{y_k}(t,y)\in\Omega^{(2+\alpha)}_{[0,\delta]}$ by $g\in\Omega^{{(3+\alpha)}}_{[0,T]}$, applying Corollary \ref{Schauder estimates} to \eqref{u_(y_k) equation} completes the proof.      
\end{proof}

Next, we consider a backward nonlocal fully nonlinear PDE of the form
\begin{equation} \label{Backward nonlocal fully nonlinear equation}  
	\left\{
	\begin{array}{rcl}
		u_s(t,s,y) & = & F\big(t,s,y,u(t,s,y),u_y(t,s,y),u_{yy}(t,s,y), \\
		& & \qquad \qquad u(s,s,y),u_y(s,s,y),u_{yy}(s,s,y)\big), \\
		u(t,T,y) & = & g(t,y),\quad t_0\leq t\leq s\leq T,\quad y\in\mathbb{R}^d.
	\end{array}
	\right.
\end{equation} 
where $F$ has enough regularities of \eqref{Uniform ellipticity condition of F 1}-\eqref{Lipschitz continuity of F} and $t_0$ is suitable in the sense that $[t_0,T]$ is a subset of the time interval for the maximally defined solution of \eqref{Backward nonlocal fully nonlinear equation}. Then the following theorem reveals the relationship between nonlocal fully nonlinear PDEs like \eqref{Backward nonlocal fully nonlinear equation} and flow of 2FBSDEs \eqref{Flow of 2FBSDEs}. 

\begin{theorem} \label{F-K formula} 
	For any $k=1,\ldots,d$, suppose that $F_{y_k}$, $F_\mathcal{X}$ in Table \ref{tab:table1}, and $F_\mathcal{XY}$ in Table \ref{tab:table2} satisfy conditions \eqref{Uniform ellipticity condition of F 1}-\eqref{Lipschitz continuity of F},  $\sigma(s,y)\in C^{1,2}([t_0,T]\times\mathbb{R}^d)$, and $g\in\Omega^{{(3+\alpha)}}_{[t_0,T]}$. Then, \eqref{Backward nonlocal fully nonlinear equation} admits a unique solution $u(t,s,y)$ that is first-order continuously differentiable in $s$ and third-order continuously differentiable with respect to $y$ in $\nabla[t_0,T]\times\mathbb{R}^d$. Moreover, let
	\begin{eqnarray} \label{F-K formula for 2FBSVIE}
		& Y(t,s) := u(t,s,X(s)), \qquad & Z(t,s) :=  \left(\sigma^\top u_y\right)(t,s,X(s)), \\
		&\Gamma(t,s) :=\left(\sigma^\top\left(\sigma^\top u_y\right)_y\right)(t,s,X(s)), \qquad   & A(t,s) := \mathcal{D}\left(\sigma^\top u_y\right)(t,s,X(s)), \nonumber 
	\end{eqnarray} 
	where $\left(\sigma^\top u\right)(t,s,y)=\sigma^\top(s,y)u(t,s,y)$ and the operator $\mathcal{D}$ is defined by 
	\begin{equation*}
		\mathcal{D}\varphi=\varphi_s+\frac{1}{2}\sum^d_{i,j=1}\left(\sigma\sigma^\top\right)_{ij}\frac{\partial^2\varphi}{\partial y_i\partial y_j}+\sum^d_{i=1}b_i\frac{\partial \varphi}{\partial y_i},
	\end{equation*}
	then the family of random fields $\left(X(\cdot),Y(\cdot,\cdot),Z(\cdot,\cdot),\Gamma(\cdot,\cdot),A(\cdot,\cdot)\right)$ is an adapted solution of the following flow of 2FBSDEs: 
	\begin{align} \label{Flow of 2FBSDEs}
		X(s) & = y+\int^s_{t_0}b(\tau,X(\tau))d\tau+\int^s_{t_0}\sigma(\tau,X(\tau))dW(\tau), \\
		Y(t,s) & = g(t,X(T))-\int^T_s\mathbb{F}(t,\tau,X(\tau),Y(t,\tau),Y(\tau,\tau),Z(t,\tau),Z(\tau,\tau),\Gamma(t,\tau),\Gamma(\tau,\tau))d\tau \nonumber \\
		& \qquad \qquad \quad -\int^T_sZ^\top(t,\tau)dW(\tau), \nonumber \\
		Z(t,s) & = Z(t,t_0)+\int^s_{t_0}A(t,\tau)d\tau+\int^s_{t_0}\Gamma(t,\tau)dW(\tau), \quad t_0\leq t\leq s\leq T,\quad y\in\mathbb{R}^d \nonumber   
	\end{align} 
	where $\mathbb{F}$ is defined by 
	\begin{equation}
		\begin{split}
			& \mathbb{F}(t,\tau,X(\tau),Y(t,\tau),Y(\tau,\tau),Z(t,\tau),Z(\tau,\tau),\Gamma(t,\tau),\Gamma(\tau,\tau)) \\
			=& \overline{F}(t,\tau,X(\tau),u(t,\tau,X(\tau)),u_y(t,\tau,X(\tau)),u_{yy}(t,\tau,X(\tau)), \\
			& \qquad\qquad\qquad  u(\tau,\tau,X(\tau)),u_y(\tau,\tau,X(\tau)),u_{yy}(\tau,\tau,X(\tau))), 
		\end{split}
	\end{equation}
	with the definition of $\overline{F}$
	\begin{equation*}
		\begin{split}
			& \overline{F}(t,\tau,y,u(t,\tau,y),u_y(t,\tau,y),u_{yy}(t,\tau,y),u(\tau,\tau,y),u_y(\tau,\tau,y),u_{yy}(\tau,\tau,y)) \\
			=& F(t,\tau,y,u(t,\tau,y),u_y(t,\tau,y),u_{yy}(t,\tau,y),u(\tau,\tau,y),u_y(\tau,\tau,y),u_{yy}(\tau,\tau,y)) \\
			& \qquad +\frac{1}{2}\sum^d_{i,j=1}\left(\sigma\sigma^\top\right)_{ij}(\tau,y)\frac{\partial^2 u}{\partial y_i\partial y_j}(t,\tau,y)+\sum^d_{i=1}b_i(\tau,y)\frac{\partial u}{\partial y_i}(t,\tau,y). 
		\end{split}
	\end{equation*}
\end{theorem}
\begin{proof}
	First, under the regularity assumptions of $F$ and $g$, Corollary \ref{Schauder estimates} and Lemma \ref{Thrice differentiability} guarantee that there exists a unique solution $u(t,s,y)$ of \eqref{Backward nonlocal fully nonlinear equation}, which is first-order continuously differentiable in $s$ and third-order continuously differentiable with respect to $y$. Consequently, the family of random fields $\left(X,Y,Z,\Gamma,A\right)$ defined by \eqref{F-K formula for 2FBSVIE} is well-defined (adapted).
	
	Next, we show that the random field solves the flow of 2FBSDEs, i.e. \eqref{Flow of 2FBSDEs}. For any fixed $(t,s)\in \nabla[t_0,T]$, we apply the It\^{o}'s formula to the map $\tau\to u(t,\tau,X(\tau))$ on $[s,T]$. Then we have 
	\begin{equation*}
		\begin{split}
			& du(t,\tau,X(\tau)) \\
			=& \Big[u_s(t,\tau,X(\tau))+\sum^d_{i=1}b_i(\tau,X(\tau))\frac{\partial u}{\partial y_i}(t,\tau,X(\tau)) \\
			& ~~+\frac{1}{2}\sum^d_{i,j=1}\left(\sigma\sigma^\top\right)_{ij}(\tau,X(\tau))\frac{\partial^2 u}{\partial y_i\partial y_j}(t,\tau,X(\tau))\Big]d\tau+u^\top_y(t,\tau,X(\tau))\sigma(\tau,X(\tau))dW(\tau) \\
			=& \Big[F(t,\tau,X(\tau),u(t,\tau,X(\tau)),u_y(t,\tau,X(\tau)),u_{yy}(t,\tau,X(\tau)), \\
			& \qquad\qquad\qquad u(\tau,\tau,X(\tau)),u_y(\tau,\tau,X(\tau)),u_{yy}(\tau,\tau,X(\tau)) \\
			& ~~+\frac{1}{2}\sum^d_{i,j=1}\left(\sigma\sigma^\top\right)_{ij}(\tau,X(\tau))\frac{\partial^2 u}{\partial y_i\partial y_j}(t,\tau,X(\tau))+\sum^d_{i=1}b_i(\tau,X(\tau))\frac{\partial u}{\partial y_i}(t,\tau,X(\tau))\Big]d\tau \\
			& ~~+u^\top_y(t,\tau,X(\tau))\sigma(\tau,X(\tau))dW(\tau) \\
			=& \mathbb{F}(t,\tau,X(\tau),Y(t,\tau),Y(\tau,\tau),Z(t,\tau),Z(\tau,\tau),\Gamma(t,\tau),\Gamma(\tau,\tau))d\tau \\
			& \qquad +Z^\top(t,\tau)dW(\tau)
		\end{split}
	\end{equation*}
	which indicates $dY(t,\tau)=\mathbb{F}ds+Z^\top(t,\tau)dW(\tau)$. Similarly, for any fixed $(t,s)\in \nabla[t_0,T]$, by applying the It\^{o}'s formula to $\tau\to \left(\sigma^\top u_y\right)(t,\tau,X(\tau))$ on $[t_0,s]$, we can also verify that $dZ(t,\tau)=A(t,\tau)d\tau+\Gamma(t,\tau)dW(\tau)$. Hence, \eqref{F-K formula for 2FBSVIE} is an adapted solution of \eqref{Flow of 2FBSDEs}.     
\end{proof}

We make three important observations about the stochastic system \eqref{Flow of 2FBSDEs}: (I) When the generator $\mathbb{F}$ is independent of diagonal terms, i.e. $Y(\tau,\tau)$, $Z(\tau,\tau)$, and $\Gamma(\tau,\tau)$, the flow of FBSDEs \eqref{Flow of 2FBSDEs} are reduced to a family of 2FBSDEs parameterized by $t$, which is exactly the 2FBSDE in \cite{Kong2015} and equivalent to the ones in \cite{Cheridito2007} for any fixed $t$; (II) \eqref{Flow of 2FBSDEs} is more general than the system of \eqref{Forward SDE}-\eqref{Backward BSVIE} since it allows for a nonlinearity of $(Y(t,\tau),Z(t,\tau),\Gamma(t,\tau))$ by introducing an additional SDE and also contains their diagonal terms $(Y(\tau,\tau),Z(\tau,\tau),\Gamma(\tau,\tau))$ in almost arbitrary way; (III) Theorem \ref{F-K formula} shows how to solve the flow of 2FBSDE \eqref{Flow of 2FBSDEs} from the perspective of nonlocal PDEs. Inspired by \cite{Cheridito2007} and \cite{Soner2011}, the opposite implication of solutions (from 2FBSDEs to PDE) is likely valid by establishing the well-posedness of \eqref{Flow of 2FBSDEs} in the theoretical framework of SDEs. However, it is beyond the scope of this paper, while we will prove the existence and uniqueness of \eqref{Flow of 2FBSDEs} in our future publications.


\section{Conclusions} \label{sec:conclusion}
This paper studies nonlocal parabolic PDEs originated from time-inconsistent problems in the theory of stochastic differential equations and control problems. Tailor-made norms and Banach spaces are proposed for the study of nonlocality and time-inconsistency. Based on which, the first key step is to establish the well-posedness of nonlocal linear PDEs. Subsequently, a Schauder-type prior estimate is provided to control the behaviour of solutions of nonlocal linear PDEs. The study of the linear case laid a solid foundation for studying nonlocal fully nonlinear PDEs by the linearization method. By constructing a contraction in the space of possible solutions, the Banach's fixed point theorem promises the existence and uniqueness of nonlocal fully nonlinear equations.

Thanks to the close connections among PDE, FBSDE, and stochastic controls, our well-posedness results of nonlocal PDEs contribute significantly to the development of the related fields: (1) From the perspective of TIC stochastic control problems, the main breakthrough is that the nonlocal fully nonlinear PDEs in this paper allow control enter the diffusion of state process, which extends essentially the restricted results in the existing literature. The solvability of nonlocal fully nonlinear PDEs indicates directly the well-posedness of the related HJB-type equations in TIC problems, including equilibrium HJB equations and HJB systems, both of which are used to identify the subgame perfect equilibrium; (2) In terms of the theory of FBSDEs, our study of nonlocal PDEs provides a general and unified treatment for the Feynman--Kac formulas of a flow of 2FBSDEs \eqref{Flow of 2FBSDEs}. The advance in nonlocal PDEs of this paper provides a new insight into the study of the flow of FBSDEs or 2FBSDEs from the perspective of PDEs. Besides of pushing the frontiers of PDE, SDE, and stochastic control theory, our work sheds light on some finance and decision-making theories that involve reference points (initial-time dependence), such as prospect and regret theories. 




\bibliographystyle{elsarticle-num} 
\bibliography{NonlocalityRef}

\end{document}